\documentclass[reqno]{amsart}

\usepackage{mathtools}
\usepackage{amsmath}
\usepackage{amssymb}
\usepackage{amsthm}
\usepackage[all]{xypic}
\usepackage{amsfonts}
\usepackage{verbatim,color}
\usepackage{url}
\usepackage{enumitem}
\usepackage{hyperref}
\numberwithin{equation}{section}
\theoremstyle{plain}
\newtheorem{theorem}[equation]{Theorem}
\newtheorem{proposition}[equation]{Proposition}
\newtheorem{corollary}[equation]{Corollary}
\newtheorem{lemma}[equation]{Lemma}

\theoremstyle{definition}
\newtheorem{definition}[equation]{Definition}

\newtheorem{example}[equation]{Example}

\newtheorem{remark}[equation]{Remark}

\newcommand{\hdet}{\operatorname{hdet}}

\newcommand{\beq}{\begin{equation}}
\newcommand{\eeq}{\end{equation}}

\DeclareMathOperator{\Hom}{{Hom}}

\DeclareMathOperator{\End}{{End}}
\DeclareMathOperator{\Ext}{{Ext}}

\DeclareMathOperator{\Tor}{Tor}

\DeclareMathOperator{\Ker}{Ker}
\DeclareMathOperator{\coker}{Coker}
\DeclareMathOperator{\im}{Im}

\DeclareMathOperator{\GK}{GKdim}

\DeclareMathOperator{\soc}{soc}

\DeclareMathOperator{\gldim}{gl.dim}

\newcommand{\mc}{\mathcal}

\newcommand{\Z}{\mathbb{Z}}
\newcommand{\ZZ}{{\mathbb Z}}

\newcommand{\QQ}{{\mathbb Q}}
\newcommand{\mb}{\mathbb}

\newcommand{\ol}{\overline}

\newenvironment{psmatrix}
  {\left(\begin{smallmatrix}}
  {\end{smallmatrix}\right)}

\DeclareMathOperator{\lMod}{\!-Mod}

\DeclareMathOperator{\GKdim}{GKdim}

\newcommand{\diag}{\operatorname{diag}}
\newcommand{\lra}{\longrightarrow}
\newcommand{\del}{\partial}

\newcommand{\on}{\operatorname}

\newcommand{\wh}{\widehat}

\newcommand{\lsoc}{\soc_l}
\newcommand{\rsoc}{\soc_r}

 \newcommand{\C}{\mathbb{C}}
 \newcommand{\M}{\mathbb{M}}
 \newcommand{\N}{\mathbb{N}}

 \newcommand{\op}{^{\mathrm{op}}}
 \newcommand{\tot}{^{\mathrm{tot}}}

\newcommand{\eps}{\epsilon}

\begin{document}
\title{Growth of graded twisted Calabi-Yau algebras}

\author{Manuel L. Reyes}
\address{Department of Mathematics\\ Bowdoin College\\
8600 College Station\\ Brunswick, ME 04011-8486\\ USA}
\email{reyes@bowdoin.edu}

\author{Daniel Rogalski}
\address{UCSD\\ Department of Mathematics\\ 9500 Gilman Dr. \# 0112 \\
La Jolla, CA 92093-0112\\ USA}
\email{drogalsk@math.ucsd.edu}

\date{\today}

\thanks{Reyes was partially supported by the NSF grant DMS-1407152.
Rogalski was partially supported by the NSF grant
DMS-1201572 and the NSA grant H98230-15-1-0317.}

 %(1991-now) Graded rings and modules


\begin{abstract}
We initiate a study of the growth and matrix-valued Hilbert series of $\N$-graded twisted Calabi-Yau algebras that are homomorphic images of path algebras of weighted quivers, generalizing techniques previously used to investigate Artin-Schelter regular algebras and graded Calabi-Yau algebras. Several results are proved without imposing any assumptions on the degrees of generators or relations of the algebras.
We give particular attention to twisted Calabi-Yau algebras of dimension $d \leq 3$, giving precise
descriptions of their matrix-valued Hilbert series and partial results describing which underlying quivers
yield algebras of finite GK-dimension.
For $d = 2$, we show that these are algebras with mesh relations.
For $d = 3$, we show that the resulting algebras are a kind of derivation-quotient algebra arising from
an element that is similar to a twisted superpotential.
\end{abstract}

 \maketitle

\setcounter{tocdepth}{1} %Include sections but not subsections in table of contents
 \tableofcontents

\section{Introduction}
\label{sec:intro}

Twisted Calabi-Yau algebras, whose definition will be recalled in Section~\ref{sec:CY} below,
form a common generalization of two classes of algebras that are each important in their own right.
The first is the class of Calabi-Yau algebras as defined by Ginzburg~\cite{Gi}, which have developed into
a prominent focus of research in the past several years; the second is the class of Artin-Schelter (AS)
regular algebras~\cite{AS}, which are graded algebras that have long been viewed as an important
type of ``noncommutative polynomial algebra.''
The aim of this paper is to study the growth of $\mb{N}$-graded twisted Calabi-Yau algebras.
Below, we let $k$ denote an arbitrary field; a \emph{graded algebra} always means an $\mb{N}$-graded $k$-algebra.
A graded algebra $A$ is \emph{locally finite} if $\dim_k A_n < \infty$ for all $n \geq 0$ and is
\emph{connected} if $A_0 = k$.

While Calabi-Yau and twisted Calabi-Yau algebras need not be graded, the more tractable subclass
of graded (twisted) Calabi-Yau algebras have been the subject of serious interest for quite some time~\cite{Bo, BSW}.
Let us discuss an important special case that has motivated much of the theory.
Let $Q$ be a finite quiver (possibly with loops and multiple edges).  Our convention is to compose arrows in the path algebra $kQ$ from left to right.  Let $\mu$ be an automorphism of $kQ$.  A \emph{$\mu$-twisted superpotential} $\omega$ on $Q$ is a linear combination of paths of length $d$ which is invariant under the linear map sending each path $\alpha_1 \alpha_2 \cdots \alpha_d$ to $(-1)^{d+1} \mu(\alpha_d) \alpha_1 \alpha_2 \cdots \alpha_{d-1}$.  For each arrow $\alpha$ there is a linear partial derivative map $\del_{\alpha}$ sending the path $\alpha_1 \alpha_2 \cdots \alpha_d$ to $\alpha_2 \cdots \alpha_d$ if $\alpha_1 = \alpha$ and to $0$ otherwise.   Given a $\mu$-twisted superpotential $\omega$ and some $k \geq 0$ the corresponding \emph{derivation-quotient algebra} is $A = \mc{D}(\omega, k) = kQ/I$, where $I$ is generated by all order $k$ mixed partial derivatives of $\omega$ (when $k = 0$, $\omega$ itself generates $I$).  In nice cases, such an algebra $A$ is twisted Calabi-Yau of global dimension $k+2$.  Moreover, it has been proved that under certain restrictions, such as the $m$-Koszul property, twisted Calabi-Yau algebras must arise from twisted superpotentials in this way~\cite{BSW}.  The reader can find detailed examples in \cite{Bo} and \cite{BSW}.
In the case of connected graded algebras (so that the quiver $Q$ has a single vertex), recent work~\cite{MoSm, MoU} has even focused
on classifying those superpotentials whose derivation-quotient algebras are Calabi-Yau of dimension~3.

Our point of view in this paper, and a companion paper \cite{RR1} which has many related results,
is more general.  Rather than considering the special properties of algebras arising from potentials on quivers,
we aim to study twisted Calabi-Yau locally finite graded $k$-algebras with no further restrictions on the
base field or the degrees of generators and relations.  Twisted Calabi-Yau derivation-quotient algebras are of course a special
case, being naturally graded by path length (since we only considered homogeneous potentials).

The paper \cite{RR1} establishes many basic results about locally finite graded twisted Calabi-Yau algebras, including the stability of the twisted Calabi-Yau property under many common constructions such as tensor product, base field extension, and so on.
In particular, we show that the twisted Calabi-Yau property is equivalent to various other formulations we call
 \emph{generalized Artin-Schelter regular}; see \cite[Theorem 5.2, Theorem 5.15]{RR1} and Theorem~\ref{thm:cy-vers-asreg} below.
This result generalizes a known theorem that for locally finite connected graded algebras, twisted Calabi-Yau and Artin-Schelter regular
are equivalent concepts \cite[Theorem 1.2]{RRZ1}.   We also offer in \cite{RR1} a classification of twisted Calabi-Yau algebras of the smallest
dimensions $0$ and $1$, and study the noetherian property for twisted Calabi-Yau algebras of dimension $2$.

This paper will focus on results about the growth of graded locally finite twisted Calabi-Yau algebras.  We assume here that the reader is familiar with the basic ideas of growth of algebras and Gelfand-Kirillov dimension (GK-dimension) as in \cite{KL}; we will briefly review the definitions in Section~\ref{sec:growth}.  In fact, the original definition of AS regularity for a connected graded algebra $A$ included the additional condition that $\GKdim(A) < \infty$.  Under this condition, AS regular algebras are conjectured to be noetherian domains, while AS regular algebras with $\GKdim(A) = \infty$ have exponential growth, and are not noetherian \cite[Theorem 1.2]{SZ}.     We expect a similar dichotomy to hold for locally finite graded twisted Calabi-Yau algebras:  the ones with finite GK-dimension should be much more well-behaved.  Surprisingly, however, little attention seems to have been given to the growth of twisted Calabi-Yau algebras, so we hope to begin to remedy that with this paper.  We start with a general study of the growth of locally finite graded twisted Calabi-Yau algebras (in fact many of our results hold under weaker hypotheses).  Then in the second half of the paper, we prove some basic structural results for locally finite graded twisted Calabi-Yau algebras of dimension $2$ and $3$.  In particular we examine how one can understand which of these have finite GK-dimension.

We now give an overview of the results of the paper.   Let $A$ be a locally finite graded algebra.  The main tool for investigating the growth of a finitely graded twisted Calabi-Yau algebra $A$ is its matrix-valued Hilbert series.  We fix a decomposition $1 = e_1 + e_2 + \cdots + e_n$ of $1$ as a sum of primitive orthogonal idempotents $e_i \in A_0$.  The matrix-valued Hilbert series is the formal series $h_A(t) = \sum_{n \geq 0} [\dim_k (e_i A_n e_j)] t^n$ in $M_n(\mb{Z})[[t]]$.
In Section~\ref{sec:growth}, we generalize some standard results about the Hilbert series of connected graded algebras to the matrix Hilbert series of locally finite graded algebras.    In particular, when the matrix Hilbert series is \emph{rational} of the form
$q(t)^{-1}p(t)$ for matrix polynomials $p(t), q(t) \in M_n(\mb{Z}[t])$, then we show that the growth of $A$ is controlled by the roots of $\det q(t)$ (Proposition~\ref{prop:growth}).

Let $J(A)$ be the graded Jacobson radical of a locally finite graded algebra $A$ and let $S = A/J(A)$, which is a finite-dimensional semisimple algebra.  We focus for most of the paper on \emph{elementary}
algebras, for which $S \cong k^n$  for some $n$.  This is justified in Theorem~\ref{thm:reduce-elem} and Lemma~\ref{lem:cy-reduce-elem}, where we show via standard techniques that one can reduce to the elementary case by extending the base field and passing to a Morita equivalent algebra, both of which preserve the twisted Calabi-Yau property and the growth of an algebra. In the elementary case, $A$ has a presentation $A = kQ/I$ for some uniquely determined \emph{weighted} quiver $Q$, where the arrows are assigned nonnegative degrees (see Lemma~\ref{lem:quiver facts}).

Weighted arrows complicate the theory significantly.  Any graded twisted Calabi-Yau algebra has a \emph{Nakayama bimodule} $U_A$
which is graded invertible $(A, A)$-bimodule (See Definition~\ref{def:graded twisted CY}).  Thus it is important to understand what graded invertible bimodules over an elementary locally finite graded algebra can look like.  We work out the surprisingly complicated answer in
Section~\ref{sec:bimodule}.  Namely, any such graded invertible bimodule $U$ has the form $\bigoplus_{i=1}^n Ae_i(\ell_i)$ as a left module, so $U \cong A$ as \emph{ungraded} left modules, but the different indecomposable projectives making up $A$ may have different graded shifts; moreover $U \cong {}^1 A ^{\mu}$ as a graded bimodule, where the right action is given by some automorphism $\mu: A \to A$ which may not preserve degrees in $A$! When $U = U_A$ is the Nakayama bimodule of an elementary locally finite graded twisted Calabi-Yau algebra $A$, we call $\mu$ the \emph{Nakayama automorphism} of $A$ and the vector $(\ell_1, \dots, \ell_n)$ the \emph{Artin-Schelter index}.  It was a surprise to find that a graded twisted Calabi-Yau algebra might have a Nakayama automorphism that is not a graded automorphism.  When $A$ is indecomposable and the underlying quiver $Q$ has arrows of weight $1$, then
such strange behavior does not happen:  all the $\ell_i$ are equal and $\mu$ is graded.
Allowing for generators in arbitrary degrees is important, for example, if one is interested in invariant theory of graded twisted Calabi-Yau algebras, since invariant rings tend to have generators in high degree.  Some studies of noncommutative invariant theory in this context have recently appeared in~\cite{Weis1}, \cite{Weis2}.

We now state our main theorems.  First, we have a general result on the structure of the matrix Hilbert series of an elementary graded twisted Calabi-Yau algebra.  Given a vector of integers $(\ell_1, \dots, \ell_n)$, we will form the diagonal matrix $L = \diag(\ell_1, \dots, \ell_n)$ and use the matrix exponential to denote the following matrix Laurent polynomial:
\[
t^L = \exp(\log(t) L) =
\begin{pmatrix}
t^{\ell_1} & \cdots & 0 \\
\vdots & \ddots & \vdots\\
0 &  \cdots & t^\ell_n
\end{pmatrix}.
\]

\begin{theorem}
\label{thm:main1}
(Propositions~\ref{prop:growth}, \ref{prop:perfect-hs}, \ref{prop:cygk}, and~\ref{prop:exact})
Let $A$ be a locally finite elementary graded twisted Calabi-Yau algebra with Nakayama automorphism $\mu$ and AS-index $(\ell_1, \dots, \ell_n)$.  Let $P$ be the permutation matrix associated to the action of $\mu$ on the idempotents $e_i$, and
let $L = \diag(\ell_1, \dots, \ell_n)$ and $t^L$ be as above.
  \begin{enumerate}
\item The matrix Hilbert series of $A$ is of the form $h_A(t) = q(t)^{-1}$ for some matrix polynomial $q(t)$.
$\GKdim(A) < \infty$ if and only if all roots of $\det q(t)$ are roots of unity, and otherwise $A$ has exponential growth.
When $\GKdim(A)$ is finite it is bounded above by the multiplicity of $1$ as a root of $\det q(t)$.
\item $q(t)$ satisfies the functional equation $q(t) = (-1)^d P t^L q(t^{-1})^T$, and $q(t)$ commutes with $P t^L$.
\item When $\GKdim(A) < \infty$ and $A$ is noetherian, then GK-dimension for $A$-modules is graded exact and graded finitely partitive.
\end{enumerate}
\end{theorem}

In Section~\ref{sec:dim2} we specialize to the case of twisted Calabi-Yau algebras of dimension $2$,
proving the following results.   For any elementary locally finite graded algebra with underlying weighted quiver $Q$, we let
$N(t) = \sum_{m \geq 0} H_m t^m$ be the weighted incidence matrix of $Q$, where $(H_m)_{ij}$ is the number of
arrows of weight $m$ from $i$ to $j$.

\begin{theorem}
\label{thm:main2}  (Proposition~\ref{prop:gldim2}, Theorem~\ref{thm:GK2})
Let $A$ be a locally finite elementary graded twisted Calabi-Yau algebra of dimension $2$, and keep the notation of Theorem~\ref{thm:main1}.
\begin{enumerate}
\item The matrix Hilbert series of $A$ is $h_A(t) = (I - N(t) + P t^L)^{-1}$.
\item $A \cong kQ/(\omega)$ for an element $\omega = \sum_i \tau(x_i) x_i$, where $\{ x_i \}$ runs over the arrows of $Q$,
and where $\{ \tau(x_i) \}$ is a basis for an arrow space in $kQ$ (as in Definition~\ref{def:arrow space}).  Here, if $x_i$ is an arrow of degree $d$ from $r$ to $j$ then $\tau(x_i)$ is an arrow of degree $\ell_j - d$ from $\mu^{-1}(j)$ to $r$.
\item If all arrows of $Q$ have weight $1$, so that $N = Mt$ where $M = N(1)$ is the usual incidence matrix of $Q$,
then $M$ has spectral radius $\rho(M) \geq 2$.  Moreover, $\GKdim(A) < \infty$ if and only $\rho(M) = 2$, and in this case
$\GKdim(A) = 2$.
\end{enumerate}
\end{theorem}
\noindent In \cite{RR1}, we prove in addition that any locally finite (not necessarily elementary) graded twisted Calabi-Yau algebra $A$ of dimension $2$ is noetherian if and only $\GKdim(A) < \infty$ \cite[Theorem 6.6]{RR1}.  In this paper we also prove a subsidiary
result about the GK-dimension for modules over twisted Calabi-Yau algebras of dimension $2$ (Proposition~\ref{prop:forRR1}) which is needed to complete the proof in \cite{RR1}.

Finally, in Section~\ref{sec:dim3} we consider twisted Calabi-Yau algebras of dimension~$3$.

\begin{theorem}
\label{thm:main4} (Proposition~\ref{prop:gldim3}, Proposition~\ref{prop:diagonalize})
Let $A$ be a locally finite elementary graded twisted Calabi-Yau algebra of dimension $3$, and keep the notation of the earlier theorems.
\begin{enumerate}
\item The matrix Hilbert series of $A$ is $h_A(t) = (I - N(t) + P t^L N(t^{-1})^T  - P t^L)^{-1}$.
\item There is an element $\omega = \sum_{i,j} y_i g_{ij} x_j$, where $\{ x_j \}$ runs over the arrows of $Q$
and $\{y_i\}$ is a basis for an arrow space in $kQ$, such that if $x_i$ is an arrow of degree $d$ from $r$ to $j$ then $y_i$ is an arrow of degree $\ell_j - d$ from $\mu^{-1}(j)$ to $r$, and if $h_i = \sum_j g_{ij} x_j$ and $h'_j = \sum_i y_i g_{ij}$ then
both $\{ h_i \}$ and $\{ h'_j \}$ are minimal sets of generators for the ideal $I$, where $A = kQ/I$.
\item Suppose that $A$ is indecomposable and that all arrows of $Q$ have degree $1$, so that $N = Mt$ for the usual incidence matrix $M$ of $Q$.  Suppose that $3 \leq \GKdim(A)$.  If in addition $M$ is
a normal matrix, then we can give an explicit condition on the eigenvalues of $M$ (involving some special regions of the complex plane) that characterizes $\GKdim(A) < \infty$.  In particular, if $\GKdim(A) < \infty$ then $\GKdim(A) = 3$, the $\ell_i$ are equal to a single AS-index $\ell$ where $\ell =3$ or $\ell = 4$, and the spectral radius $\rho = \rho(M)$ satisfies $\rho = 6- \ell$.
\end{enumerate}
\end{theorem}
\noindent When the incidence matrix $M$ of $Q$ is not normal, we are not sure how to characterize when $A$ has finite
GK-dimension in Theorem~\ref{thm:main4}.  In related work in progress~\cite{GR}, the second-named author and Jason Gaddis classify all (unweighted) quivers $Q$ with at most three vertices for which there exists a twisted Calabi-Yau algebra $A = kQ/I$ of dimension $3$ with $\GKdim(A) < \infty$.
In particular, we note that there is an infinite family of examples of Calabi-Yau algebras of dimension $3$ with incidence matrices $\begin{psmallmatrix} 0 & a & 0 \\\ 0 & 0 & b \\ c & 0 & 0\end{psmallmatrix}$ where $a, b, c$ are any positive integers satisfying Markov's equation $a^2 + b^2 + c^2 = abc$.  These examples can be seen to arise through the process of quiver mutation.  There is certainly no upper bound to the spectral radius of such examples, unlike in the case of normal matrices.

The elements $\omega$ that determine the relations in both the dimension 2 and dimension 3 case do not appear to satisfy the usual definition of twisted superpotential in general, but they are clearly analogous in some way.  We think of them as  a kind of generalized twisted superpotential, and the theorems above show that the algebras must arise as generalized derivation-quotient algebras.
It would be interesting to try to formulate a more precise theory of twisted superpotentials in this context of weighted quivers.  Another project, which is work in progress, is to give a classification of all elementary locally finite twisted Calabi-Yau algebras of dimension $2$ (whether of finite GK-dimension or not).  The conjectural answer is that given any unweighted quiver with incidence matrix having spectral radius at least $2$, and an element $\omega$ as in Theorem~\ref{thm:main2}, then $kQ/(\omega)$
should be twisted Calabi-Yau.

\subsection*{Acknowledgments.}
We thank Jason Bell, Jason Gaddis, Colin Ingalls, and Stephan Weispfenning for helpful conversations.

\section{Matrix valued Hilbert series and growth}
\label{sec:growth}

In this section, we develop some results for matrix-valued Hilbert series which are analogues of
some well-known results for usual Hilbert series.  The arguments in the matrix-valued case are
generally similar.

First we recall some of the definitions of growth and Gelfand-Kirillov (GK) dimension.  Let $A$ be a finitely generated (unital) $k$-algebra, and let
$M$ be a finitely generated left $A$-module.  Let $V$ be a finite-dimensional $k$-subspace of $A$ containing 1 that generates $A$ as an algebra, and let $W$ be a finite-dimensional $k$-subspace of $M$ that generates $M$ as an $A$-module.   The \emph{growth}
of $M$ is the function $\mb{N} \to \mb{N}$ given by $n \mapsto \dim_k V^n W$; it is independent of the choice of $V$ and $W$ up to a natural equivalence relation on functions, so that all of the the following definitions are also independent of this choice.
We say that $M$ has \emph{exponential growth} if $\limsup_{n \to \infty} (\dim_k V^n W)^{1/n} > 1$.  If $M$ does not have exponential growth, then we
say that $M$ has \emph{subexponential growth}, and in this case (as long as $M \neq 0$) it follows that
$\limsup_{n \to \infty} (\dim_k V^n W)^{1/n} = 1$.
We say that $M$ has \emph{polynomial growth}
or \emph{finite GK-dimension} if $\limsup_{n \to \infty}  \log_n (\dim_k V^n W) = d < \infty$, with
the value of $d$ being called the GK-dimension of $M$.   The growth of the algebra $A$ and its GK-dimension are defined by applying
the definitions above to $A$ as a module over itself (one may take $W = k 1_A$).  See \cite{KL} for more details on these definitions.

Let $R[t]$ be the ring of polynomials over a ring $R$, let $R[t, t^{-1}]$ be the ring of Laurent polynomials over a ring $R$, let $R[[t]]$ be the ring of formal power series, and let $R((t))$ be the ring of formal Laurent series.   A \emph{matrix polynomial} is an element of the ring $M_n(\mb{Q})[t] \cong M_n(\mb{Q}[t])$.
Similarly, a \emph{matrix Laurent polynomial} is an element of the ring $M_n(\mb{Q})[t, t^{-1}] \cong M_n(\mb{Q}[t, t^{-1}])$.
We extend these definitions to power series: a \emph{matrix power series} is an element of $M_n(\mb{Q})[[t]] \cong M_n(\mb{Q}[[t]])$
and a \emph{matrix Laurent series} is an element of $M_n(\mb{Q})((t)) \cong M_n(\mb{Q}((t)))$.
There is an obvious notion of convergence of a matrix Laurent series:  given $\lambda \in \mb{C}$
and $p(t) = \sum_{n \geq n_0} H_n t^n \in M_m(\mb{Q}((t)))$, then we say that \emph{$p(t)$ converges at $\lambda$} if the sequence $\sum_{i = n_0}^{n} H_i \lambda^i$ converges as a sequence in $\mb{R}^{m^2}$ with its Euclidean topology; we let $p(\lambda)$ denote the matrix to which it converges.

\begin{definition}
Let $p(t)$ be a matrix polynomial.  A \emph{root} of $p(t)$ is a scalar $\lambda \in \mb{C}$ such that $p(\lambda)$ is a singular matrix.  Thus, the roots of $p(t)$ are the same as the roots of the polynomial $\det(p(t)) \in \mb{Q}[t]$.
\end{definition}

Recall that a $k$-algebra $A$ is \emph{$\mb{N}$-graded} if $A = \bigoplus_{n \geq 0} A_n$ as vector spaces, with
$A_i A_j \subseteq A_{i+j}$ for all $i,j$; we say that $A$ is \emph{locally finite} if $\dim_k A_n < \infty$ for all $n$.
In this paper when we refer to a graded algebra we always mean $\mb{N}$-graded.
A left $A$-module $M$ is $\mb{Z}$-graded if $M = \bigoplus_{n \in \mb{Z}} M_n$ with $A_i M_j \subseteq M_{i+j}$
for all $i,j$. For any $\ell \in \Z$, the shifted module $M(\ell)$ is the same module with the degrees shifted
so that $M(\ell)_n = M_{\ell + n}$.
We say that $M$ is \emph{left bounded} if $M_n = 0$ for $n  \ll 0$ and $M$ is \emph{locally finite} if $\dim_k M_n < \infty$ for all $n$.
Note that if $A$ is locally finite, then any finitely generated $\Z$-graded left $A$-module $M$ is left
bounded and locally finite.

Let $A$ be a graded $k$-algebra with $\dim_k A_0 < \infty$. Then one may write $1$ as a sum $1 = e_1 + \cdots + e_r$ where the $e_i$ are pairwise orthogonal primitive idempotents in $A_0$. It is easy to see that the $e_i$ are also primitive as idempotents of $A$; see the proof of \cite[Lemma 2.7]{RR1}.  In general the $e_i$ are not uniquely determined. We must fix a choice of $e_i$ fixed for the following definitions, but
we will describe the extent to which they are independent of the particular choice in Remark~\ref{rem:decomposition} below.

We will frequently make use of module decompositions via idempotents. If $S$ is an algebra containing an
idempotent $e$ and $M$ is a left $S$-module, then we have the vector space decomposition $M = eM \oplus (1-e)M$.
Furthermore, given idempotents $e,f \in S$, if $M$ is a left $S$-module, right $S$-module, or $(S,S)$-bimodule
then we have the following respective identifications:
\begin{align*}
eM &= \{x \in M \mid x = ex\}, \\
Mf &= \{x \in M \mid x = xf\}, \\
eMf &= \{x \in M \mid x = exf\} =  eM \cap Mf.
\end{align*}
We refer readers to~\cite[Section~21]{FC} for further details about idempotent decompositions.

\begin{definition}
Let $A$ be a locally finite graded algebra with a fixed decomposition $1 = e_1 + \cdots + e_r$ where the $e_i \in A_0$ are pairwise orthogonal primitive idempotents.  Let $M$ be a left bounded, locally finite $\mb{Z}$-graded left $A$-module.  The \emph{(vector) Hilbert series} of $M$
is the column vector in $\mb{Z}((t))^r$ given by $h_M(t) = (s_1(t), \dots, s_r(t))^T$, where $s_i(t) = \sum_{n \in \mb{Z}} \dim_k (e_i M)_n t^n$.  The \emph{total Hilbert series} of $M$ is $h_M\tot(t) = \sum_{n \in \mb{Z}} (\dim_k M_n) t^n \in \mb{Z}((t))$.  Clearly the coefficient of $t^n$ in the total Hilbert series is the sum of the coefficients of $t^n$ in all of the coordinates of the vector Hilbert series.

If $M$ is a $\mb{Z}$-graded left bounded locally finite  $(A, A)$-bimodule, we can consider an even finer invariant:
The \emph{(matrix) Hilbert series} of $M$ is the matrix Laurent series $h_M(t)  \in  M_r(\mb{Z}((t)))$  whose $(i,j)$ entry $[h_M(t)]_{i,j}$ is given by $\sum_{n \in \mb{Z}} \dim_k (e_i M e_j)_n t^n$.  Again, the coefficient of $t^n$ in the total Hilbert series of $M$ is given
by summing the coefficients of $t^n$ in all of the entries of the matrix Hilbert series.  Also, the $j$th column of the matrix Hilbert series
of $M$ is the vector Hilbert series of the left module $Me_j$.
\end{definition}

Of course, one may consider the matrix Hilbert series of the algebra $A$ itself thought of as an $(A, A)$-bimodule in the usual way.
We are especially interested in the matrix Hilbert series of graded algebras $A$ in this paper and of the vector Hilbert series of graded left $A$-modules $M$.  Thus when we just refer to the Hilbert series of the algebra $A$ or of a module $M$ with no modifiers, we will mean the matrix Hilbert series (or vector Hilbert series, respectively).

\begin{remark}\label{rem:decomposition}
Because the primitive idempotent decomposition $1 = e_1 + \cdots + e_r \in A_0$ used in the
definition above is not unique, one may worry to what extent the corresponding Hilbert series are
non-unique.
First, notice that the total Hilbert series $h_M\tot(t)$ is completely independent of the $\{e_i\}$.
Now suppose that $1 = f_1 + \cdots + f_m$ is another decomposition into orthogonal primitive
idempotents in $A_0$. Because $\bigoplus A_0e_i = A_0 = \bigoplus A_0f_i$, the Krull-Schmidt theorem implies
that $m =r$ and, up to a permutation, we have $A_0 e_i \cong A_0 f_i$ as left $A_0$-modules for
all~$i$.
It follows~\cite[Exercise~21.15]{FC} that there is a unit $u \in A_0$ such that $f_i = u e_i u^{-1}$
for all $i$. Then any locally finite left $A$-module $M$ satisfies $f_i M = u e_i u^{-1} M = u e_i M$,
so that
\[
\dim_k(f_iM)_n = \dim_k(u e_i M)_n = \dim_k (e_iM)_n.
\]
Thus up to a permutation of vector entries, the two vector Hilbert series of $M$ are equal. If
$M$ is in fact a locally finite graded $(A,A)$-bimodule, a similar analysis yields
$\dim_k (e_i M e_j)_n = \dim_k (f_i M f_j)_n$ for all $i$ and $j$, so that the two matrix-valued Hilbert
series for $A$ are equal after a suitable permutation of the rows and columns.
\end{remark}

\begin{remark}
\label{rem:graded growth}
Suppose that the locally finite graded algebra $A$ is finitely generated as an algebra and that $M$ is a $\mb{Z}$-graded finitely generated left $A$-module.  Then $M$ is left bounded and locally finite, say $M = \bigoplus_{n \geq n_0} M_n$, and the growth and GK-dimension of $M$ are both determined by the dimensions $\dim_k M_n$ of the graded pieces of $M$ in the following way: $M$ has subexponential growth
if and only if $\lim \sup_{n \to \infty} (\dim_k M_n)^{1/n} \leq 1$, and the GK-dimension of $M$ is equal to
$\lim \sup_{n \to \infty} \log_n (\sum_{i=n_0}^n \dim_k M_i)$. See~\cite[Lemma~ 6.1]{KL}.
\end{remark}

Based on the above, we say that an integer Laurent series $h(t) = \sum_{n \geq n_0} a_n t^n$ with nonnegative
coefficients (or any left bounded $\mb{Z}$-graded vector space $V$ with $a_n = \dim_k V_n$) has \emph{subexponential growth}
if $\lim \sup_{n \to \infty} (a_n)^{1/n} \leq 1$, and we define its \emph{GK-dimension} to be
$\lim \sup_{n \to \infty} \log_n \left( \sum_{i=n_0}^n a_i \right)$.

We will be interested in matrix Laurent series that are analogous to rational Laurent series.
It will often be convenient for us to consider matrix-valued series that can be expressed as quotients
$q(t)^{-1} p(t)$ for some matrix polynomials $p,q \in M_n(\QQ)[t]$. This raises the question of
whether such an expression can also be written in the form $r(t) s(t)^{-1}$ for matrix polynomials
$r$ and $t$; fortunately, these two properties are equivalent.
In the result below, we will say that a matrix polynomial $q(t) \in M_n(\QQ)(t)$ is \emph{Laurent invertible}
if it is invertible as an element of the ring $M_n(\QQ)((t)) \cong M_n(\QQ((t)))$ of matrix Laurent
series.

\begin{lemma}
Let $h(t) \in M_n(\QQ((t))) \cong M_n(\QQ)((t))$. The following are equivalent:
\begin{enumerate}[label=(\alph*)]
\item $h(t) = q(t)^{-1} p(t)$ for some $p,q \in M_n(\QQ)[t]$ such
that $q(t)$ is Laurent invertible;
\item $h(t) = r(t) s(t)^{-1}$ for some $r,s \in M_n(\QQ)[t]$ such
that $s(t)$ is Laurent invertible;
\item $h(t) \in M_n(\QQ(t)) \cong M_n(\QQ)(t)$ is a matrix of rational functions.
\end{enumerate}
\end{lemma}

\begin{proof}
Assume~(a) holds.  Let $r(t)$ be the inverse of $q(t)$ in $M_n(\mb{Q}((t)))$.  Then
since $q(t) r(t) = I_n$ is the identity matrix, we have $\det (q(t)) \det (r(t)) = 1$.  In particular,
the polynomial $D(t) = \det(q(t)) \in \QQ[t]$ has inverse $\frac{1}{D} \in \QQ(t)$.
Let $c(t) \in M_n(\QQ[t])$ denote the adjugate matrix of $q(t) \in M_n(\QQ[t])$;
it follows~\cite[Proposition~XIII.4.16]{Lang} that $q(t)^{-1} = \frac{1}{D} c(t)$
is a matrix of rational functions. Thus $q(t)^{-1}p(t) \in M_n(\QQ(t))$, establishing~(c).
A symmetric argument yields (b)$\implies$(c).

Conversely, suppose that~(c) holds. Since the entries of $h(t)$ are ratios of polynomials
in $\QQ[t]$, we may fix a common denominator $0 \neq D(t) \in \QQ[t]$ for every entry
of $h(t)$. So $p(t) = D \cdot h(t) \in \M_n(\QQ[t])$ is a matrix polynomial, and
\[
h(t) = D^{-1} \cdot p(t) = (D(t) \cdot I_n)^{-1} p(t) = p(t) (D(t) \cdot I_n)^{-1}
\]
so that~(a) and~(b) both hold.
\end{proof}

\begin{definition}
We say that a matrix Laurent series $h(t) \in M_n(\mb{Q}((t)))$ is \emph{rational} if it
satisfies the equivalent conditions above.
\end{definition}

When a locally finite $\mb{N}$-graded algebra $A$ has a rational Hilbert series $p(t) q(t)^{-1}$, the
growth of the algebra is closely related to the properties of the roots of $p(t)$ and $q(t)$, as we will see below.
Before discussing the matrix case further, we review the case where $n = 1$.
Given a nonzero rational power series $r(t) = p(t)q(t)^{-1} \in \mb{Q}((t))$, with $p(t) \in \mb{Q}[t]$ and $q(t) \in \mb{Q}[t]$, note that since $\mb{Q}[t] = \mb{Q}[1-t]$ we can also expand $r(t)$ as a Laurent series in powers of $(1-t)$,
say
\[
r(t) =  a_n(1-t)^n + a_{n+1}(1-t)^{n+1} + \cdots
\]
 where $n \in \mb{Z}$ and $0 \neq a_n \in \mb{Q}$.  Then following \cite{ATV2}, $a_n$ is called the \emph{multiplicity} of $r(t)$ and we write it as $\eps(r(t))$.  By convention, we also set $\eps(0) = 0$.
In particular, if $M$ is a graded module with a rational total Hilbert series $h_M\tot(t) = r(t)$, then we define the multiplicity
$\eps(M)$ of $M$ to be $\eps(r(t))$.

The following is well-known in the theory of GK dimension.
\begin{lemma}
\label{lem:rational series}
Let $0 \neq h(t) = p(t)q(t)^{-1}$ where $p(t) \in \mb{Z}[t, t^{-1}]$, $q(t) \in \mb{Z}[t]$ with $q(0) = \pm 1$, and $p(t), q(t)$ are relatively prime.
Suppose that its Laurent series expansion $h(t) = \sum_{n \geq n_0} a_n t^n \in \Z((t))$ centered at zero has nonnegative coefficients.
Let
\[
d = \lim \sup_{n \to \infty} \log_n \left( \sum_{i=n_0}^n a_i \right)
\]
 denote the GK dimension of $h(t)$. Then:
\begin{enumerate}
\item The following are equivalent:
\begin{enumerate}
\item $h(t)$ has subexponential growth;
\item The roots of $q(t)$ all lie on the unit circle in $\mb{C}$;
\item The roots of $q(t)$ are all roots of unity;
\item $d < \infty$.
\end{enumerate}
\item If $d < \infty$, then $d$ is equal to the multiplicity of $t = 1$ as a root of $q(t)$ (equivalently,
the order of the pole of $h(t)$ at $t = 1$).
\item If $d < \infty$, the multiplicity $\eps(h(t))$ of $h(t)$ is positive and is an integer multiple of $\eps(q(t))^{-1}$.
\end{enumerate}
\end{lemma}
\begin{proof}
This is basically a combination of \cite[Proposition 2.21]{ATV2} and \cite[Lemma 2.1, Corollary 2.2]{SZ}.
First if $q(0) = -1$, by replacing $p$ and $q$ by $-p$ and $-q$ the hypotheses are preserved, and so
we can assume that $q(0) = 1$.

(1) Suppose that $h(t)$ has subexponential growth.  Then the results \cite[Lemma 2.1, Corollary 2.2]{SZ} show that the roots of $q(t)$ are all on the unit circle.
Since $q(0) = 1$, we can write $q(t) = \prod_i (1 - r_i t)$; if the roots of $q(t)$ are all on the unit circle, then since $q(t) \in \mb{Z}[t]$ we get that the leading coefficient of $q(t)$ is $\pm 1$.
As noted in \cite{SZ}, this implies that the roots of $q(t)$ are all roots of unity (by a theorem of Kronecker, which states that given a monic integer polynomial whose roots are all in the unit disk, the roots must all be roots of unity).
Next, assuming that the roots of $q(t)$ are all roots of unity, the proof of \cite[Proposition 2.21]{ATV2} shows that
$h(t) = b_{-d} (1-t)^{-d} + b_{-d+1}(1-t)^{-d+1} + \cdots $ where $0 < b_{-d} = \eps(h(t))$, and where the order $d$ of the pole of $h(t)$ at $t = 1$ is equal to the GK-dimension of $h(t)$.  This order is the same as the multiplicity of $t =1$ as a root of $q(t)$ since
$p(t)$ and $q(t)$ are relatively prime.  (Although this proposition from \cite{ATV2} is stated for the Hilbert series of a module over
a regular algebra, the proof only uses the properties of the Hilbert series we assume in our hypotheses.)  Finally it is trivial that if the GK-dimension of $h(t)$
is a finite number $d$, then $h(t)$ has subexponential growth.

(2) This was proven in the course of the proof of (1).

(3) It was established that $\eps(h(t)) > 0$ in the proof of (1).  Note that both $t = 1 -(1-t)$ and $t^{-1} = (1 - (1-t))^{-1} = (1 + (1-t) + (1-t)^2 + \cdots)$ have integer coefficients when expressed as series in powers of $(1-t)$.  Then the same is true of any power of $t$ and hence any element of $\mb{Z}[t, t^{-1}]$.  Thus $\eps(p(t)) \in \mb{Z}$.  It is also clear that $\eps(p(t)q(t)^{-1}) = \eps(p(t))\eps(q(t))^{-1}$ since multiplicity is multiplicative.  Thus $\eps(h(t))$ is an integer multiple of $\eps(q(t))^{-1}$ as required.
\end{proof}

We now give some extensions of the previous results to the matrix case.

\begin{proposition}
\label{prop:growth}
Let $A$ be a locally finite $\mb{N}$-graded algebra with fixed decomposition $1 = e_1 + \cdots + e_r$ of $1$ as a sum of pairwise orthogonal primitive idempotents in $A_0$.   Suppose that
$q(t) \in M_r(\mb{Z}[t])$ has $(\det q)(0) = \pm 1$ and let $c(t) \in M_r(\ZZ[t])$ be the adjugate matrix of $q(t)$,
so that $q(t)^{-1} = (\det q(t))^{-1} c(t)$ in the ring $M_r(\mb{Z}[[t]])$.
\begin{enumerate}
\item Suppose that $M$ is a nonzero graded left $A$-module with vector Hilbert series of the form $q(t)^{-1} v(t)$,
where $v(t) \in \mb{Z}[t, t^{-1}]^r$ is a column vector.   If $\GKdim(M) < \infty$, then $\GKdim(M)$ is the maximal order of $1$ as a pole of the set of rational functions $h_i =  (\det q(t))^{-1} \sum_j c_{ij}(t) v_j(t)$ over all $1 \leq i \leq r$.
\item Assume that $A$ has a rational matrix-valued Hilbert series $h_A(t) = q(t)^{-1} p(t)$, where
$p(t) \in M_r(\mb{Z}[t])$.   If $\GKdim(A)  < \infty$, then $\GKdim(A)$ is equal to the maximal order of $1$ as a pole of the rational functions
\[
(h_A(t))_{il} = (\det q(t))^{-1} (c(t) p(t))_{il} = (\det q(t))^{-1} \sum_j c_{ij}(t) p_{jl}(t)
\]
 over all $1 \leq i, l \leq r$.    In particular, if $d \leq \GKdim(A) < \infty$ then $\det q(t)$ must vanish at $t = 1$ to order at least~$d$.
\item Assume that $h_A(t) = q(t)^{-1} p(t)$ as in (2), where the matrix polynomials $p(t)$ and $q(t)$ have no roots in common.
Then the set $\mc{R}$ of roots of $q(t)$ is equal to the union of the set of poles of the matrix entries of
$h_A(t)$.  The series $h\tot_A(t)$ is also a rational function and its set of poles is contained in $\mc{R}$.
In particular, $\GKdim(A) < \infty$ if and only if every root of $q(t)$ is a root of unity, if and only if every root of $q(t)$ is on the unit circle; otherwise,
$A$ has exponential growth.
\end{enumerate}
\end{proposition}

\begin{proof}
(1)  The vector Hilbert series of $M$ is
\[
h_M(t) = q(t)^{-1} v(t) = (\det q(t))^{-1} c(t) v(t) = (h_1(t), h_2(t), \dots, h_r(t))^T
\]
by definition.  Write $h_i(t) = f_i(t)g_i(t)^{-1}$ with $f_i(t) \in \ZZ[t, t^{-1}]$, $g_i(t) \in \ZZ[t]$ where
$f_i(t), g_i(t)$ have no common factors for each $i$, and where necessarily $g_i(t)$ is a factor of $\det q(t)$ in $\ZZ[t]$.
 Since the constant term of $(\det q(t))$ is $1$, $g_i(t)$ has constant term $\pm 1$. Then Lemma~\ref{lem:rational series} applies to each nonzero $h_i(t)$; if $\GKdim(M) < \infty$, then each of the $h_i(t)$ has polynomially bounded growth and so its GK-dimension is the maximum order of $1$ as a pole of $h_i(t)$.  Since $M \neq 0$, at least one of the $h_i(t) \neq 0$.
The GK-dimension of $h_M(t)$ is equal to the maximum of the GK-dimensions of the nonzero $h_i(t)$ (for example,
by the same proof as \cite[Proposition 3.2]{KL}), and the result follows.

 (2).  Since $A \cong Ae_1\oplus \cdots \oplus Ae_r$, we have that $\GK(A)$ is the maximum of $\GK(Ae_l)$ over $1 \leq l \leq r$.
Moreover, each $\GK(Ae_l)$ is a module whose vector Hilbert series is $q(t)^{-1} v_l(t)$ where $v_l(t)$ is the
$l$th column of $p(t)$.  The first statement now follows from (1).

Now note that each rational function
$(h_A(t))_{il} = (\det q(t))^{-1} \sum_j c_{ij}(t) p_{jl}(t)$ may be written in lowest terms as $f(t) g(t)^{-1}$ with $f(t) \in \mb{Z}[t, t^{-1}]$
and $g(t) \in \mb{Z}[t]$ a factor of $\det q(t)$, so the order of the pole of $(h_A(t))_{il}$ at $t = 1$ is at most the order of vanishing of $\det q(t)$ at $t = 1$.  This proves the second statement.

(3)  We consider $h_A(t) = q(t)^{-1}p(t) = (\det q(t))^{-1} p(t) c(t)$ as a matrix-valued function whose entries $(h_A(t))_{il}$
are meromorphic (indeed, rational) functions of a complex variable $t$.  That is, if $\mb{F}$ is the field
of meromorphic functions, we work in $M_{r}(\mb{F})$.   It is clear that if $M,N \in M_r(\mb{F})$ where no entry of $M$ or $N$ has a pole at $\lambda \in \mb{C}$, then every entry of the product $MN$ also has no pole at $\lambda$, so the evaluations $M(\lambda)$, $N(\lambda)$, and $[MN](\lambda)$ make sense, and
$[MN](\lambda) = M(\lambda)N(\lambda)$.

Note that the roots of the matrix polynomial $q(t)$ are the same as the usual roots of the polynomial $\det q(t)$.
Now if $\lambda \in \C$ is not a root of $q(t)$, then every matrix entry of $h_A(t) =  (\det q(t))^{-1} p(t) c(t)$ has no pole
at $\lambda$.
Conversely, suppose that $\lambda \in \C$ is not a pole of any entry of $h_A(t)$.   Then $h_A(\lambda)$ is defined.
 Since by assumption $p(t), q(t) \in M_r(\mb{Q}[t])$, these matrix polynomials have no poles anywhere.
Now in $M_r(\mb{F})$ we have $h_A(t) q(t) = p(t)$, and so $h_A(\lambda) q(\lambda) = p(\lambda)$ in $M_r(\mb{C})$.  If $\lambda$ is a root of $q(t)$, then $q(\lambda)$ is singular.  This forces $p(\lambda)$ to be singular as well, so that $\lambda$ is also a root of $p(t)$, which contradicts the hypothesis that $p(t)$ and $q(t)$ have no
common roots.  Thus $\lambda$ is not a root of $q(t)$.

We have now proved that the roots of $q(t)$ are the same set $\mc{R}$ as the union of the poles of the matrix entries $(h_A(t))_{il}$.   As $h\tot_A(t)$ is the
sum of these entries over all $i, l$, we see that it is a rational function with poles contained in $\mc{R}$.

Now $A$ has finite GK-dimension if and only if each matrix entry $(h_A(t))_{il}$ is a series with finite GK-dimension.  Each of these
entries $(h_A(t))_{il}$ can we written as a rational function $f(t)g(t)^{-1}$ where $f(t), g(t) \in \mb{Z}[t]$ have no common factors and where $g(t)$ is a factor of $\det q(t)$ in $\mb{Z}[t]$.  Thus $g(0) = \pm 1$ and by Lemma~\ref{lem:rational series}, $(h_A(t))_{il}$ has
finite GK-dimension if and only if every root of $g(t)$ is a root of unity, if and only if every root of $g(t)$ is on the unit circle.  Since the union of
the poles of the $(h_A(t))_{il})$ is $\mc{R}$, we see that $A$ has finite GK-dimension if and only if every element of $\mc{R}$ is a root of unity.
Otherwise, by Lemma~\ref{lem:rational series}, some matrix entry of $h_A(t)$ must have exponential growth, and thus $A$
has exponential growth.
\end{proof}

\begin{example}
\label{ex:matrix-vs-total}
Recall that for a graded $k$-algebra $B$ and a finite group $G$ acting on $B$ by graded automorphisms, we have the skew group algebra $B \# kG$.  This algebra is equal to $B \otimes_k kG$ as a vector space, but with multiplication $(a \otimes g) * (b \otimes h) = ag(b) \otimes gh$
for $a, b \in B$, $g, h \in G$.  The algebra $B \# kG$ is again graded, where $kG$ has degree $0$.

Now let $G = \mb{Z}_2 = \langle \sigma \rangle$ act on $k[x, y]$ where $\sigma(x) = x, \sigma(y) = -y$, and let $A = k[x,y] \# kG$.  Since the polynomial ring $k[x,y]$ has Hilbert series $1/(1-t)^2$, it is obvious that $h\tot_A(t) = 2/(1-t)^2$. 
It is straightforward to calculate the matrix Hilbert series of $A$, as follows.
We identify $k[x,y]$ and $kG$ with subalgebras of $A$.   We have $A_0 = kG = ke_1 + ke_2$ where $1 = e_1 + e_2$ is a decomposition into primitive orthogonal idempotents; in this case, if $G = \{e, g\}$ with $e$ the identity element and $g^2 = e$, then we may take $e_1 = (e+g)/2$, $e_2 = (e-g)/2$.  A short calculation shows that $e_1 A e_1 = k[x, y^2] e_1$ and $e_2 A e_1 = y k[x, y^2] e_1$, and similarly
$e_2 A e_2 = k[x, y^2] e_2$, $e_1 A e_2 = y k[x, y^2] e_2$.  Thus we have
\[
h_A(t) = \begin{pmatrix} \frac{1}{(1-t)(1- t^2)} & \frac{t}{(1-t)(1-t^2)}  \\ \frac{t}{(1-t)(1-t^2)}  & \frac{1}{(1-t)(1-t^2)}  \end{pmatrix}.
\]
A direct check shows that we also have $h_A(t) = q(t)^{-1}$, where
$q(t) = \begin{psmatrix} 1-t & -t+t^2 \\ -t+t^2 & 1-t \end{psmatrix}$ and $\det q(t) = (1-t^2)(1-t)^2$.
Thus we may apply Proposition~\ref{prop:growth}(2)(3) to $A$, with $p(t) = I$.
This example shows that in Proposition~\ref{prop:growth}(2), $\det q(t)$ may well vanish to a greater order at $t = 1$ than
$\GKdim(A)$.  In our example, $\det q(t)$ has 1 as a root of multiplicity 3, but $\GKdim(A) = 2$ because each matrix entry of $h_A(t)$ has a pole
at $t =1$ only of multiplicity $2$.  The reason for this is cancellation that happens when putting each $(h_A(t))_{il} = (\det q(t))^{-1} \sum_j c_{ij}(t) p_{jl}(t)$ in lowest terms.  Similarly, in Proposition~\ref{prop:growth}(3) the set $\mc{R}$ of roots of $q(t)$ may be bigger than the set of poles of the series $h\tot_A(t)$.  In our example, $\mc{R} = \{ 1, -1 \}$ whereas $h\tot_A(t)$ has a pole only at $t = 1$.  This is caused by cancellation  that happens after adding together the matrix entries of $h_A(t)$ to get the total Hilbert series.

In fact, it is well known that $A$ is a twisted Calabi-Yau algebra of global dimension $2$; for example,
see \cite[Theorem 4.1(c)]{RRZ1}. The matrix Hilbert series of $A$ can also be calculated from the theory of Section~\ref{sec:dim2} below.
\end{example}

\section{Elementary algebras and quivers}
\label{sec:elemquiv}

The basic ideas in this section are standard, but we give full details since the case of algebras that are not necessarily generated in degree~1 seems to be less well-known.

Let $A$ be a locally finite $\mb{N}$-graded $k$-algebra.  Let $J = J(A)$ be the graded Jacobson radical, that is, the intersection of all graded maximal right ideals of $A$.  Then $A_{\geq 1} \subseteq J$ and so $J = A_{\geq 1} \oplus J(A_0)$, where $J(A_0)$ is the usual Jacobson radical
of the finite-dimensional algebra $A_0$; see~\cite[Corollary~A.II.6.5]{NV1}. Thus $S = A/J \cong A_0/J(A_0)$ is a finite-dimensional semisimple algebra.  By Wedderburn's theorem, $S$ is a product of matrix rings over division rings which are finite-dimensional over $k$.
Following similar terminology in the case of finite-dimensional algebras~\cite[p.~65]{ARS}, we will say that $A$ is \emph{elementary} if $S$ is isomorphic to a product $k^n$ of copies of the base field $k$.   This is not usually a serious restriction; for instance, if $k$ is an algebraically closed field and $S$ is
commutative, then the semisimple algebra $S$ is a product of copies of $k$, so that $A$ is elementary.
For many purposes, in particular for the study of twisted Calabi-Yau algebras, it is easy to reduce to this case, as we will explain in Section~\ref{sec:CY}.

Suppose that $A$ as above is elementary.
The $n$ orthogonal primitive idempotents of $S = k^n$ can be lifted (not necessarily uniquely) to an orthogonal family of
primitive idempotents with $1 = e_1 + \cdots + e_n$ in $A_0$; see~\cite[Corollary~21.32]{FC}.
Note that  writing $1 = e_1 + \cdots + e_n$ as a sum of primitive pairwise orthogonal idempotents in $A_0$, then $T = ke_1 + \cdots + ke_n$ is a subalgebra of $A$ isomorphic to $k^n$.  Thus when $A$ is elementary, we can identify $T$ with $S$ via the map $T \to A \to A/J(A) = S$ which is an isomorphism of $k$-algebras.
In this way we may identify $S$ with a subalgebra of $A$, which again is not necessarily unique. Nevertheless, this identification is frequently useful; for example, this allows us to refer to left $S$-submodules of $A$, which are $k$-subspaces of $A$ closed under left multiplication by all of the $e_i$.   Since $A = \bigoplus_{i=1}^n Ae_i$,
each $Ae_i$ is a graded projective left $A$-module, and it is an indecomposable module since $e_i$ is a primitive idempotent.  As $i$ varies we get $n$ distinct simple left modules $Se_i = Ae_i/Je_i$, and every graded simple left module is isomorphic to a shift of one of the $Se_i$.

Continue to assume that $A$ is elementary.  Let $P$ and $Q$ be left bounded graded projective $A$-modules.  We claim that there is a graded isomorphism $P \cong Q$ if and only if there is a graded isomorphism of left $S$-modules $P/JP \cong Q/JQ$.  The forward direction is immediate; for the converse, by projectivity the isomorphism $f: P/JP \cong Q/JQ$ can be lifted to a graded $A$-module map $\widetilde{f}: P \to Q$, which is surjective by the graded Nakayama Lemma \cite[Lemma 2.2]{RR1}.  Since the short exact sequence $0 \to K \to P \overset{\widetilde{f}}{\to} Q \to 0$ is split, the sequence $0 \to K/JK \to P/JP \to Q/JQ \to 0$ is also exact, forcing $K/JK = 0$ and hence $K = 0$ by Nakayama's lemma again; then $\widetilde{f}$ is an isomorphism, proving the claim.  Since $S$ is semisimple, there is a graded isomorphism $P/JP \cong \bigoplus_i Se_{m_i}(\ell_i)$, and hence by the claim there is an isomorphism $P \cong \bigoplus_{i} A e_{m_i}(\ell_i)$.  The list of pairs $(m_i, \ell_i)$ is unique up to permutation, using the Krull-Schmidt theorem for $S$-modules.  In particular, the $Ae_i$ are the only indecomposable
left bounded graded projective modules.

In the next result, we give some basic information on the shapes of the minimal projective resolutions of modules over
an elementary algebra.  For a more detailed review of this topic, see \cite[Section 2]{RR1}.
All complexes will be homological by convention in this paper. By contrast, the complexes in~\cite{RR1} were cohomological (cochain complexes), and we have made appropriate notational changes when applying results from that paper.  Let $\cdots \to P_n \to \cdots \overset{d_2}{\to} P_1 \overset{d_1}{\to} P_0 \overset{\epsilon}{\to} M \to 0$ be a projective resolution of the graded module $M$ over the graded algebra $A$.  The resolution is called \emph{graded} if all $P_i$ are graded projective modules and all maps $d_i$ as well as the augmentation map $\epsilon$ are homogeneous of degree $0$.   If $P$ is graded projective and $f: P \to M$ is a graded surjection, we call this a \emph{minimal} surjection if $f$ induces an isomorphism
$P/JP \to M/JM$, or equivalently, if $\ker f \subseteq JP$.  By the graded Nakayama lemma, every left bounded graded module admits a minimal surjection from a graded projective module.  A graded resolution as above is \emph{minimal} if all of the $d_i$ as well as $\epsilon$ are minimal surjections onto their images, or equivalently if $\ker d_i \in JP_i$ for all $i$ and $\ker \epsilon \in JP_0$.   Then every left bounded graded module has a minimal graded projective resolution, which is unique up to isomorphism of complexes.

\begin{proposition}
\label{prop:resolution form1}
Let $A$ be a finitely graded $k$-algebra with $S = A/J(A)$.
Assume that $A$ is elementary with $S = k^n$ and write $1 = e_1 + \cdots + e_n$ as a sum of primitive pairwise orthogonal idempotents $e_i \in A_0$.
Let $P_{\bullet}$ be the minimal projective resolution of the finitely generated graded left module $M$, and
assume that each $P_i$ is finitely generated.  Then
$P_i = \bigoplus_{j, m} [A e_j(-m)]^{n(M,i,j,m)}$ where
\[
n(M, i,j,m) = \dim_k \Tor_i^A(e_jS, M)_m = \dim_k \Ext^i_A(M, Se_j)_{-m}.
\]
\end{proposition}
\begin{proof}
(See~\cite[Lemma~2.6]{RR1} for a similar result.)
First, we can write each of the projective modules $P_i$ as $P_i = \bigoplus_{j, m} [A e_j(-m)]^{n(M,i,j,m)}$ for some numbers $n(M,i,j,m)$, because
$P_i$ is a direct sum of finitely many indecomposable graded projectives, each of which is a shift of one of the $Ae_j$.

Because $P_{\bullet}$ is a minimal resolution, $S \otimes_A P_{\bullet}$ is a complex whose morphisms are all $0$.  Thus
\[
\Tor_i^A(S, M) \cong S \otimes_A P_i = P_i/JP_i = \bigoplus_{j, m} [S e_j(-m)]^{n(M,i,j,m)}.
\]
Note that $\Tor_i^A(e_j S, M)$ is the $i$th homology of $e_j S \otimes_A P_{\bullet} = e_j(S \otimes_A P_{\bullet})$ and so
\[
\Tor_i^A(e_j S, M) = \bigoplus_{r,m} [e_j S e_r(-m)]^{n(M,i,r,m)}.
\]
Then since $A$ is elementary, we have $\dim_k e_j S e_r = \delta_{jr}$ and thus
\[
\dim_k \Tor_i^A(e_j S, M)_m =  \dim_k \bigoplus_{r, p} [e_j S e_r(-p)]_m^{n(M,i,r,p)} = n(M,i,j,m).
\]
Similarly, $\Hom_A(P_{\bullet}, S)$ is a complex with $0$ morphisms and hence
\begin{align*}
\Ext^i_A(M,S) &\cong \Hom_A(P_i, S) =\Hom_A(P_i/JP_i, S) \\
&=  \Hom_A(\bigoplus_{j, m} [S e_j(-m)]^{n(M,i,j,m)}, S) = \bigoplus_{j, m} e_j S(m)^{n(M,i,j,m)}.
\end{align*}
Now $\Ext^i_A(M, Se_j)$ is the $i$th homology of $\Hom_A(P_{\bullet}, Se_j) = \Hom_A(P_{\bullet}, S) e_j$
(see \cite[Lemma~2.8]{RR1}) and so we have
\[
\dim_k \Ext^i_A(M, Se_j)_{-m} =  \dim_k \bigoplus_{r, p} [e_r S e_j(p)]_{-m}^{n(M,i,r,p)} = n(M,i,j,m),
\]
as claimed.
\end{proof}

Next, we recall how any elementary $k$-algebra can be presented as a path algebra of a weighted quiver with relations.
The reader can find general background on quivers in \cite{ASS}.  For us, a \emph{weighted quiver} $Q$ is a finite directed graph with a nonnegative integer (the \emph{degree} or \emph{weight}) associated to each arrow.  The (weighted) path algebra $kQ$ is the usual path algebra with the unique $\mb{N}$-grading coming from
giving each arrow a degree equal to its weight.
Following \cite{ASS}, for an arrow $\alpha \in Q$ pointing
from vertex $i$ to vertex $j$, we call $i$ the \emph{source} of the arrow and $j$ the \emph{target}, and write $s(\alpha) = i$ and $t(\alpha) = j$.
We use the convention that arrows are composed in $kQ$ from left to right.
The grading of the weighted path algebra is such that its graded components $(kQ)_n$ are given by the span of all
paths of total degree~$n$.

In $kQ$ we have the canonical decomposition of $1 = e_1 + \cdots + e_n$ of $1$ as  a sum of primitive
orthogonal idempotents, where $e_i$ is the trivial path at vertex $i$.  We let $S = \bigoplus k e_i$ denote the
span of these idempotents, which forms a subalgebra of $kQ$, called the \emph{vertex space}.
We define the \emph{arrow space} $V$ of $kQ$ to be the $k$-span of the arrows in $Q$.
It is straightforward to see that $V$ is a graded $(S,S)$-subbimodule of the weighted path algebra $kQ$.

It is well known (see~\cite[Proposition~III.1.3]{ARS}, for instance) that the path algebra is isomorphic to a tensor algebra
in the following way. Given an integer $n \geq 0$, let $V^{\otimes n} = V \otimes_S \otimes \cdots \otimes_S V$
denote the $n$-fold tensor power of the $(S,S)$-bimodule $V$, with the convention that $V^{\otimes 0} = S$.
Then we may form the tensor algebra $T_S(V) = \bigoplus_{n = 0}^\infty V^{\otimes n}$.
One can check that $V^{\otimes n}$ is isomorphic as an $(S,S)$-bimodule to the $k$-linear span $P_n$ of all paths
with length (\emph{not} degree!)\ equal to~$n$ in $kQ$, where $P_0 = S$. But then
\[
kQ = \bigoplus_{n \geq 0} P_n \cong \bigoplus_{n \geq 0} V^{\otimes n} = T_S(V)
\]
as $(S,S)$-bimodules, and it is not hard to show that this is in fact an isomorphism of $k$-algebras that fixes $S$.
If one endows $kQ$ with the weight grading and $T_S(V)$ with the grading inherited from $V$, then this
is a graded isomorphism.
In particular, it follows that $kQ$ possesses the universal property~\cite[Lemma~III.1.2]{ARS} of the
tensor algebra $T_S(V)$.

\begin{lemma}\label{lem:path algebra}
Let $Q$ be a weighted quiver, and let $kQ$ be the associated graded path algebra. Let $Q_0$ denote
the subquiver of $Q$ consisting of the same vertices as $Q$, but only those edges of weight~0. Then:
\begin{enumerate}
\item The degree zero subalgebra of $kQ$ is $(kQ)_0 = kQ_0$.
\item $kQ$ is locally finite if and only if $Q_0$ is acyclic.
\item If $kQ$ is locally finite, then $J(kQ) = J(kQ_0) \oplus (kQ)_{\geq 1}$ is the $k$-linear span of all paths
of length $\geq 1$.
\end{enumerate}
\end{lemma}

\begin{proof}
(1) As mentioned above, $(kQ)_0$ is the $k$-linear span of all paths of total degree~0. This includes the
trivial paths $e_i$, and any nontrivial paths of degree~0 are paths of arrows lying in $Q_0$. It follows
that $(kQ)_0 = kQ_0$.

(2) If $kQ$ is locally finite, then $(kQ_0) = kQ_0$ is finite-dimensional, which implies that $Q_0$ is acyclic.
Conversely, suppose that $Q_0$ is acyclic. Let $N$ be an integer larger than the number of arrows in $Q_0$.
As stated above, the graded $(S,S)$-bimodule $V^{\otimes N}$ can be identified with the span of all paths of length~$N$ in $kQ$.  Because $Q_0$ is acyclic, it contains no paths of length $N$; equivalently, every path of length $N$ in $Q$ must contain an arrow outside of $Q_0$. Thus $V^{\otimes N}$ has non nonzero elements of degree~$0$.
It follows as in the discussion preceding \cite[Lemma 6.9]{RR1} that $T_S(V)$ is locally finite. Thus $kQ \cong T_S(V)$ is locally finite.

(3) Assuming that $kQ$ is locally finite, its graded Jacobson radical decomposes as
\[
J(kQ) = J((kQ)_0) \oplus (kQ)_{\geq 1} = J(kQ_0) \oplus (kQ)_{\geq 1},
\]
where the second equality follows from part~(1).
The first summand is the span of all nontrivial paths of total degree~$0$, and the second summand is the span of all paths of total degree~$\geq 1$.
The direct sum of these two subspaces is now readily seen to be the span of all paths of length~$\geq 1$.
\end{proof}

We now discuss how any elementary locally finite graded $k$-algebra $A$ can be minimally presented as a factor of a path algebra of a weighted quiver, depending on some non-canonical choices in $A$. The construction of the quiver is facilitated by the following.

\begin{definition}\label{def:arrow space}
Let $A$ be an elementary locally finite graded algebra. A \emph{vertex basis} for $A$ is an
(ordered) collection of orthogonal primitive idempotents $\{e_1, \dots, e_n \}$ in $A_0$ such that
$1 = \sum e_i$. The corresponding \emph{vertex space} is $ke_1 + \cdots + ke_n$; as stated earlier,
this is a subalgebra of $A_0$ that can be identified with $S = A/J(A)$.
Given such a vertex space, an \emph{arrow space} is a graded $(S,S)$-subbimodule $V$ of $A$
such that
\[
J(A) = V \oplus J(A)^2
\]
as $(S,S)$-bimodules. An \emph{arrow basis} for $V$ is a basis $\mc{B}$ that is
a union of homogeneous bases for each of the summands in the graded vector space decomposition
$V = \bigoplus_{i,j} e_i V e_j$;
that is, $\mc{B} = \{\alpha_1, \dots, \alpha_r\}$ is an (ordered) basis for $V$ such that each
$\alpha_m \in e_i A e_j$ for
some $i,j$.
Fixing  a vertex basis and arrow basis, the \emph{associated weighted quiver} $Q$ has
vertices $1, \dots, n$, and $r$ arrows we label $\wh{\alpha}_m$, where $\wh{\alpha}_m$ is an
arrow from $i \to j$, with the same weight as the degree of $\alpha_m \in A$.
\end{definition}

We note if $A$ is a locally finite elementary algebra and $S \subseteq A$ is any fixed vertex space,
there always exists an arrow space $V$ for $A$. Indeed, note that $S \cong k^n$ is a separable $k$-algebra,
and thus the enveloping algebra $S^e  = S \otimes S\op$ is semisimple. The graded $(S,S)$-bimodule $A$ is
$k$-central, and thus can be viewed as a graded $S^e$-module in the usual way (see~\cite[Section~1]{RR1},
for instance). Thus the inclusion $J(A)^2 \subseteq J(A)$ of graded $S^e$-submodules is split, and there
is an $S^e$-submodule $V$ such that $J(A) = V \oplus J(A)^2$. But this $V$ is simply a graded $k$-central
$(S,S)$-subbimodule.

Recall that the $k$-algebra $A$ is \emph{indecomposable}
if it cannot be written as a product of two nontrivial algebras; this is equivalent to $A$ having no nontrivial central idempotents.  An $\mb{N}$-graded algebra $A$ is indecomposable if and only if it is graded indecomposable, that is, $A$ cannot be written as a product of two nontrivial graded algebras, or equivalently, $A$ has no homogeneous nontrivial central idempotents \cite[Lemma~2.7]{RR1}.

\begin{lemma}
\label{lem:quiver facts}
Suppose that $A$ is a finitely graded elementary $k$-algebra with fixed choice of vertex basis and compatible arrow basis.  Let $Q$
be the associated weighted quiver as described above, and denote $J = J(kQ)$.
\begin{enumerate}
\item There is a graded isomorphism $kQ/I \cong A$, for some homogeneous ideal $I \subseteq J^2$.
The quiver $Q$ is the unique weighted quiver up to isomorphism satisfying this condition.
\item For vertices $i$ and $j$ in $Q$, the number of arrows of weight $m$ from $j$ to $i$ in $Q$ is equal to $\dim_k \Ext^1_A(Se_i, Se_j)_{-m}$.
\item The quiver $Q$ is connected (as an undirected graph) if and only if $A$ is indecomposable as an algebra.
\end{enumerate}
\end{lemma}
\begin{proof}
(1) Let the trivial path at vertex $i$ in $Q$ be $\wh{e_i}$.  By the universal property of $kQ$, there exists a unique graded $k$-algebra homomorphism $\phi: kQ \to A$ such that $\phi(\wh{e_i}) = e_i$ for all $i$ and $\phi(\wh{\alpha_i}) = \alpha_i$ for all $i$.  Let $V = k\alpha_1 + \cdots + k \alpha_r$ be the arrow space.  As in the proof of \cite[Lemma~2.3]{RR1}, we have $A = S + V + V^2 + \cdots$ as graded $k$-spaces.    By construction, the image of $\phi$ contains $S$ and $V$, so $\phi$ is surjective.  In fact, by construction $\phi$ induces a graded bijection from the $k$-span of the paths of length $\leq 1$ in $Q$ to $S \oplus V$, and $A = S \oplus V \oplus J^2(A)$.  So $\ker \phi \subseteq J^2$, since $J^2$ is the $k$-span of paths in $Q$ of length at least $2$ by Lemma~\ref{lem:path algebra}(3).
The uniqueness of $Q$ follows directly from part~(2) of the statement, which will be proved next.

(2) The minimal left projective resolution of $Se_i$ begins with
\[
P_1 \to P_0 \overset{\epsilon}{\to} Se_i \to 0.
\]
Clearly $P_0 = Ae_i$, and Proposition~\ref{prop:resolution form1} states $P_1 = \bigoplus_{j, m} A e_j(-m)^{n(Se_i, 1, j,m)}$
for
\[
n(Se_i, 1, j, m) = \dim_k \Ext^1_A(Se_i, Se_j)_{-m}.
\]
  The projective $P_1$ minimally surjects onto $\ker \epsilon = Je_i$, and so we must also have
\[
Je_i/J^2e_i \cong P_1/JP_1 = \bigoplus_{j,m} Se_j(-m)^{n(Se_i, 1, j, m)}
\]
as left modules. Because $\dim_k e_j S e_i = \delta_{ij}$, it follows that
\[
n(Se_i, 1, j, m) = \dim_k (e_jJ e_i/e_j J^2 e_i)_m.
\]
Since the arrow space $V$ satisfies $J = V \oplus J^2$, we have $e_jJe_i/e_jJ^2e_i \cong e_jVe_i$.
Thus
\[
n(Se_i, 1, j, m) = \dim_k (e_j V e_i)_m
\]
is also the number of arrows of weight $m$ from $j$ to $i$ in $Q$.

(3)  Write $A \cong kQ/I$ for the weighted quiver $Q$ and $I \subseteq J^2$, as in part (1).   Let $e \in A$ be
a nontrivial central idempotent.  By the proof of \cite[Lemma 2.7]{RR1}, if $e = e_0 + e_1 + \cdots$ with $e_i \in A_i$, then $e_0$ is also a nontrivial central idempotent in $A$.   Since $I \subseteq J^2$, we have $kQ/J^2 \cong A/J^2(A)$ as rings, and thus $kQ/J^2$ has a nontrivial central idempotent of degree $0$.  The latter ring can be identified with the path algebra of $Q$ with relations given by every path of length $2$.  It is easy to see that such a ring has no nontrivial central idempotents of degree $0$ when $Q$ is connected; see \cite[Lemma II.1.7]{ASS} for the idea (having weighted arrows makes no difference to the proof).
 \end{proof}

\begin{definition}
Given a locally finite elementary $k$-algebra $A$, we call the weighted quiver $Q$ in the theorem above the
\emph{underlying (weighted) quiver} of the algebra $A$.  We define the (weighted) incidence matrix of $Q$ to be
$N(t) = \sum_{m \geq 0} V_m t^m \in M_n(\mb{Z}[t])$, where $V_m \in M_n(\mb{Z})$ is a matrix with $(V_m)_{ij}$ equal to the number of arrows from $i$ to $j$ of weight $m$.  The usual unweighted incidence matrix of $Q$ is $M \in M_n(\mb{Z})$ where $M_{ij}$ is the number of
arrows of any weight from $i$ to $j$, so that $M = N(1)$.
\end{definition}

\begin{example}
Let $Q$ be a quiver with two vertices $1, 2$ and one arrow $\alpha$ from $1$ to $2$, so $kQ = k e_1 + k \alpha + k e_2$.  Then $kQ$ has the obvious vertex basis $\{e_1, e_2 \}$ and arrow basis $\{ \alpha \}$, but the algebra $A = kQ$ also
has vertex basis $\{e_1 + c \alpha, e_2 - c \alpha \}$ for any $c \in k$ (with the same arrow basis).

On the other hand, even when a vertex basis is fixed, a path algebra $kQ$ will generally have many arrow bases, in particular
when the quiver has multiple arrows.  For example, let $Q$ be a quiver with one vertex and two loops $x, y$ of weights $1$ and $2$ respectively, with vertex basis $ \{e_1 \}$, so that $kQ \cong k\langle x,y \rangle$.   Of course $\{x, y \}$ is an arrow basis, but so is $\{x, y + x^2 \}$.
\end{example}

In the next result, we relate the beginning of the minimal projective resolutions of the simple modules over a factor of a finitely graded elementary $k$-algebra $A$ to the generators and relations of an algebra.  Again, this is a rather standard affair, but because we work in the setting of weighted quivers we give details for lack of an exact reference.
Let $I$ be a ideal of an algebra $A$.  A subset $X \subseteq I$ of an algebra is said to \emph{minimally generate} the ideal $I$ if no proper subset of $X$ generates $I$ as a $2$-sided ideal.  For a fixed decomposition $1 = e_1 + \cdots + e_n$ of $1$ as a sum of primitive orthogonal idempotents in $A$, we say that a subset $Y$ of $A$ is \emph{$S$-compatible}  if each $y \in Y$ is in $e_j A_i e_k$ for some $i, j, k$.

\begin{lemma}
\label{lem:minres}
Let $Q$ be a weighted quiver and let $A = kQ/I$ for some homogeneous $I \subseteq J^2(kQ)$.  Write $e_i$ for the trivial path at vertex $i$ in $kQ$, identify it with its image in $A$ which we also write as $e_i$, and let $S = ke_1 + \cdots + ke_n$.  Let $X = \{x_1, \dots, x_p\} \subseteq J(kQ)$ be a $k$-linearly independent $S$-compatible set and  let $V = kX = kx_1 + \cdots + k x_p$.  Let $G = \{ g_m \mid 1 \leq m \leq b \}$ be an $S$-compatible set of elements in $I$.  Write $\deg g_i= s_i$, and suppose that we can write $g_i = \sum_j g_{ij} x_j$ for some elements $g_{ij} \in kQ$.   Let $\overline{x}$ be the image in $A$ of an element $x$ of $kQ$.

\begin{enumerate}
\item The complex
\begin{equation}
\label{eq:begin}
\xymatrix{
\displaystyle{\bigoplus_{i = 1}^b Ae_{m_i}(-s_i)} \ar[rr]^{\delta_2 = \left(\begin{smallmatrix}\overline{g_{ij}} \end{smallmatrix}\right)}
& & \displaystyle{\bigoplus_{j=1}^p Ae_{k_j}(-d_j)} \ar[rr]^{\delta_1 =  \left(\begin{smallmatrix} \overline{x_1} \\ \vdots \\ \overline{x_p} \end{smallmatrix}\right)}
& & A \to S \to 0
}
\end{equation}
(where the free modules consist of row vectors and the maps are right multiplication by the indicated matrices)
is the beginning of a minimal graded projective resolution of the left $A$-module $S$ if and only if $V$ is an arrow space for $kQ$
and the $ \{ g_i \}$ minimally generate the ideal $I$ of relations.

\item
Similarly, if $X \subseteq J(kQ)e_r$ and $G \subseteq kQ e_r$,
the complex
\begin{equation}
\label{eq:begin2}
\xymatrix{
\displaystyle{\bigoplus_{1 = 1}^b Ae_{m_i}(-s_i)} \ar[rr]^{\delta_2 = \left(\begin{smallmatrix}\overline{g_{ij}} \end{smallmatrix}\right)}
&& \displaystyle{\bigoplus_{j=1}^p Ae_{k_j}(-d_j)} \ar[rr]^{\delta_1 =  \left(\begin{smallmatrix} \overline{x_1} \\ \vdots \\ \overline{x_p} \end{smallmatrix}\right)}
&&  Ae_r \to Se_r \to 0
}
\end{equation}
is the beginning of a minimal projective resolution of the simple left $A$-module $Se_r$ if and only if
$V e_r\oplus J^2(KQ)e_r = J(KQ)e_r$ as $k$-spaces, and there is an $S$-compatible set $H$ which minimally
generates the ideal $I$ of relations, such that  $H \cap kQ e_r = G$.
\end{enumerate}
\end{lemma}

\begin{proof}
(1)
It is clear that for the first complex to be exact at $A = \bigoplus_{i=1}^n Ae_i$ we need $\delta_1$ to be a minimal surjection
onto $J(A)$.  By Nakayama's lemma, this is if and only if $\overline{x_1}, \dots, \overline{x_p}$ have images in $J(A)/J^2(A)$ which
are a $k$-basis.  Since $I \subseteq J^2(kQ)$, we have an identification $J(A)/J^2(A) = J(kQ)/J^2(kQ)$ and it is equivalent to
demand that $x_1, \dots, x_p$ have images in $J(kQ)/J^2(kQ)$ which are a $k$-basis.  This is the same as the
the condition that $V$ is an arrow space in $kQ$.

Assuming that $V$ is an arrow space, we claim that $\delta_2$ is a minimal surjection onto $\ker \delta_1$ if and only if the set $\{ g_i \}$ minimally generates the ideal $I$ of relations.  This is proved by a similar argument as in the connected graded case, as can be found in \cite[Lemma~2.1.3]{Ro}, but we give details here for the reader's convenience.

Consider $\ker \delta_1 = \{ (h_1, \dots, h_p)  \mid \sum_j h_j \overline{x_j}=0 \}$, which is a left $A$-submodule  
of $\bigoplus_{j=1}^p Ae_{k_j}(-d_j)$, and define $Y = \{ (h_1, \dots, h_p) \mid \sum_j h_j x_j \in I \} \subseteq \bigoplus_{j=1}^p kQe_{k_j}(-d_j)$.
Suppose that we are given a set of homogeneous elements
$\{ f_{ij} \mid 1 \leq i \leq b, 1 \leq j \leq p \} \subseteq kQ$, such that $f_i = \sum_j f_{ij} x_j \in I$ for all $i$, where
the set of elements $\{f_i \}$ is $S$-compatible.
Straightforward arguments prove that the following are all equivalent:
\begin{enumerate}
\item $\{ f_i \}$ generates $I$ as a 2-sided ideal of $kQ$.
\item $I = I x_1 + \cdots + I x_p + \textstyle \sum_i kQ f_i$.
\item $Y = I^{\oplus n} + \sum_i kQ (f_{i1}, \dots, f_{ip})$.
\item $\ker \delta_1 =  \sum_i A (\overline{f_{i1}}, \dots, \overline{f_{ip}})$.
\end{enumerate}
It follows that the set $\{ g_i \}$ generates $I$ as a $2$-sided ideal if and only if $\im \delta_2 = \ker \delta_1$.
Then $\{ g_i \}$ minimally generates $I$ if and only if $\delta_2$ is a minimal surjection onto $\ker \delta_1$.

(2)  The same argument as in the first part clearly shows that exactness at $Ae_r$
is equivalent to $V$ satisfying $V \oplus J^2(kQ)e_r = J(kQ)e_r$.   Assuming this holds,
given a complex as in \eqref{eq:begin2}, taking a direct sum of this complex with
the beginning of minimal projective resolutions of the simple modules $Se_j$ for all $j \neq r$ gives a potential minimal projective
resolution of $S$ as in \eqref{eq:begin}, which is the beginning of a minimal projective resolution of $S$ if and only if the complex in \eqref{eq:begin2} is the beginning of a minimal projective resolution of $Se_r$.  By the result of the first part,
this is if and only if the $\{ g_i \}$ are the part of a $S$-compatible minimal generating set of $I$ consisting of those relations
ending at vertex $r$, as required.
\end{proof}

\section{Matrix Hilbert series of homologically smooth elementary algebras}

In this section, we show how the results on matrix Hilbert series from the previous section apply to elementary algebras whose simple modules have nice projective resolutions.

We begin with some necessary definitions.
Recall that a module $M$ over an algebra $A$ is called \emph{perfect} if $M$ has a finite length projective resolution $0 \to P_d \to P_{d-1} \to \cdots \to P_{0} \to M \to 0$ where each $P_i$ is a finitely generated projective $A$-module.   If
$A$ is $\mb{N}$-graded and $M$ is a graded $A$-module, perfect is equivalent to \emph{graded perfect}, where we demand that
the $P_i$ are graded projective and the maps in the resolution are homogeneous of degree $0$ \cite[Lemma 2.4]{RR1}.

For a $k$-algebra $A$, let $A^e = A \otimes_k A^{\op}$ be its \emph{enveloping algebra}.  Note that $(A, A)$-bimodules are the same thing as left $A^e$-modules; in particular $A$ has a natural left $A^e$-structure given by its canonical $(A, A)$-bimodule structure.
An algebra $A$ is said to be \emph{homologically smooth} if $A$ is perfect as a left $A^e$-module.  If $A$ is $\mb{N}$-graded, then $A$ is homologically smooth if and only if it is \emph{graded} homologically smooth: that is, there is a graded projective resolution $0 \to P_d \to \cdots \to P_1 \to P_0 \to A \to 0$ where each $P_i$ is a finitely generated graded projective $A^e$-module.

Recall that a finite-dimensional semisimple $k$-algebra $B$ is \emph{separable} if $B \otimes_k L$ remains semisimple for all field extensions $k \subseteq L$.  In \cite{RR1} we show the following, based on a result due to Jeremy Rickard.

\begin{lemma}\label{lem:smooth separable}
(\cite[Theorem~3.10]{RR1})
Let $A$ be a locally finite $\mb{N}$-graded $k$-algebra with $S = A/J(A)$.  Then $A$ is homologically smooth over $k$ if and only if
$S$ is separable as a $k$-algebra and $S$ is a perfect left $A$-module.
\end{lemma}

\noindent
Now note that if $A$ is a locally finite elementary graded $k$-algebra, so that $S = A/J \cong k^{n}$, then
$S$ is obviously separable over $k$.  So in the elementary case, homological smoothness of $A$ over $k$ is equivalent to $S$ being perfect, or equivalently, to $Se_i$ being perfect for all $i$ where $1 = e_1 + \cdots + e_n$ is an idempotent decomposition.
We see next that this condition also guarantees that the matrix Hilbert series of an elementary algebra has a nice form.
\begin{proposition}
\label{prop:perfect-hs}
Let $A$ be a locally finite elementary $k$-algebra, with fixed decomposition $1 = e_1 + \cdots + e_n$ of $1$ as a sum of primitive pairwise orthogonal idempotents.   Assume that $A$ is homologically smooth over $k$.
  \begin{enumerate}
\item
The matrix Hilbert series of $A$ is of the form $h_A(t) = q(t)^{-1}$ for some matrix polynomial $q(t) \in M_n(\mb{Z}[t])$.   If
$D(t) = \det q(t) \in \mb{Z}[t]$, then $D(0) = \pm 1$.

\item  If $M$ is a perfect $\mb{Z}$-graded left $A$-module, then the vector Hilbert series $h_M(t)$ has the form $(r_1(t) D(t)^{-1}, r_2(t) D(t)^{-1}, \dots, r_n(t) D(t)^{-1})^T$ for some Laurent polynomials $r_i(t) \in \mb{Z}[t, t^{-1}]$.

\item If $M$ is a perfect graded left $A$-module with total Hilbert series $h(t) = h\tot_M(t)$ and with $\GKdim(M) = m < \infty$,
then $\eps(h(t))$ is positive and an integer multiple of $\eps(D(t))^{-1}$.

\item Let $0 \to M \to N \to P \to 0$ be a short exact sequence of perfect graded $A$-modules of finite GK-dimension, with graded homomorphisms.
Then
\[
\GKdim(N) = \max(\GKdim(M), \GKdim(P)).
\]
Moreover, if $\GKdim(M) = \GKdim(P)$ then the multiplicities satisfy $\eps(N) = \eps(M) + \eps(P)$.
\end{enumerate}
\end{proposition}

\begin{proof}
(1)  As noted above, homological smoothness guarantees that each simple module $Se_r$ is perfect.
The minimal graded projective resolution of the simple module $S e_r$ has the form $P_{\bullet} \to S e_r \to 0$,
where $P_i = \bigoplus_{j,m} [ Ae_j(-m)]^{n(Se_r, i, j, m)}$, in the notation of Proposition~\ref{prop:resolution form1}.
By assumption, the numbers $n(Se_r, i, j,m)$ are finite and are $0$ for $i \gg 0$.

Now each $Ae_j$ has a vector Hilbert series
$h_{Ae_j}(t) = (h_{e_1 A e_j}(t), \dots, h_{e_n A e_j}(t))^T$.  From the projective resolution of $Se_r$ we determine
that the $r$th standard basis vector $\mathbf{e}_r = (0, \dots, 1, \dots, 0)^T$ can be written as
\[
\mathbf{e}_r  = \sum_i (-1)^i \sum_{j,m} n(Se_r, i, j, m) t^m h_{Ae_j}(t) = \sum_j q_{jr}(t) h_{Ae_j}(t),
\]
where $q_{jr}(t) = \sum_{i, m} (-1)^i n(Se_r, i, j, m)t^m \in \mb{Z}[t]$.

Now letting $q(t) = (q_{jr}(t)) \in M_n(\mb{Z})$, and letting $h_A(t)$ be the matrix Hilbert series of $A$,
the $(m,r) $-entry of $h_A(t) q(t)$ is $\sum_{j=1}^n h_{e_m Ae_j}(t) q_{jr}(t)$, which as we saw above, is also equal to the $m$th coordinate of the $r$th standard basis vector, and so is equal to $\delta_{mr}$.  Thus $h_A(t) q(t) = I$ and $h_A(t) = q(t)^{-1}$ as claimed.

Since $h_A(t) q(t) = I$ holds in the matrix ring $M_n(\mb{Z}[[t]])$, taking determinants yields
$\det h_A(t) \det q(t) = 1$.  Since $\det h_A(t)$ and $\det q(t)$ are in $\mb{Z}[[t]]$, this forces both $\det h_A(t)$ and $\det q(t)$
to have constant term $\pm 1$.

(2) Let $P_{\bullet} \to M$ be a minimal graded projective resolution of $M$, and write $P_i = \sum_{j,m} [Ae_j(-m)]^{n(M, i,j,m)}$, where again the numbers $n(M,i,j,m)$ are finite since $M$ is perfect.

The same argument as in part (1) shows that the vector Hilbert series of $M$ satisfies
\[
h_M(t) = \sum_i (-1)^i \sum_{j,m} n(M,i,j,m) t^m h_{Ae_j}(t) = \sum_j p_j(t) h_{Ae_j}(t)
\]
where $p_j(t) = \sum_{i,m} (-1)^i n(M,i,j,m)t^m \in \mb{Z}[t,t^{-1}]$.

Now since $h_{Ae_j}(t)$ is equal to the $j$th column of the matrix Hilbert series $h_A(t) = q(t)^{-1}$, we
get
\[
h_M(t) = q(t)^{-1} \begin{pmatrix} p_1(t) \\ p_2(t) \\ \vdots \\ p_n(t) \end{pmatrix}.
\]
In particular, since $q(t)^{-1} = c(t) D(t)^{-1}$ where $c(t) \in \mb{Z}[t]$ is the adjugate matrix of $q(t)$,  we get that
the $j$th entry of $h_M(t)$ is $r_j(t) D(t)^{-1}$ where $r_j(t) = \sum_{i=1}^n c_{ji}(t) p_i(t) \in \mb{Z}[t, t^{-1}]$.

(3) The total Hilbert function of $M$ can be found by summing the terms of the vector Hilbert series calculated in part (2),
and thus it has the form $h\tot_M(t) = h(t) = r(t)D(t)^{-1}$.   By cancelling common factors in $\mb{Z}[t]$, this has an expression as $h(t) = f(t)g(t)^{-1}$ with $f,g \in \Z[t]$ relatively prime, and we still have $g(0) = \pm 1$.
By Lemma~\ref{lem:rational series}(1), since $\GKdim(M) < \infty$ all roots of $g(t)$ are roots of unity.  Then by Lemma~\ref{lem:rational series}(3), $\eps(h(t))$ is positive and an integer multiple of $\eps(g(t))^{-1}$.

Write $D(t) = g(t) s(t)$ where $s(t) \in \mb{Z}[t]$.  As we saw in the proof of Lemma~\ref{lem:rational series}, $\eps(s(t)) \in \mb{Z}$ since $s(t)$ has integer coefficients.  Then
$\eps(D(t)) = \eps(g(t)) \eps(s(t))$ and so $\eps(g(t))^{-1}$ is an integer multiple of $\eps(D(t))^{-1}$.  It follows that $\eps(h(t))$ is also an integer multiple of $\eps(D(t))^{-1}$ as required.

(4)  If $M$ is a perfect graded $A$-module, as already noted in part (3) its total Hilbert series $h(t)$ is rational and can
be written in lowest terms as $h(t) = f(t)g(t)^{-1}$ where $f(t) \in \mb{Z}[t, t^{-1}]$, $g(t) \in \mb{Z}[t]$, $g(0) = \pm 1$ and $f(t), g(t)$ have no common factors.  By Lemma~\ref{lem:rational series}(2) if $\GK(M) = m < \infty$, then $m$ is
the order of the pole of $h(t)$ at $t = 1$.  Then by Lemma~\ref{lem:rational series}(3),
\[
h(t) = \eps(M) (1-t)^{-m} + a_{-m+1} (1-t)^{-m+1} + \cdots,
\]
where the multiplicity of $M$ is $\eps(M) = \eps(h(t)) > 0$.   Now if $0 \to M \to N \to P \to 0$ is an exact sequence of perfect modules of finite GK-dimension with graded homomorphisms, then we get $h\tot_N(t) = h\tot_M(t) + h\tot_P(t)$.  Applying the observations above for each of $M$, $N$, and $P$,
it is easy to see that $\GK(N) = \max(\GKdim(M), \GKdim(P))$ as claimed.  Also, in case $\GK(M) = \GK(P) = m$
then $\GK(N) = m$ also and so $\eps(N) = \eps(M) + \eps(P)$ is immediate.
\end{proof}

To end this section we discuss how in most cases, to prove results about locally finite graded algebras $A$ that are not necessarily elementary, particularly regarding GK-dimension, one can often pass to the elementary case in a standard way.  The key result is the following theorem.
For the definition of graded Morita equivalence, we refer readers to~\cite[p.~496]{Sierra}.

\begin{theorem}
\label{thm:reduce-elem}
Let $A$ be a locally finite graded $k$-algebra, with $S = A/J(A)$.  Assume that $S$ is separable over $k$.
\begin{enumerate}
\item There is a finite degree separable field extension $k \subseteq L$ such that
$S \otimes_k L$ is isomorphic to a product of matrix rings over $L$, and $A' = A \otimes_k L$ is a locally finite $L$-algebra
with $J(A') = J(A) \otimes_k L$, $A'/J(A') \cong S \otimes_k L$,
and $\GKdim_L(A') = \GKdim_k(A)$.
\item There is a full idempotent $e \in A'_0$ such that $A'' = eA'e$ is graded Morita equivalent to $A'$ and $A''$ is an elementary locally finite graded $L$-algebra with  $\GKdim_L(A'') = \GKdim_L(A')$.
\end{enumerate}
\end{theorem}

\begin{proof}
(1) Because $S$ is a separable $k$-algebra, there exists a finite separable extension $L/k$ such that $S \otimes_k L$ is
isomorphic to a product of matrix rings over $L$; see~\cite[\S13, Th\'{e}or\`{e}me~1]{BourbakiVIII}.
By~\cite[Lemma 3.7]{RR1}, the algebra $A' = A \otimes L$ has $J(A') = J(A) \otimes_k L$ and
$A'/J(A') = S \otimes_k L$.  It is clear that $\GKdim_L(A') = \GKdim_k(A)$;  in fact, the total Hilbert
series of $A$ over $k$ is the same as the total Hilbert series of $A'$ over $L$.

(2)  Since $S' = A'/J(A')$ is a product of matrix rings over its base field $L$, it is standard that there is a full idempotent $e \in A'_0$ such that $A'' = e A' e$ is an elementary $L$-algebra \cite[Section I.6]{ASS}.  The rings $A''$ and $A'$ are Morita
equivalent since $e$ is a full idempotent.  Morita equivalence preserves GK-dimension, so that $\GKdim_L(A') = \GKdim_L(A'')$.
\end{proof}

We give an example of how the result above can be used to reduce to the elementary case in the proofs of some important basic properties.  Let us recall some standard notions in the theory of GK-dimension.  Let $A$ be a locally finite graded $k$-algebra.  We say that GK-dimension is \emph{graded exact} over $A$ if, given a short exact sequence $0 \to M \to N \to P \to 0$ in the category of finitely generated graded left $A$-modules, then $\GKdim(N) = \max(\GKdim(M), \GKdim(P))$.   We say that GK-dimension is \emph{graded finitely partitive} over $A$ if, given a finitely generated graded $A$-module $M$, there is a number $n_0 = n_0(M)$ such that if there exists a descending chain $M = M_0 \supsetneq M_1 \supsetneq \cdots \supsetneq M_n$ of graded submodules $M_i$ with $\GKdim(M_i/M_{i+1}) = \GKdim(M)$ for all $i$, then $n \leq n_0$.
\begin{proposition}
\label{prop:exact}
Let $A$ be a locally finite graded homologically smooth $k$-algebra.
\begin{enumerate}
\item Either $\GK(A) < \infty$ or else $A$ has exponential growth.
\item If $A$ is noetherian, then $\GK(A) < \infty$ and GK-dimension over $A$ is graded exact and graded finitely partitive.
\end{enumerate}
\end{proposition}
\begin{proof}
By Lemma~\ref{lem:smooth separable}, $S = A/J(A)$ is separable.  Thus we can use Theorem~\ref{thm:reduce-elem} to pass to an algebra $A''$ which is elementary.  By construction $A'' = e(A \otimes_k L)e$ for a full idempotent $e \in (A \otimes L)_0$, where $k \subseteq L$ is a finite extension of fields.   Note that if $A$ is noetherian, so is $A' = A \otimes_k L$, since it is a finitely generated $A$-module; then $A''$ is also noetherian since it is Morita equivalent to $A'$.  As noted in Theorem~\ref{thm:reduce-elem} already, $\GKdim_L(A'') = \GKdim_k (A)$.  In addition, since $A$ is homologically smooth, so is $A''$ \cite[Corollary 3.14, Proposition 3.17]{RR1}.

If $A$ is not graded finitely partitive, there is a graded $A$-module $M$ such that for each $n >0$ there is a descending chain
$M = M_0 \supsetneq M_1 \supsetneq \cdots \supsetneq M_n$ of graded submodules $M_i$ with $\GKdim(M_i/M_{i+1}) = \GKdim(M)$ for all $i$.  Since $- \otimes_k L$ and $e(-)$ are exact functors which preserve GK-dimension, we easily get
that $A''$ has a module $M'' = e(M \otimes_k L)$ with the same properties, so $A''$ is not graded finitely partitive.  A similar argument shows  that if $A$ is not graded exact, then $A''$ is not graded exact.

By the observations above, we see that it is enough to prove the proposition for the elementary algebra $A''$.
Changing notation back, we assume that $A$ is elementary for the rest of the proof.

(1)  By Proposition~\ref{prop:perfect-hs}(1) and Proposition~\ref{prop:growth}(3), either $\GKdim(A) < \infty$ or else $A$ has exponential growth.

(2) If $A$ has exponential growth, it is not noetherian by a result of Stephenson and Zhang \cite[Theorem 1.2]{SZ}.  Thus we
must have $\GKdim(A) < \infty$.  Since $A$ is noetherian and of finite global dimension, the perfect $A$-modules are the same as the finitely generated $A$-modules.  Also, each finitely generated $A$-module $M$ has $\GKdim(M) \leq \GKdim(A) < \infty$.  Thus Proposition~\ref{prop:perfect-hs}(4) shows that GK-dimension is exact on short exact sequences of finitely generated graded modules.   The graded finitely partitive property is an easy consequence of Proposition~\ref{prop:perfect-hs}(3) and (4).
\end{proof}

\section{The structure of graded invertible bimodules}
\label{sec:bimodule}

Let $A$ be a $k$-algebra.  Let $U$ be an $(A, A)$-bimodule.  When studying a bimodule over a $k$-algebra, we always assume the bimodule is $k$-central, that is $au = ua$ for all $u \in U$ and
$a \in k$.  Recall that an $(A, A)$-bimodule $U$ is \emph{invertible} if there is an $(A, A)$-bimodule $V$ with
\begin{equation}
\label{eq:inv}
U \otimes_A V \cong A \cong V \otimes_A U
\end{equation}
 as $(A, A)$-bimodules.  If $A$ is a graded $k$-algebra, a graded $(A, A)$-bimodule is \emph{graded invertible} if there is another graded $(A, A)$-bimodule $V$ such that \eqref{eq:inv} holds with graded isomorphisms.
 By~\cite[Lemma~2.10]{RR1}, a graded $(A,A)$-bimodule is graded invertible if and only if it is
invertible when considered as an ungraded bimodule.

Given a $k$-algebra endomorphism $\mu: A \to A$, we can define an $(A, A)$-bimodule $^1 A^{\mu}$, which is the same as $A$ as a left module, but with right $A$-action $a * x = a\mu(x)$.
Given another endomorphism $\tau: A \to A$, it is easy to check that $^1 A^{\mu} \otimes_A {} ^1 A^{\tau} \cong {} ^1 A^{\mu \tau}$ as $(A, A)$-bimodules.   In particular, if $\mu$ is an automorphism then $^1 A^{\mu}$ is invertible, with inverse $^1 A^{\mu^{-1}}$.  Similarly, if $A$ is graded and $\mu$ is a graded automorphism, then $^1 A^{\mu}$ is graded invertible.

For a connected graded $k$-algebra $A$, generated by degree one elements, it is well known that every graded invertible bimodule is graded isomorphic to $^1 A^{\mu}(\ell)$ for some graded automorphism $\mu$ and some shift $\ell$.  On the other hand, graded invertible bimodules over a general locally finite graded $k$-algebra, even an elementary algebra, can have a surprisingly complicated structure.   The details are given in the next result, which will be important for
our applications to Calabi-Yau algebras later.

\begin{proposition}
\label{prop:graded bimodule}
Let $A$ be a locally finite graded elementary $k$-algebra with graded Jacobson radical $J$ and semisimple part $S = A/J$.
Choose a fixed decomposition $1 = e_1 + \cdots + e_n$ of $1$ as a sum of primitive orthogonal idempotents $e_i \in A_0$.
Let $U$ be a graded invertible bimodule over $A$.
\begin{enumerate}
\item There exists a (not necessarily graded) automorphism $\mu: A \to A$ such that $U \cong {}^1 A^{\mu}$ as ungraded bimodules, and such that there is permutation of $\{1, \dots, n \}$, which we also write as $\mu$, with $\mu(e_i) = e_{\mu(i)}$ for all $i$.   The permutation $\mu$ is uniquely determined by $U$.

\item There is a unique sequence of integers $\underline{\ell} = (\ell_1, \dots, \ell_n)$ such that $U \cong \sum_{i=1}^n Ae_i(\ell_i)$ as graded left modules and $U \cong \sum_{i=1}^n e_iA(\ell_{\mu(i)})$ as graded right modules.

\item The restriction of $\mu$ to $e_i Ae_j$ yields an isomorphism $e_i A e_j \to e_{\mu(i)} A e_{\mu(j)}$ of graded vector spaces which is homogeneous of degree $\ell_{\mu(j)} - \ell_{\mu(i)}$.

\item $U \otimes_A Se_j \cong Se_{\mu(j)}(\ell_{\mu(j)})$, as graded left $A$-modules, for all $j$.

\item Suppose that $A_0 = S$ and that $A$ is generated as a $k$-algebra by elements of degree $0$ and $1$ (equivalently,
the underlying weighted quiver of $A$ has all arrows of weight $1$).   Suppose in addition that $A$ is indecomposable
as an algebra.  Then all of the $\ell_i$ are equal to a single $\ell$, the automorphism $\mu: A \to A$ is a graded automorphism; and $U \cong {}^1 A^{\mu}(\ell)$.
\end{enumerate}
\end{proposition}
\begin{proof}
(1) Since $U$ is graded invertible, $U$ is finitely generated projective as a left $A$-module.  By the remarks at the
beginning of Section~\ref{sec:elemquiv}, we have $U \cong \bigoplus_{i = 1}^m A e_{d_i}(\ell_i)$, as graded left modules, for some uniquely determined $1 \leq d_i \leq n$ (possibly with repeats) and shifts $\ell_i \in \mb{Z}$.   In addition, $U$ is invertible in the ungraded sense and so must be a generator of the category $A \lMod$ of left $A$-modules.  Ignoring the grading for the moment, then $A e_i$ is a surjective image of $U^{\oplus r}$ for some $r$; say $\phi \colon U^{\oplus r} \to Ae_i$ is a surjection, where we can think of $\phi$ as an element of  $\bigoplus_{j = 1}^r \bigoplus_{k = 1}^m \Hom_A(Ae_{d_k}, Ae_i)$.   For $h \neq i$, since
$Ae_h$ and $Ae_i$ are not isomorphic by the assumption that $A$ is elementary, then there can
be no surjective homomorphism $Ae_h \to Ae_i$, and thus any element in $\Hom_A(Ae_h, Ae_i)$ has image in the unique maximal left submodule $Je_i$ of $Ae_i$.    Since $\phi$ is surjective, it follows that $i = d_k$ for some $k$.  In other words,
every indecomposable projective occurs at least once as summand of $U$, so the $d_i$ include every integer from $1$ to $n$,
and thus $m \geq n$.  On the other hand, since $U$ is graded invertible, and hence $U \otimes_A -$ is a graded Morita equivalence, there is a graded isomorphism of rings $A^{\op} \cong \End_A(U)$.  If $f_i \in \End_A(U)$ is projection onto the $i$th summand of $U$ followed by inclusion,
then $1 = f_1 + \cdots + f_m$ and the $f_i$ are orthogonal idempotents in $\End_A(U)$.   Since $A^{\op} \cong \End_A(U)$, there must be at most $n$ idempotents in any decomposition of $1$ into a sum of orthogonal idempotents, and thus $m \leq n$.  This forces $m = n$.  Thus the $d_i$ include each integer from $1$ to $n$ exactly once, so that we can relabel now to assume that $U \cong \bigoplus_{i = 1}^n A e_i(\ell_i)$.

Now $U = \bigoplus_{i=1}^n U e_i$ as graded left modules.  Since we just saw that $U$ is a direct sum of $n$ indecomposable
graded projective left $A$-modules, each $Ue_i$ must be indecomposable.  As we remarked at the beginning of Section~\ref{sec:elemquiv}, the modules $Ae_i$ are the only graded indecomposable projectives, and so
there must be a permutation $\mu$ of $\{1, \dots, n \}$
and isomorphisms $\psi_j: U e_j \cong  A e_{\mu(j)} (\ell_{\mu(j)} )$, as graded left modules, for all $j$.
Let $x_j \in Ue_j$
be the homogeneous generator of this module such that $\psi_j(x_j) = e_{\mu(j)} \in A e_{\mu(j)} (\ell_{\mu(j)})$.
Note that $\deg(x_j) = -\ell_{\mu(j)}$.

Now we fix an (ungraded) isomorphism $\psi: U \to A$ of left modules which is given by following the graded isomorphism
$\bigoplus \psi_j: \bigoplus_{j=1}^n Ue_j \to \bigoplus_{j=1}^n A e_{\mu(j)} (\ell_{\mu(j)})$ with the natural identification
$\bigoplus_{j=1}^n A e_{\mu(j)} (\ell_{\mu(j)}) \to \bigoplus_{j=1}^n A e_{\mu(j)} $, which is a map of degree
$\ell_{\mu(j)}$ in the $j$th component.

To produce the automorphism $\mu$, we use a similar idea as in \cite[Lemma 2.9]{MM}.   Because $\psi: U \to A$ is an isomorphism of left modules,
there is an $(A, A)$-bimodule isomorphism $U \cong {}^1 A^{\mu}$ where $\mu: A \to A$ is the endomorphism defined by $\mu(a) = \psi(\psi^{-1}(1) * a)$, with $*$ indicating the right $A$-module action on $U$.  Note that since $\psi$ is not necessarily graded, $\mu$ is not necessarily a graded endomorphism.

Let $V$ be the inverse of $U$, so that $U \otimes_A V \cong A \cong V \otimes_A U$ as graded $(A, A)$-bimodules.  Then $V$ is also graded invertible and so the same arguments as above give $V \cong {}^1 A^{\tau}$ for some endomorphism $\tau: A \to A$.
Thus $A \cong {}^1 A^{\mu} \otimes_A {}^1 A^{\tau} \cong {}^1 A^{\mu\tau}$ as ungraded bimodules.  Consider some bimodule
isomorphism $\theta: A \to {}^1 A^{\mu \tau}$, and let $a = \theta(1)$.  Then $a$ generates ${}^1 A^{\mu \tau}$ as a left module
and so there is $b \in A$ such that $ba = 1$.  Similarly $a$ is a right generator and so there is $c \in A$ such that $a \mu(\tau(c)) = 1$.
Thus $a$ has a right and left inverse and so is a unit in $A$.  Now for all $b \in A$ we have $b = b(1) = (1)b$ and so applying $\theta$ we get $ba = a\mu(\tau(b))$.  This shows that $\mu\tau(b) = a^{-1} b a$ and we conclude that $\mu \tau$ is an inner automorphism of $A$.
A symmetric argument using that $A \cong V \otimes_A U$ shows that $\tau \mu$ is an inner automorphism.  In particular,
$\mu$ is an automorphism of $A$.

By the construction of $\psi$, we have $\psi^{-1}(1) = x_1 + \cdots + x_n$ with each $x_i$ a generator of $Ue_i$ as defined above.  Then $\mu(e_j) = \psi((\sum x_i) * e_j) = \psi(x_j) = e_{\mu(j)}$ by definition.   The uniqueness of the permutation
$\mu$ is clear from the isomorphisms  $U e_j \cong  A e_{\mu(j)} (\ell_{\mu(j)})$.

(2) We have already found integers $\ell_i$ such that $U \cong \bigoplus_{i=1}^n Ae_i(\ell_i)$ as graded left $A$-modules, and the uniqueness of the integers $\ell_i$ follows from the remarks at the beginning of Section~\ref{sec:elemquiv}.
An analogous argument on the other side gives $U \cong \bigoplus_{i=1}^n e_iA(\ell'_i)$ for some uniquely determined integers $\ell'_i$, as graded right $A$-modules.

Since we work with left modules by default in this paper, we use the notation $\Hom_{A^{\op}}(M, N)$ to indicate a Hom group between right modules $M_A$ and $N_A$.
Since $U$ is graded invertible, the natural
map $A \to \Hom_{A^{\op}}(U, U) = \End_{A^{\op}}(U)$ is a graded isomorphism (of rings, and of $(A, A)$-bimodules).  Thus we have
\begin{align*}
A \cong  \End_{A^{\op}}(U) &= \Hom_{A^{\op}}\left( \bigoplus_{j=1}^n e_jA(\ell'_j), U \right) \\
&= \bigoplus_{j=1}^n \Hom_{A^{\op}}(e_jA(\ell'_j), U) \cong \bigoplus_{j=1}^n Ue_j(-\ell'_j),
\end{align*}
as graded left $A$-modules.  In part (1) we already saw that $U e_j \cong  A e_{\mu(j)} (\ell_{\mu(j)} )$, so we conclude
that $A \cong \bigoplus_{j=1}^n A e_{\mu(j)} (\ell_{\mu(j)} - \ell'_j)$ as graded left $A$-modules.  This forces
$\ell'_j = \ell_{\mu(j)}$ as claimed.

(3)  Suppose that $a \in e_i A e_j$ is homogeneous with $\deg(a) = m$.  Since $a \in e_i A$,
$\mu(a) = \psi( (\sum_h x_h) * a) = \psi(x_i * a)$.  Now since $\deg(x_i) = -\ell_{\mu(i)}$, as noted in part (1),
$\deg(x_i * a) = m - \ell_{\mu(i)}$.  But since $a \in A e_j$, we also have $x_i * a \in U e_j$.  Since $\psi: U \to A$ is homogeneous
of degree $\ell_{\mu(j)}$ when restricted to $Ue_j$, $\deg(\psi(x_i *a)) = m + \ell_{\mu(j)} - \ell_{\mu(i)}$.  The result follows.

(4) We have $U \otimes_A S \cong S \otimes_A U$, as graded $(A, A)$-bimodules, by \cite[Lemma 2.11]{RR1}.
Now $U \otimes_A Se_j \cong U \otimes_A S \otimes_A Ae_j \cong S \otimes_A U \otimes_A Ae_j = S \otimes_A Ue_j$.  But
we have seen that $Ue_j \cong A e_{\mu(j)} (\ell_{\mu(j)} )$, and so
$S \otimes_A Ue_j \cong Se_{\mu(j)} (\ell_{\mu(j)} )$.

(5) Since $U \cong \sum_{i=1}^n Ae_i(\ell_i)$ as graded left $A$-modules, we have
\[
\End_A(U) = \Hom_A\left( \bigoplus_{i = 1}^m A e_i(\ell_i), \bigoplus_{j = 1}^m A e_j(\ell_j) \right)  
\cong \left(\bigoplus_{i,j} e_i A e_j  (\ell_j -\ell_i)\right)^{\op}
\]
 as graded algebras; we denote the algebra above as $B = \bigoplus_{i,j} e_i A e_j  (\ell_j -\ell_i)$.
As already noted, since $U$ is graded invertible the natural map $A^{\op} \to \End_A(U)$ is a graded isomorphism; thus $B \cong A$ as graded rings.
By Hypothesis, $J = J(A) = A_{\geq 1}$, so $A = A_0 \oplus J$ and $J = A_1 \oplus A_{\geq 2} = J^2$ as vector spaces; in other words, $A_0$ is
a vertex space for $A$ and $A_1$ is an arrow space, so the associated quiver $Q$ has all arrows of weight $1$.
Now the ring $B$ is clearly essentially the same as the ring $A$, with some elements possibly having shifted degrees.
Clearly $J(B) = \bigoplus_{i,j} e_i J e_j  (\ell_j -\ell_i)$ and it follows
that $X = \bigoplus_{i,j} e_i A_0 e_j  (\ell_j -\ell_i) = B_0$ is a vertex
space for $B$ and $V = \bigoplus_{i,j} e_i A_1 e_j  (\ell_j -\ell_i)$ is an arrow space for $B$.  Since $A$ and $B$ are graded isomorphic, they
have the same associated quiver, so we see that $e_i A_1 e_j(\ell_j -\ell_i)$ must be of degree $1$ as long as it is nonzero.
Thus if $Q$ is the associated quiver of $A$, then whenever there is an arrow $i \to j$ in $Q$ we must have $\ell_i = \ell_j$.  Since $A$ is indecomposable as an algebra, $Q$ is connected by Lemma~\ref{lem:quiver facts}.   It follows
that the numbers $\ell_i$ are all equal to a single number $\ell$.  By part (3), this implies that $\mu$ is a graded automorphism.  Since we now have $U \cong A(\ell)$ as graded left (or right) $A$-modules, we must also have $U \cong {}^1 A^{\mu}(\ell)$ as graded $(A, A)$-bimodules.
\end{proof}

It is easy to find examples where $A$ is decomposable as an algebra and a graded invertible bimodule over $A$ has shifts $\ell_i$ as in part (2) of the proposition which are not all equal.  For a very simple example, let $A = k \times k$ with $1 = e_1 + e_2$ and let $U = ke_1(\ell_1) \oplus ke_2(\ell_2)$ for any integers $\ell_1, \ell_2$.  Clearly $U$ is graded invertible with inverse $V = ke_1(-\ell_1) \oplus ke_2(-\ell_2)$.  In this example $\mu$ is trivial.

More interestingly, if the quiver associated to a locally finite elementary algebra $A$ does not have arrows of weight $1$, then even if $A$ is indecomposable, the shifts $\ell_i$ associated to a graded bimodule $U$ may be different and $\mu$ may not be a graded automorphism.
We will work out a detailed example in the next section (Example~\ref{ex:switch}).

\section{Growth of graded twisted Calabi-Yau algebras}
\label{sec:CY}

For the rest of the paper, we will specialize to graded twisted Calabi-Yau algebras that are locally finite and elementary.
So we first review the definition and some facts about such algebras.

\begin{definition}\label{def:graded twisted CY}
Let $A$ be an $\N$-graded $k$-algebra. We say that $A$ is \emph{graded twisted Calabi-Yau} of dimension $d$ if
\begin{enumerate}[label=(\roman*)]
\item $A$ is graded homologically smooth; and
\item there is a graded invertible $(A,A)$-bimodule $U$ such that
\[
\Ext^i_{A^e}(A,A^e) \cong
\begin{cases}
0, & i \neq d \\
U, & i = d
\end{cases}
\]
as graded right $A^e$-modules.
\end{enumerate}
This $U$ is called the \emph{Nakayama bimodule} of $A$. We say that $A$ is \emph{Calabi-Yau} if $U \cong A$.
\end{definition}

For a not necessarily graded algebra $A$, the twisted Calabi-Yau property is defined similarly, with all instances of the
word ``graded'' removed from the definition.  In \cite[Theorem 4.2]{RR1}, we show that an algebra which is $\mb{N}$-graded and twisted Calabi-Yau is automatically graded twisted Calabi-Yau in the sense of the above definition.  Thus in this paper, in which we are primarily interested in graded algebras, we will just refer to them as twisted Calabi-Yau rather than graded twisted Calabi-Yau.

Suppose that $A$ is a locally finite  graded twisted Calabi-Yau $k$-algebra, with graded Jacobson radical $J = J(A)$ and
semisimple factor $S = A/J = A_0/J(A_0)$. We know from Lemma~\ref{lem:smooth separable} that $S$ must
be separable. This allows us to reduce to the elementary case in the study of graded twisted Calabi-Yau algebras.

\begin{lemma}
\label{lem:cy-reduce-elem}
If $A$ is a locally finite twisted Calabi-Yau $k$-algebra, then there is a finite separable extension $k \subseteq L$ and a
full idempotent $e \in A' = A \otimes_k L$ such that $A'' = e A' e$ is an elementary graded twisted Calabi-Yau $L$-algebra, with
$\GKdim_k A = \GKdim_L A''$.
\end{lemma}

\begin{proof}
By Lemma~\ref{lem:smooth separable}, $S = A/J$ is separable over $k$.  Thus Theorem~\ref{thm:reduce-elem} applies.  The algebra
$A' = A \otimes_k L$ remains twisted Calabi-Yau over $L$ since this property is stable under base field extension \cite[Proposition 4.7]{RR1}. Similarly, the twisted Calabi-Yau property is stable under Morita equivalence  \cite[Proposition 4.11]{RR1}, so
$A''$ is still twisted Calabi-Yau as well.
\end{proof}
\noindent Since our main subject in this paper is the GK-dimension of locally finite twisted Calabi-Yau algebras, we will
focus on the elementary case from now on, as justified by the previous result.

In the companion paper \cite{RR1} we also show that for locally finite graded $k$-algebras, the twisted Calabi-Yau condition
can be characterized in various ways via homological conditions on one-sided modules rather than bimodules.  In the elementary case, the statement along these lines that we will use in this paper is the following.
\begin{theorem}
\label{thm:cy-vers-asreg}
Suppose that $A$ is a locally finite graded elementary $k$-algebra, where $1 = e_1 + e_2 + \cdots + e_n$ is an expression of $1$ as a sum of primitive orthogonal idempotents.  Let $S = A/J$.
Then the following are equivalent:
\begin{enumerate}
\item $A$ is twisted Calabi-Yau of dimension $d$.
\item $A$ is \emph{generalized AS regular} in the following sense:  The ring $A$ has graded global dimension $d$, and there is a
permutation $\pi$ of $\{1, \dots, n \}$ and integers $\ell_i$  such that for all $i$ there are isomorphisms of right $A$-modules
\[
\Ext^i_A(S e_i, A) \cong \begin{cases} e_{\pi(i)}S (\ell_i)  & i = d \\ 0 & i \neq d \end{cases}.
\]
\item $A$ has graded global dimension~$d$, and there is an invertible $(S,S)$-bimodule $V$ such that there are isomorphisms
of $(S,S)$-bimodules
\[
\Ext^i_A(S, A) \cong \begin{cases} V & i = d \\ 0 & i \neq d \end{cases}.
\]
\end{enumerate}
\end{theorem}

\begin{proof}
Because $A$ is elementary, $S$ is automatically separable over $k$.  Thus this follows from~\cite[Theorem 5.15]{RR1} and the characterizations of generalized AS~regular algebras given in \cite[Theorem 5.2(a,d$'$)]{RR1}.
\end{proof}

Let $A$ be a locally finite elementary twisted Calabi-Yau algebra, and write $1 = e_1 + \cdots + e_n$
where the $e_i$ are primitive and orthogonal.  The Nakayama bimodule $U$ of $A$ is
graded invertible, and hence its structure is given by Proposition~\ref{prop:graded bimodule}.
Thus $U \cong {}^1 A^{\mu}$, as ungraded bimodules, for some not necessarily graded automorphism
$\mu$ of $A$ we call a \emph{Nakayama automorphism} of $A$ (the automorphism $\mu$ is determined only up to composition with an inner automorphism of $A$).  With the choice of $\mu$ found in Proposition~\ref{prop:graded bimodule},
there is an associated action $\mu$ on the vertex labels $\{1, \dots, n \}$ such that $\mu(e_i) = e_{\mu(i)}$; in particular, $\mu$ restricts to an automorphism of the associated vertex space $S = \bigoplus ke_i$.  The vector $\underline{\ell}$ such that $U \cong \sum_{i=1}^n Ae_i(\ell_i)$ as graded left modules is called the (left) \emph{AS-index} or \emph{Gorenstein parameter} of $A$.  Example~\ref{ex:switch} below will illustrate a case where a vector of distinct integers is indeed required for the AS-index and where the Nakayama automorphism is not graded.

We note that a non-elementary finitely graded twisted Calabi-Yau algebra $A$ need not have a Nakayama automorphism at all;
that is, the Nakayama bimodule $U$ need not be of the form $^1 A^{\mu}$.  See \cite[Example 7.2]{RR1} for an explicit example.

Note that any locally finite graded algebra $A$ can be written as
\[
A \cong A_1 \times \cdots \times A_m
\]
where each $A_i$ is a graded indecomposable algebra (corresponding to a block decomposition~\cite[Section~22]{FC}
of $A_0$), and it was shown in~\cite[Proposition~4.6]{RR1} that $A$ is twisted Calabi-Yau of dimension~$d$ if and only if
each $A_i$ is twisted Calabi-Yau of dimension~$d$. In this way, the study of locally finite twisted Calabi-Yau algebras can
be reduced to the study of \emph{indecomposable} such algebras. In the case where $A$ and the $A_i$ are elementary,
Lemma~\ref{lem:quiver facts}(3) shows that we may restrict our attention to homomorphic images of path algebras of
weighted quivers that are connected.

We now prove our main results on the growth of graded twisted Calabi-Yau algebras and their modules, as an application of the general results of the previous sections and the structure theory of \cite{RR1}.

\begin{proposition}
\label{prop:cygk}
Let $A$ be a locally finite elementary twisted Calabi-Yau $k$-algebra of dimension $d$, with fixed decomposition $1 = e_1 + \cdots + e_n$ of $1$ as a sum of primitive pairwise orthogonal idempotents.   Let $S = A/J$.   Let $\mu$ be the Nakayama automorphism of $A$, and let $\underline{\ell} = (\ell_1, \dots, \ell_n)$ be the AS-index of $A$ and $L = \diag(\ell_1, \dots, \ell_n)$.  By Proposition~\ref{prop:perfect-hs}, the matrix Hilbert series of $A$ is of the form $h_A(t) = q(t)^{-1}$ for some matrix polynomial $q(t) \in M_n(\mb{Z}[t])$.
\begin{enumerate}
 \item  In the notation of Proposition~\ref{prop:resolution form1}, we have
\[
n(S e_r, i, j, m) = n(S e_{\mu(j)}, (d-i),  r , \ell_{\mu(j)}-m)
\]
for all $r, i,j,m$. As a consequence, $q(t)$ satisfies the functional equation
\[
q(t) = (-1)^d P t^L q(t^{-1})^T,
\]
where $P$ is the permutation matrix associated to the action of the Nakayama automorphism of $A$ on the vertices $\{1, 2, \dots n\}$ with $P_{ij} = \delta_{\mu(i) j}$ and $t^L = \operatorname{diag}(t^{\ell_1}, \dots, t^{\ell_n})$ as described in Section~\ref{sec:intro}.
\item $q(t)$ commutes with $Pt^L$.  As a consequence, for each $j$
there is an equality of vector Hilbert series
\[
t^{\ell_j} h_{Ae_{\mu^{-1}(j)}}(t) = P t^L h_{Ae_j}(t),
\] and
$\GKdim Ae_j = \GKdim A e_{\mu(j)}$.
\end{enumerate}
\end{proposition}

\begin{proof}
(1)
By \cite[Corollary~4.13]{RR1}, we have $\Ext^i_A(M,N)^* \cong \Ext^{d-i}_A(U \otimes_A N, M)$ as graded $k$-spaces, for all finite-dimensional modules $M$ and $N$, where $U$ is the Nakayama bimodule of $A$. Taking $M = Se_r$ and $N = Se_j$ and applying Proposition~\ref{prop:graded bimodule}(4),
we get
\[
\Ext^i_A(Se_r, Se_j)^* \cong  \Ext^{d-i}_A(Se_{\mu(j)}(\ell_{\mu(j)}), Se_r) \cong \Ext^{d-i}_A(Se_{\mu(j)}, Se_r)(-\ell_{\mu(j)}).
\]
Here $(-)^*$ refers to the graded dual of a locally finite $\mb{Z}$-graded vector space $V$, which satisfies $\dim_k (V^*)_m = \dim_k V_{-m}$,
and so
\begin{align*}
\dim_k \Ext^i_A(Se_r, Se_j)_{-m} &= \dim_k \Ext^{d-i}_A(Se_{\mu(j)}, Se_r)(-\ell_{\mu(j)})_m \\
&= \dim_k \Ext^{d-i}_A(Se_{\mu(j)}, Se_r)_{m-\ell_{\mu(j)}},
\end{align*}
implying by Proposition~\ref{prop:resolution form1} that $n(Se_r, i,j,m) =n(Se_{\mu(j)}, d-i, r, \ell_{\mu(j)}-m)$.

For the functional equation, we begin with the following calculation, where the third equality
below follows by the reindexing $i \mapsto d-i$, $m \mapsto \ell_{\mu(j)}-m$:
\begin{align*}
q_{jr}(t)  & = \sum_{i, m} (-1)^i n(Se_r, i, j, m)t^m \\
& = \sum_{i, m} (-1)^i n(Se_{\mu(j)}, d-i, r, \ell_{\mu(j)}-m)t^m \\
& = \sum_{i, m} (-1)^{d-i} n(Se_{\mu(j)}, i, r, m)t^{\ell_{\mu(j)}-m} \\
& = (-1)^d t^{\ell_{\mu(j)}} \sum_{i,m} (-1)^i n(Se_{\mu(j)}, i, r, m) t^{-m} \\
& = (-1)^d t^{\ell_{\mu(j)}} q_{r, \mu(j)}(t^{-1}).
\end{align*}
By definition, the permutation matrix $P$ has $P_{ij} = \delta_{\mu(i) j}$.  Thus $(P^{-1})_{ij} = (P^T)_{ij} = \delta_{i \mu(j)}$.
Then $q_{r, \mu(j)}(t^{-1}) = (q(t^{-1}) P^{-1})_{rj}$ so that
\[
t^{\ell_{\mu(j)}} q_{r, \mu(j)}(t^{-1})=(q(t^{-1}) P^{-1} t^{L'})_{rj},
\]
where $L' = \operatorname{diag}(\ell_{\mu(1)}, \dots, \ell_{\mu(n)})$.  Finally it is easy to see that
$ t^{L'} P= P t^L$, where $L = \operatorname{diag}(\ell_1, \dots, \ell_n)$.
Thus we get
\[
q(t) = (-1)^d [q(t^{-1}) P^{-1} t^{L'}]^T = (-1)^d t^{L'} P q(t^{-1})^T = (-1)^d P t^L q(t^{-1})^T
\]
as stated.

(2) By part (1), we have $q(t) = (-1)^d P t^L q(t^{-1})^T$.  Replacing $t$ by $t^{-1}$ in this equation yields
\[
q(t^{-1}) = (-1)^d P t^{-L} q(t)^T,
\] since $(t^{-1})^L= t^{-L}$.
Taking the transpose gives
\[
q(t^{-1})^T = (-1)^d q(t) t^{-L} P^{-1}.
\]
Substituting this expression in the original formula yields
\[
q(t) = (-1)^d P t^L (-1)^d q(t) t^{-L} P^{-1} = P t^L q(t) (P t^L)^{-1}.
\]
So $q(t)$ commutes with $P t^L$ as claimed.   Let $\mathbf{e}_1, \dots, \mathbf{e}_n$ be the standard basis vectors for $\mb{Z}^n$,
so that $q(t) \mathbf{e}_j$ is the $j$th column of $q(t)$, which is the vector Hilbert series $h_{Ae_j}(t)$.  Then applying both sides of the equality
$q(t) P t^L = P t^L q(t)$ to $\mathbf{e}_j$, we have
\[
q(t) P t^L \mathbf{e}_j = q(t) P t^{\ell_j} \mathbf{e}_j = t^{\ell_j} q(t) P \mathbf{e}_j = t^{\ell_j} q(t) \mathbf{e}_{\mu^{-1}(j)} = t^{\ell_j} h_{Ae_{\mu^{-1}(j)}}(t)
\]
and
\[
P t^L q(t) \mathbf{e}_j = P t^L h_{Ae_j}(t)
\]
from which the equality of vector Hilbert series is now clear.

The equality $\GKdim Ae_{\mu^{-1}(j)} = \GKdim Ae_j$ easily follows from this equality of Hilbert series, since the adjustment by $t^{\ell_j}$  or  $P t^L$ on the left does not change the overall growth.  Since this is true for all $j$, it is equivalent to
$\GKdim Ae_{\mu(j)} = \GKdim Ae_j$ for all $j$.
\end{proof}

To close this section, we will briefly describe the structure of those locally finite elementary twisted
Calabi-Yau algebras of dimension $d =0,1$. The structure of these algebras is relatively simple, and it
was discussed in~\cite{RR1} without requiring the algebras to be elementary. We specialize those
results to the case where $S \cong k^n$. The case $d = 0$ is particularly trivial.

\begin{remark}
The locally finite elementary twisted Calabi-Yau algebras $A$ of dimension $d = 0$ are those algebras
of the form $A = S \cong k^n$. Indeed, it was shown in~\cite[Theorem~4.19]{RR1} that such an algebra
is a finite-dimensional separable algebra, which implies that $A = S$, and is also Calabi-Yau.
But if $A$ is elementary, then $S \cong k^n$ for some integer $n \geq 1$. Thus the underlying quiver of
$A$ consists of $n$ vertices and no arrows.
\end{remark}

Now we describe the locally finite elementary algebras that are twisted Calabi-Yau of dimension~1 as
certain quiver algebras. Recall that we may reduce to the case of an indecomposable algebra as
described earlier in this section.

\begin{proposition}\label{prop:twisted CY 1}
Let $A$ be an indecomposable elementary locally finite graded algebra.
Then $A$ is twisted Calabi-Yau of dimension~1 if and only if $A \cong kQ$ for a weighted quiver
$Q$ such that $Q$ is a directed cycle with at least one arrow of positive weight.
If the vertices $\{1,\dots,n\}$ of $Q$ are indexed such that each arrow
is of the form $i-1 \to i \pmod{n}$ and has weight $\ell_i$, then this graded path algebra $kQ$ has
AS index $(\ell_1,\dots,\ell_n)$.  There is a choice of Nakayama automorphism $\mu$ of $kQ$ which
satisfies $\mu(e_i) = e_{i+1}$ and $\mu(\alpha_i) = \alpha_{i+1}$, where indices are computed modulo~$n$.
These algebras are noetherian and have GK-dimension~1.
\end{proposition}

\begin{proof}
It is shown in~\cite[Theorem~6.11]{RR1} that $A$ is twisted Calabi-Yau of dimension~1 if and only if
$A \cong T_S(V)$, where $S \cong A/J(A)$ is separable and $V$ is an invertible $(S,S)$-bimodule.
Because $A$ is elementary, this means that $S$ is a vertex space and $V$ is an arrow space for $A$,
so that $A \cong T_S(V) \cong kQ$ where $Q$ is the underlying quiver of $A$. Thus we may
identify $A = kQ = T_S(V)$ where $S = ke_1 \oplus \cdots \oplus ke_n$ is the canonical vertex space
of $kQ$ and $V$ is its canonical arrow space.

It remains to determine those conditions that make $V$ into an invertible $(S,S)$-bimodule.
First, suppose that $V$ is invertible, with $W = V^{-1}$. For each primitive idempotent $e_i \in S$,
we have $e_i V \otimes_S W \cong e_i S = ke_i$, which is a 1-dimensional vector space. But note that
\[
e_i V \otimes_S W = \bigoplus_j e_i V e_j \otimes e_j W.
\]
Since $W$ is invertible, $e_j W \neq 0$ for each $j$.
As the above vector space is 1-dimensional, it follows that for each $i$, there exists a unique $j$
such that $e_i V e_j \neq 0$ and that this $e_i V e_j$ is spanned by a single arrow from $i$ to $j$.
A similar argument beginning with the isomorphism $W \otimes_S V e_i \cong Se_i = ke_i$ shows
that each vertex is the target of a unique arrow.  Thus $Q$ is a union of finitely many directed cycles.
By Lemma~\ref{lem:quiver facts}, since $A$ is indecomposable the quiver $Q$ is connected and hence is a single directed cycle.
Thus, up to relabeling the vertices, we may assume that each vertex $i$ is the source of a single
arrow whose target is $i + 1 \pmod{n}$.  Lemma~\ref{lem:path algebra} shows that this path algebra is locally finite if and only if at least one arrow of $Q$ has positive weight.

Conversely, suppose that $A = kQ$ for a weighted quiver $Q$ which is a single directed cycle, with vertices labeled as
above so that it has a unique arrow $\alpha_i \colon i-1 \to i \pmod{n}$ of weight $\ell_i$, and where $\ell_i > 0$ for at least one $i$.  Notice that the $n$th tensor power $V^{\otimes n} = V \otimes_S \cdots \otimes_S V$, which corresponds in $T_S(V) \cong kQ$ to the paths of length~$n$, has basis consisting of a single cycle for each vertex $i$ that
begins and ends at $i$ and has total degree $\ell = \sum \ell_i > 0$. It follows that
$V^{\otimes n} \cong S(\ell)$, so that $V$ is invertible with graded inverse
$V^{-1} \cong V^{\otimes (n-1)}(-\ell)$.  Thus $A = kQ$ is twisted Calabi-Yau of dimension~1.
By \cite[Corollary~6.4]{RR1}, $A$ is noetherian.

Note that $k \alpha_i$ is a one-dimensional right $S$-module in degree $\ell_i$ that
is annihilated by $1-e_i$, while $e_i$ acts as the identity.
Thus $k \alpha_i \cong e_i S(-\ell_i)$,
so that $V \cong \bigoplus e_i S(-\ell_i)$ as graded right $S$-modules.
Now we must have $V^{-1} = \Hom_S(V_S, S_S) \cong \bigoplus Se_i(\ell_i)$ as graded left $S$-modules.
This forces $U \cong \bigoplus Ae_i(\ell_i)$ as graded left $A$-modules, and hence the left AS-index of $A$ is
$\underline{\ell} = (\ell_1, \dots, \ell_n)$.

For the rest of the proof, we reweight the quiver $Q$ so all arrows have weight $1$.  This is still a locally finite algebra
and does not change the path algebra $kQ$ up to ungraded isomorphism, and so will not affect the calculation of the 
Nakayama automorphism or the verification of the GK-dimension.  
Let $G = \langle g \rangle$ be a cyclic group of order $n$, and let $G$ act on a polynomial ring $k[x]$ by automorphisms where $g(x) = \zeta x$ for a primitive $n$th root of unity $\zeta$.  Then
$Q$ is the McKay quiver of this group action and so it is standard that $kQ$ is isomorphic to the skew group algebra $B = k[x] \# kG$ \cite[Corollary 4.1]{BSW}; under this isomorphism $e_i$ corresponds to $1 \# f_i$ and $\alpha_i$ corresponds to $x \# f_i$, where $1 = f_1 + \cdots + f_n$ is
the decomposition of $1$ as a sum of primitive idempotents in $kG$.  Explicitly, we may take $f_i = \sum_{j=1}^n (\zeta^{i-1})^j g^j$.
Now it is well known that $k[x]$ is Calabi-Yau of dimension $1$, so by \cite[Theorem 4.1]{RRZ1} the Nakayama automorphism
$\mu_B$ of $k[x] \# kG$ is equal to $1 \# \Xi^l_{\hdet}$, where $\Xi^l_{\hdet}$ is the left  winding automorphism of $kQ$ associated to the homological determinant.   In this case, since $k[x]$ is commutative, $\hdet(g)$ is
simply the determinant $\zeta$ of the action of $g$ on $kx$, and so by definition $\Xi^l_{\hdet}(g^i) =   \hdet(g^i) g^i = \zeta^i g^i$ for all $i$.  Now we see that $\mu_B(f_i) = f_{i+1}$ and $\mu_B(x \# f_i) = x \# f_{i+1}$.  Transferring this back to the algebra $A$ we get $\mu_A(e_i) = e_{i+1}$ and $\mu_A(\alpha_i) = \alpha_{i+1}$ as claimed.

Finally, in this unweighted case it is clear that the total Hilbert series of $kQ$ is $h\tot_{KQ}(t) = n/(1-t)$, since there are precisely $n$ paths in $Q$ of length $d$, for each $d$.  Then $\GKdim kQ = 1$ by Lemma~\ref{lem:rational series}.
\end{proof}

\begin{example}
\label{ex:switch}
Let $Q$ be the quiver with vertices $\{1, \dots, n\}$ and whose only arrows are $\alpha_i \colon i-1 \to i \pmod{n}$
where $\alpha_i$ has weight~$\ell_i$. Let $A = kQ$ and assume that at least one $\ell_i \neq 0$.
Then the path algebra $A = kQ$ is twisted Calabi-Yau of dimension $1$ by Proposition~\ref{prop:twisted CY 1}.
As shown in the proof above, the Nakayama automorphism $\mu$ satisfies $\mu(e_i) = e_{i+1}$ (with indices
computed modulo~$n$), so by part~(3) of Proposition~\ref{prop:graded bimodule}, the restriction
$\mu: e_{i-1} A e_i \to e_i A e_{i+1}$ is homogeneous of degree $\ell_{i+1}-\ell_i$.
Thus the Nakayama automorphism $\mu$ of $A$ is graded if and only if all of the $\ell_i$ are equal.
\end{example}

\section{Twisted CY algebras of global dimension $2$}
\label{sec:dim2}

In this section, we specialize to the case of twisted Calabi-Yau algebras with global dimension $2$.
There is much about twisted Calabi-Yau algebras of global dimension 2 that is known, but these
algebras have not typically been studied in the generality we consider here of elementary finitely
graded algebras, with no restrictions on the degrees of the generators.

There are several important precursors to the work below. In~\cite{Z}, Zhang investigated the
structure of AS~regular algebras of dimension~2, which are the locally finite twisted Calabi-Yau algebras
that are connected. In~\cite[Theorem~3.2]{Bo}, Bocklandt characterized the Calabi-Yau algebras
of the form $kQ/I$ (with $Q$ having all arrows of weight~1) as preprojective algebras on certain quivers.
A similar analysis was carried out for twisted Calabi-Yau algebras of the form $kQ/I$ (again, with
all arrows of $Q$ having weight~1) in~\cite[Section~6]{BSW}, although there was no
corresponding investigation into which quivers $Q$ occur.

Throughout this section, let $A$ be a locally finite graded elementary twisted Calabi-Yau algebra of dimension $2$.
Fix a vertex basis $\{e_1, \dots, e_n \}$ for $A$, so that $1 = e_1 + \cdots + e_n$  of $1$ as a sum of primitive orthogonal idempotents $e_i \in A_0$, and let $S = ke_1 + \cdots + ke_n$.   Similarly, fix an arrow basis
and thus an arrow space in $A$.  Let $Q$ be the underlying weighted quiver and fix the corresponding presentation
$kQ/I \cong A$ given by Lemma~\ref{lem:quiver facts}.  We write the trivial path at vertex $i$ in $kQ$ also as $e_i$.
We write the image of $x \in kQ$ in $A$ as $\overline{x}$.
We let $\mc{B}$ be the basis of arrows in $kQ$ and $V$ the $k$-span of $\mc{B}$,
so the images in $A$ of the arrows in $\mc{B}$ are the given fixed arrow basis.
Since $A$ is twisted Calabi-Yau, as described in the previous section we also have a (not necessarily graded) Nakayama automorphism $\mu: A \to A$, where $\mu$ also acts on the vertices $\{1, \dots, n \}$.  We also have the left AS-index $(\ell_1, \dots, \ell_n)$.

Let $Q$ be a quiver with vertices $\{1,\dots,n\}$ and let $\mu$ be an algebra automorphism of the path algebra $kQ$ that restricts to an automorphism of the canonical arrow space $S = \bigoplus ke_i$, so that $\mu(e_i) = e_{\mu(i)}$
for an associated automorphism of the vertices that we also denote by $\mu$.
Following~\cite[Section~2]{BSW}, an element $\omega \in kQ$ will be called a \emph{$\mu$-twisted weak potential} if
it satisfies the following conditions, which are readily seen to be equivalent:
\begin{itemize}
\item $\omega s = \mu(s) \omega$ for all $s \in S$;
\item $\omega e_i = e_{\mu^{-1}(i)} \omega$ for $i = 1,\dots,n$;
\item $\omega = \sum \omega_i$ for some $\omega_i \in e_{\mu^{-1}(i)} kQ e_i$.
\end{itemize}

\begin{proposition}
\label{prop:gldim2}
Let $A$ be a locally finite elementary graded twisted Calabi-Yau algebra of dimension~2, and
keep all of the notation above.
\begin{enumerate}
\item
There exists an injective $k$-linear map $\tau: V \to kQ$ such that $W = \tau(V)$  is also an arrow space for $kQ$;
$\tau( e_i  V_d e_r) \subseteq e_{\mu^{-1}(r)} (kQ)_{\ell_r-d} e_i$ for all $i,d,r$; and
where defining
\[
h_r = \sum_{x \in \mc{B} \cap kQe_r } \tau(x) x \in e_{\mu^{-1}(r)} (kQ)_{\ell_r} e_r
\]
for each $r$,
$\{h_1, \dots, h_n\}$ is an $S$-compatible minimal generating set for the ideal $I$ of relations for $A$.
Furthermore, $h = \sum h_r$ is a $\mu$-twisted weak potential that also generates $I$.

\item  The minimal graded projective resolution of the left simple $A$-module $S e_r$ has the form
\[
\xymatrix{
0 \to A e_{\mu^{-1}(r)}(-\ell_r) \ar[rrr]^-{( \overline{\tau(x_1)}, \dots, \overline{\tau(x_p)})}
&&& \displaystyle{\bigoplus_{i=1}^p Ae_{k_i}(-d_i)} \ar[r]^-{ \left(\begin{smallmatrix}\overline{x_1}\\ \vdots \\ \overline{x_p}\end{smallmatrix}\right)}
& Ae_r \to Se_r \to 0,
}
\]
where the $x_i$ range over $\mc{B} \cap kQe_r$ (that is, over the arrows whose target is $r$).

\item The matrix-valued Hilbert series of $A$ is given by
\[
h_A(t) = q(t)^{-1} \quad \text{for} \quad \ q(t) = I - N(t) + P t^L,
\]
where $P$ is the permutation matrix with $P_{ij} = \delta_{\mu(i) j}$, $N =N(t)$ is the weighted
incidence matrix of $Q$, and $L = \operatorname{diag}(\ell_1, \dots, \ell_n)$.
Furthermore, the equation
\[
N(t) = P t^L (N^T(t^{-1}))
\]
is satisfied, and the matrices $Pt^L, N$,  and $N^T(t^{-1})$ pairwise commute.
\end{enumerate}
\end{proposition}

\begin{proof}
By Lemma~\ref{lem:minres}(2) and the fact that $\gldim(A) = 2$, the minimal projective resolution of $Se_r$ has the form
\[
\xymatrix{
0 \to \displaystyle{\bigoplus_{i=1}^b Ae_{m_i}(-s_i)} \ar[r]^-{ \left(\overline{g_{ij}}\right)}
& \displaystyle{\bigoplus_{j=1}^p Ae_{k_j}(-d_j)} \ar[r]^-{ \left(\begin{smallmatrix}\overline{x_1}\\ \vdots \\ \overline{x_p}\end{smallmatrix}\right)}
& Ae_r \to Se_r \to 0
}
\]
where $d_i = \deg x_i$, the $k$-span $V_r$ of $x_1, \dots, x_p$ satisfies $V_r \oplus J^2(kQ)e_r = J(kQ)e_r$ as $k$-spaces, and where
if $g_i = \sum_j g_{ij} x_j$ and $G = \{g_1, \dots, g_b\}$, then there is a $S$-compatible minimal set of generators $H$ for the ideal
$I$ such that $G = H \cap kQe_r$.  If we write this resolution as $M_{\bullet}$ with
$M_i = \bigoplus_{j,m} [ Ae_j(-m)]^{n(Se_r, i, j, m)}$, then by Proposition~\ref{prop:cygk},
\[
n(S e_r, 2, j, m) = n(S e_{\mu(j)}, 0,  r , \ell_{\mu(j)}-m)
\]
 and so we get $M_2 = Ae_{\mu^{-1}(r)}(-\ell_r)$; in particular, $b = 1$.
We define a linear map $V_r \to kQ$ by $x_i \mapsto g_{1i}$ and extend this linearly to
obtain $\tau \colon V = \bigoplus_{r=1}^n V_r \to kQ$.
The minimal projective resolution of $Se_r$ now looks like
\[
\xymatrix{
0 \to A e_{\mu^{-1}(r)}(-\ell_r) \ar[rr]^-{( \overline{\tau(x_1)}, \dots, \overline{\tau(x_p)})}
&& \displaystyle{\bigoplus_{j=1}^p Ae_{k_j}(-d_j)}
	\ar[r]^-{\left(\begin{smallmatrix}\overline{x_1} \\ \vdots \\ \overline{x_p}\end{smallmatrix}\right)}
& Ae_r \to Se_r \to 0.
}
\]
We conclude that $h_r = \sum_{j=1}^p \tau(x_j) x_j \in e_{\mu^{-1}(r)} (kQ)_{\ell_r} e_r$
is the unique relation ending at vertex $r$ in a $S$-compatible minimal set of generators $\{ h_1, \dots, h_n \}$ for $I$.
Since $x_j$ is an arrow of degree $d_j$ from $k_j$ to $r$, clearly $\tau(x_j) \in  e_{\mu^{-1}(r)} (kQ)_{\ell_r - d_j} e_{k_j}$.
Thus $\tau$ satisfies $\tau(e_i V_d e_r) \in e_{\mu^{-1}(r)} kQ_{\ell_r - d} e_i$ for all $i, r, d$.
Note also that since each $h_r \in e_{\mu^{-1}(r)}kQe_r$, the element $h = \sum h_r$ is a
$\mu$-twisted weak potential by construction, where $\mu$ is the unique lift of $\mu_A$ to an automorphism
of $kQ$.  And because $h_r = he_r$ and $h = h1 = \sum he_r = \sum h_r$, we have $I = (h_1,\dots,h_n) = (h)$.

Since A is twisted Calabi-Yau, it is generalized AS regular in the sense of Theorem~\ref{thm:cy-vers-asreg}.
Applying $\Hom_A(-, A)$ to the deleted minimal projective resolution of $Se_r$ yields
\[
\xymatrix{
0  \leftarrow e_{\mu^{-1}(r)}A(\ell_r) && \ar[ll]_-{( \overline{\tau(x_1)}, \dots, \overline{\tau(x_p)})} \bigoplus_{i=1}^p e_{k_i}A(d_i) & \ar[l]_-{\left(\begin{smallmatrix}\overline{x_1} \\ \vdots \\ \overline{x_p}\end{smallmatrix}\right)} e_rA \leftarrow 0,
}
\]
where the free modules are now
column vectors and the maps are left multiplication by the indicated matrices.   By generalized AS~regularity, there must be a permutation $\pi$ of $\{1, \dots, n \}$ and integers $m_r$ such that this is a deleted minimal projective resolution of the right module $e_{\pi(r)}S(m_r)$ for all $r$.  Clearly this forces $\pi = \mu^{-1}$ and $m_r = \ell_r$, and shifting by $-\ell_r$ we get that
\[
\xymatrix{
0 \leftarrow  e_{\mu^{-1}(r)}S \leftarrow e_{\mu^{-1}(r)}A
&& \ar[ll]_-{( \overline{\tau(x_1)}, \dots, \overline{\tau(x_p)})} \displaystyle{\bigoplus_{i=1}^p e_{k_i}A(d_i-\ell_r)}
& \ar[l]_-{ \left(\begin{smallmatrix}\overline{x_1} \\ \vdots \\ \overline{x_p}\end{smallmatrix}\right)}\ar[l] e_rA(-\ell_r) \leftarrow 0
}
\]
is the minimal projective resolution of the right module $e_{\mu^{-1}(r)}S$.
By a right-sided version of Lemma~\ref{lem:minres}, this means that $W_r = \tau(V_r)$, the $k$-span of $\tau(x_1), \dots, \tau(x_p)$,
satisfies $e_{\mu^{-1}(r)} J^2(kQ) \oplus W_r = e_{\mu^{-1}(r)} J(kQ)$ as $k$-spaces.  Summing over $r$, we see that
$J^2(kQ) \oplus W = J(kQ)$ as $k$-spaces as required, where $W = \bigoplus_r W_r = \tau(V)$, and so $W$ is also an arrow space for $kQ$.  In addition, since the number of arrows in an arrow space is an invariant of $Q$, we must have
$\dim_k W = \dim_k V$, and this forces $\tau$ to be an injective linear map.  This finishes the proof of parts (1) and (2).

For (3), we saw in Proposition~\ref{prop:cygk} that the matrix Hilbert series of $A$ is $h_A(t) = q(t)^{-1}$, where
$q_{jr}(t) = \sum_{i, m} (-1)^i n(Se_r, i, j, m)t^m \in \mb{Z}[t]$.   In this global dimension $2$
case, by the form of the resolution calculated above we have $q(t) = H_0 - H_1 + H_2$ where $H_0 = I$, $H_2 = Pt^L$,
with $L = \operatorname{diag}(\ell_1, \dots, \ell_n)$.  We also see that $n(Se_r, 1, j, m)$ is the number of arrows in $Q$ of weight $m$ from $j$ to $r$  (or the dimension of $\Ext^1(Se_r, Se_j)_{-m}$, as we saw earlier in Lemma~\ref{lem:quiver facts}), and so
$H_1 = N(t)$, the weighted incidence matrix of $Q$.
The Hilbert series formula follows.

Now from the functional equation $q(t) =  (-1)^d P t^L q(t^{-1})^T$  given in Proposition~\ref{prop:cygk}, where
$d = 2$, we
get
\begin{align*}
q(t) &= P t^L(I - N(t^{-1}) + Pt^{-L})^T \\
&= Pt^L (I - N^T(t^{-1}) + t^{-L} P^{-1}) \\
&= Pt^L - Pt^L N^T(t^{-1}) + I,
\end{align*}
which implies that $N(t) = Pt^L N^T(t^{-1})$ as claimed.  Then $N(t^{-1})^T = (Pt^{-L}N^T)^T = N t^{-L}P^{-1}$ and so
$N = Pt^L N t^{-L}P^{-1} = (Pt^L)N(Pt^L)^{-1}$.  Thus $N$ commutes with $Pt^L$, and since $N = (Pt^L)N^T(t^{-1})$ we also get that $N$ commutes with $N^T(t^{-1})$.  Finally, since $N$ commutes with $Pt^L$,
taking the transpose and replacing $t$ by $t^{-1}$ we get that $N^T(t^{-1})$ commutes with
$(Pt^{-L})^T = t^{-L}P^{-1} = (Pt^L)^{-1}$, so $N^T(t^{-1})$ also commutes with $Pt^L$.
\end{proof}

\begin{example}
Consider $A = k \langle x, y \rangle/(yx -xy - x^{m+1})$ where $y$ has degree $m$ and $x$ has degree $1$.  The underlying weighted
quiver $Q$ of $A$ is one vertex with two loops of degree $1$ and $m$.  Using the graded lexicographic ordering with $x < y$,
the relation has leading term $yx$ which
does not overlap itself, and so by the diamond lemma $\{ x^i y^j \mid i, j \geq 0 \}$ is a $k$-basis for $A$, and $A$ has Hilbert
series $1/(1-t)(1-t^m) = 1/(1- (t + t^m) + t^{m+1})$.  It is well known that $A$ is AS regular, which is
equivalent to twisted Calabi-Yau in this connected graded case \cite[Lemma 1.2]{RRZ1}.  (One can also use Lemma~\ref{lem:cy2-criterion} below to verify that $A$ is twisted Calabi-Yau.)
In the notation of Proposition~\ref{prop:gldim2}, we have $\tau(x) =y - x^m$, $\tau(y) = -x$.  Thus $W = \tau(V)$ really can be a different arrow space for $kQ$ than $V = kx + ky$.
\end{example}

As a first consequence of the structure of twisted Calabi-Yau algebras of dimension~2, we obtain
the following information about the underlying quiver and GK-dimension of indecomposable projective
modules.

\begin{lemma}
\label{lem:strongconnected}
Let $A$ be a locally finite graded elementary twisted Calabi-Yau algebra of global dimension 2.
Let $1 = e_1 + \cdots + e_n$ be an expression of $1$ as a sum of primitive orthogonal idempotents.
Assume that $A$ is indecomposable, and let $Q$ be the underlying weighted quiver of $A$.
\begin{enumerate}
\item$Q$ is strongly connected; that is, given any two vertices $i$ and $j$ there is a directed path from $i$ to $j$.
\item Suppose that $\GK(A) < \infty$.   Then all indecomposable graded left projectives $A e_i$ satisfy $\GKdim (Ae_i) = \GKdim(A)$.
\end{enumerate}
\end{lemma}
\noindent Of course, in part (2) of the theorem one expects that $\GK(A) = 2$.  We will show this in Theorem~\ref{thm:GK2}
below, in the case where $Q$ has arrows of weight 1.
\begin{proof}
\noindent
(1) Since $A$ is indecomposable, the quiver $Q$ is connected by Lemma~\ref{lem:quiver facts}(3).
To begin the proof, we will show that if $Q$ has an arrow $i \to r$, then $Q$ also has a path from $r$ to $i$.
By Proposition~\ref{prop:gldim2}, the minimal graded projective resolution of the left simple $A$-module $S e_r$ has the form
\[
0 \to A e_{\mu^{-1}(r)}(-\ell_r) \overset{\delta_2}{\lra} M \overset{\delta_1}{\lra}  Ae_r \to Se_r \to 0,
\]
where $M$ is a direct sum of those projectives $Ae_i$ such that there is an arrow from $i$  to $r$ in $Q$.
It also follows from Proposition~\ref{prop:gldim2} that if there is an arrow $x \in e_i A e_r$, then there is a nonzero element
$\tau(x) \in e_{\mu^{-1}(r)} A e_i$ which is an element of some arrow space for $kQ$, so there is also an arrow in $Q$ from $\mu^{-1}(r)$ to $i$.  Thus there is some path
(of length 2) from $\mu^{-1}(r)$ to $r$ in $Q$.  Since $r$ is arbitrary, we get similarly a path from $\mu^{-n}(r) \to \mu^{-n+1}(r)$
for all $n \geq 1$.  Composing these we get a path from $\mu^{-n}(r)$ to $\mu^{-1}(r)$ for all $n \geq 1$, and following with the path
from $\mu^{-1}(r)$ to $i$ we get a path from $\mu^{-n}(r)$ to $i$.  Since $\mu$ is acting on a  finite set of vertices,
we have $\mu^{-n}(r) = r$ for some $n$ and thus there is a path from $r \to i$ as desired.

Now given any vertex $r$ of $Q$, let $Q'$
be the largest strongly connected subquiver of $Q$ containing $r$.  Suppose that $Q' \neq Q$.  Since $Q$ is connected, there is
a there is a vertex $h \in Q'$ and a vertex $j \in Q \setminus Q'$ such that $h$ and $j$ are joined by an arrow (in some direction).
Then there must be paths both from $h$ to $j$ and $j$ to $h$ and so clearly $Q' \cup \{ j \}$ is also strongly connected,
contradicting the choice of $Q'$.  It follows that $Q$ is strongly connected.

(2)  Consider the minimal projective resolution of $Se_r$ as in part (1).
Let $K = \coker \delta_2$.  Then there is an exact sequence
\[
0 \to A e_{\mu^{-1}(r)}(-\ell_r) \overset{\delta_2}{\lra} M \to K \to 0
\]
in which all terms are perfect, so
$\GKdim(M) = \max(\GKdim(A e_{\mu^{-1}(r)}), \GKdim K)$ by Proposition~\ref{prop:perfect-hs}(4).
Similarly, there is an exact sequence $0 \to K \to Ae_r \to Se_r \to 0$ where all terms are perfect;
note that $K \neq 0$ since every simple $Se_r$ must have projective dimension~$2$. Because
$\GKdim Se_r = 0$, we have $\GKdim Ae_r = \max(\GKdim K, 0) = \GKdim(K)$.
By Proposition~\ref{prop:cygk}(2), we have $\GKdim Ae_{\mu^{-1}(r)} = \GKdim Ae_r$.
It follows that
\[
\GKdim(M) = \max(\GKdim(A e_{\mu^{-1}(r)}), \GKdim K) = \GKdim Ae_r.
\]
Then $\GKdim Ae_i \leq \GKdim Ae_r$ for all $i$ such that there is an arrow from $i$ to $r$,
because these $Ae_i$ are direct summands of $M$ as described in the proof of part~(1).
By induction, $\GKdim Ae_i \leq \GKdim Ae_r$ for all $i$ such that there is a directed path from $i$ to $r$.
By part~(1), this is every vertex~$i$.  Letting $r$ vary we conclude that $\GKdim Ae_i = \GKdim Ae_j$
for all $i,j$.  Then since $A = \bigoplus_{i = 1}^n Ae_i$, we also have $\GKdim Ae_i = \GKdim A$ for all $i$.
\end{proof}

The next result is required for the companion paper~\cite{RR1}.  It is used in a
proof~\cite[Theorem~6.6]{RR1} that locally finite twisted Calabi-Yau algebras of dimension $2$ with finite
GK-dimension are automatically noetherian.

\begin{proposition}
\label{prop:forRR1}
Let $A$ be a graded locally finite twisted Calabi-Yau algebra of dimension $2$.  Suppose that $\GKdim(A) < \infty$.
Let $P_1 \overset{d_1}{\to} P_2 \overset{d_2}{\to} P_3$ be an exact sequence of graded projective $A$-modules.
If $P_1$ and $P_3$ are finitely generated, then so is $P_2$.
\end{proposition}
\begin{proof}
We reduce to the elementary case.  Suppose the result fails and consider an exact sequence
$P_1 \overset{d_1}{\to} P_2 \overset{d_2}{\to} P_3$ where $P_1$ and $P_3$ are finitely generated and $P_2$ is not.
Using Theorem~\ref{thm:reduce-elem}, we have an an elementary $L$-algebra $A'' = e(A \otimes_k L)e$ for
some finite extension of fields $k \subseteq L$.  Applying the exact functor from graded $A$-modules to graded $A''$-modules
given by $M \to e(M \otimes_k L)$, we produce an exact sequence of projective $A''$-modules
$P''_1 \to P''_2 \to P''_3$ where $P''_1$ and $P''_3$ are finitely generated and $P''_2$ is not.
Thus we can assume that $A$ is elementary from now on.

Now by Lemma~\ref{lem:strongconnected}, we get that writing $1 = e_1 + \cdots + e_n$ where the $e_i$ are primitive orthogonal idempotents, then $\GKdim(Ae_i) = \GKdim(A)$ for all $i$.  Suppose that $\{M_i\}_{i = 1}^{\infty}$ is a sequence
of nonzero finitely generated graded projective $A$-modules and $N$ is another finitely graded projective module such that
$\GKdim M_i = \GKdim N$ for all $i$.  We claim that there does not exist an injective graded $A$-module homomorphism
$\bigoplus_{i=1}^{\infty} M_i \to N$.  Each finitely generated graded projective module is trivially perfect.  If there is a graded embedding $\bigoplus_{i=1}^{\infty} M_i \to N$, then applying Proposition~\ref{prop:perfect-hs}(4), since all projectives in question have the same finite GK-dimension, we have in particular that $\eps(\bigoplus_{i=1}^{r} M_i) = \sum_{i=1}^r \eps(M_i) \leq \eps(N)$, for all $r$.
By Proposition~\ref{prop:perfect-hs}(3), any perfect graded $A$-module $M$ of finite GK-dimension
has multiplicity $\eps(M) > 0$ which is an integer multiple of some fixed number $\delta = \eps(D(t))^{-1}$ depending only on $A$.  This forces $r \leq \eps(N)/\delta$, contradicting the assumption that $r$ can be chosen arbitrarily large.  This proves the claim.

Finally, suppose that $P_2$ is infinitely generated, so it is a direct sum of infinitely many indecomposable graded projective $A$-modules. Since $P_1$ is finitely generated, $d_1(P_1)$ lies in finitely many of the summands of $P_2$.  Thus by exactness, $d_2$ gives an embedding of infinitely many summands of $P_2$ into $P_3$.  This contradicts the claim of the previous paragraph.
\end{proof}

Proposition~\ref{prop:gldim2} shows that a finitely graded elementary twisted Calabi-Yau algebra of global dimension 2 must have a very special form.   We can turn this around and ask which algebras of that special form are in fact twisted Calabi-Yau.

\begin{definition}
Let $Q$ be a finite weighted quiver with vertices $\{1, \dots, n \}$ and canonical arrow basis
$\mc{B}$ with
arrow space $V$ given by the span of $\mc{B}$ in $kQ$.  Suppose there is a permutation $\mu$ of $\{1, \dots, n \}$, an injective  $k$-linear graded map $\tau: V \to kQ$ such that $W = \tau(V)$ is again an arrow space for $kQ$, and a vector of integers
$(\ell_1, \dots, \ell_n )$ such that $\tau( e_i  V_d e_r) \subseteq e_{\mu^{-1}(r)} (kQ)_{\ell_r-d} e_i$ for all $i,d,r$.
Then the element $h = \sum_{x \in \mc{B}} \tau(x) x$ is a $\mu$-twisted weak potential, and we form the factor algebra
\[
A_2(Q, \tau) = kQ/(h) = kQ/(h_1, \dots, h_n)
\]
where the elements
\[
h_r = he_r = e_{\mu^{-1}(r)} h e_r = \sum_{x \in \mc{B} \cap kQe_r} \tau(x) x
\]
form a minimal $S$-compatible set of generators $\{ h_1, \dots, h_n \}$ for the ideal $(h)$.
Note that $\mu$ and the numbers $\ell_i$ are determined by $\tau$, so $A_2(Q, \tau)$ is determined just by $Q$ and $\tau$ as indicated by the notation.  In the case where all arrows of $Q$ have weight $1$ and $\tau$ is induced by a permutation of the arrows, a quiver $Q$ with such a map $\tau$ is known as a \emph{translation quiver}, and the relations $h_i$ are known as \emph{mesh relations}.
\end{definition}

In our general setting we will continue to refer to $A_2(Q, \tau)$ as the path algebra of a translation quiver with mesh relations.

Given a locally finite graded $k$-algebra $A$, we let $\lsoc(A)$ be the \emph{graded left socle} of $A$,  that is
$\{ x \in A \mid J(A) x \} = 0$, where $J(A)$ is the graded Jacobson radical as usual.  This is the same as the sum of all simple graded left ideals.  Similarly $\rsoc(A) = \{x \in A \mid xJ(A) = 0 \}$ denotes the graded right socle.
The following result can be useful in determining when the algebra $A_2(Q, \tau)$ is twisted Calabi-Yau.

\begin{lemma}
\label{lem:cy2-criterion}
Let $A = A_2(Q, \tau)$ as above with its natural grading.  Let $\mu$ be the associated permutation of $\{1, \dots, n \}$, with corresponding
permutation matrix $P$ such that $P_{ij} = \delta_{\mu(i)j}$, let $(\ell_1, \dots, \ell_n)$ be the associated
set of integers, and let $L = \operatorname{diag}(\ell_1, \dots, \ell_n)$.  Let $N = N(t)$ be the weighted incidence matrix of $Q$.  Then the following
are equivalent:
\begin{enumerate}
\item $A$ is twisted Calabi-Yau of dimension $2$.
\item  The matrix Hilbert series of $A$ satisfies
$h_A(t) = (I - N + Pt^L)^{-1}$.
\item $A$ has trivial graded right socle.
\item $A$ has trivial graded left socle.
\end{enumerate}
\end{lemma}
\begin{proof}
Keep the notation set up before the statement of the lemma.
By Lemma~\ref{lem:minres}, the minimal projective resolution of the
simple module $Se_r$ of $A$ begins with
\[
\xymatrix{
0 \to K_r \to A e_{\mu^{-1}(r)}(-\ell_r) \ar[r]^-{\delta_2}
& \displaystyle{\bigoplus_{i=1}^p Ae_{k_i}(- d_i)}
	\ar[rr]^-{ \delta_1 = \left(\begin{smallmatrix}\overline{x_1} \\ \vdots \\ \overline{x_p}\end{smallmatrix}\right)}
&& Ae_r \to Se_r \to 0,
}
\]
where $x_1, \dots, x_p$ is a basis of $\mc{B} \cap kQe_r$ and $K_r$ is the kernel of
the map 
\[
\delta_2 = ( \overline{\tau(x_1)}, \dots, \overline{\tau(x_p)}).  
\]
As usual, $\overline{x}$ indicates the image of $x \in kQ$ in $A = kQ/(h)$.

Following the argument in the proof of Proposition~\ref{prop:perfect-hs}(1),
for each $r$ and corresponding standard basis vector $\mathbf{e}_r = (0, \dots, 1, \dots, 0)^T$ we obtain
an equality of vector Hilbert series
\[
h_{K_r}(t) + \mathbf{e}_r = \sum_i (-1)^i \sum_{j,m} n(Se_r, i, j, m) t^m h_{Ae_j}(t) = \sum_j q_{jr}(t) h_{Ae_j}(t),
\]
where in this case $q(t) = I - N(t) + Pt^L$ by the calculation in  Proposition~\ref{prop:gldim2}.
 Thus if $E(t)$ is the matrix polynomial with $E(t)_{mr} = h_{e_m K_r}(t)$,
then putting these vector Hilbert series into a single matrix equation gives
 $h_A(t)(I - N(t) + Pt^L) = I + E(t)$.

If $A$ is twisted Calabi-Yau of dimension $2$, then every $K_r = 0$ and $h_A(t) = (I - N + Pt^L)^{-1}$.
Thus $(1) \implies (2)$ (as was already proved in Proposition~\ref{prop:gldim2}).
Conversely, if $h_A(t) = (I - N + Pt^L)^{-1}$, then $E(t) = 0$ and hence every $K_r = 0$,
so that every $Se_r$ has a minimal projective resolution of length $2$.  It follows that $A$ has graded global dimension $2$, by
\cite[Proposition 3.18]{RR1}.  Applying $\Hom_A(-, A)$ to the (deleted) projective resolution of $Se_r$ and shifting by $-\ell_r$ gives
a complex of right modules
\begin{equation}
\label{eq:afterhom}
\xymatrix{
 e_{\mu^{-1}(r)}A
& & \ar[ll]_-{( \overline{\tau(x_1)}, \dots, \overline{\tau(x_p)})} \displaystyle{\bigoplus_{i=1}^p e_{k_i}A(d_i -\ell_r)}
& \ar[l]_-{ \left(\begin{smallmatrix}\overline{x_1} \\ \vdots \\ \overline{x_p}\end{smallmatrix}\right)}
e_r A(-\ell_r) \leftarrow 0
}
\end{equation}
where the free modules are column vectors and the maps are now left multiplications.
Since the $x_1, \dots, x_p$ span $Ve_r$, where $V$ is the span of the canonical basis of arrows, and $\tau$ satisfies
$\tau( e_i  V e_r) \subseteq e_{\mu^{-1}(r)} (kQ) e_i$,  the span of $\tau(x_1), \dots, \tau(x_p)$ is clearly equal to $e_{\mu^{-1}(r)} W$, where $W = \tau(V)$. Because $W$ is an arrow space by assumption, we have
\[
e_{\mu^{-1}(r)}W \oplus e_{\mu^{-1}(r)}J^2(kQ) = e_{\mu^{-1}(r)}J(kQ),
\]
and thus $h_r = \sum_{i=1}^p \tau(x_i)x_i$ is only relation in the $S$-compatible minimal generating set of relations which begins at $\mu^{-1}(r)$, by the definition of $A_2(Q, \tau)$.
So by a right sided version of Lemma~\ref{lem:minres}(2), and the fact that we already know that $\gldim(A) = 2$,
\eqref{eq:afterhom} must be the (deleted) minimal projective resolution of
$e_{\mu^{-1}(r)}S$.  This shows that
\[
\Ext^i(Se_r, A) \cong \begin{cases} 0 & i \neq 2 \\ e_{\mu^{-1}(r)}S(\ell_r) & i = 2 \end{cases}
\]
holds for all $r$.  We have shown that $A$ is generalized AS regular, and
this implies that $A$ is twisted Calabi-Yau by Theorem~\ref{thm:cy-vers-asreg}.
Thus $(2) \implies (1)$ and we have shown that (1) and (2) are equivalent.

The arguments above showed that $A_2(Q, \tau)$ is skew Calabi-Yau of dimension $2$ if and only if $K_r = 0$ for all $r$.
Since $\overline{\tau(x_1)}, \dots, \overline{\tau(x_p)}$ span $e_{\mu^{-1}(r)}W$, ignoring shift we have
\begin{align*}
K_r &= \{ x \in Ae_{\mu^{-1}(r)} \mid  x e_{\mu^{-1}(r)}W = 0\} \\
&=\{ x \in A \mid x \in Ae_{\mu^{-1}(r)}\ \text{and}\ x W = 0\}.
\end{align*}
Since $W$ is an arrow space, we have $A = S + W + W^2 + \cdots$  where $J(A) = W + W^2  + \cdots$,
as we noted in Lemma~\ref{lem:quiver facts}.  Since $xW = 0$ if and only if $xW^i = 0$ for all $i$, we see that
\[
K_r = \{ x \in A \mid  x J(A) = 0 \} \cap Ae_{\mu^{-1}(r)} = \rsoc(A) \cap Ae_{\mu^{-1}(r)} = \rsoc(A) e_{\mu^{-1}(r)}.
\]
This implies that $\rsoc(A) \cong K_1 \oplus \cdots \oplus K_n$ as ungraded left $A$-modules, and hence
$\rsoc(A) = 0$ if and only if $K_r = 0$ for all $r$, if and only if $A$ is twisted Calabi-Yau of dimension $2$.
Thus (2) and (3) are equivalent.

Since $A$ is twisted Calabi-Yau if and only if $A^{\op}$ is (see the remark after \cite[Definition 4.1]{RR1}), (2) and (4) are also equivalent.
\end{proof}

Since any locally finite elementary graded twisted Calabi-Yau algebra of dimension~$2$
is isomorphic to an algebra with mesh relations $A_2(Q, \tau)$, we would like to know conversely exactly for which data $(Q, \tau)$ the algebra $A_2(Q, \tau)$ is twisted Calabi-Yau.   The answer to this question is known in many special cases and we plan to answer it fully in future work.

The final result of this section addresses the following problem:  assuming that $A_2(Q, \tau)$ is twisted Calabi-Yau of dimension~2, how does its GK-dimension depend on $(Q, \tau)$?  We will answer this question in the case where all of the arrows in $Q$ have weight $1$.  Note that
when $Q$ has all arrows of weight $1$, its weighted incidence matrix $N$ is of the form $N = Mt$ where $M$ is the usual
incidence matrix.  We will state our results in this section in terms of $M$ rather than $N$.

We need to use a bit of standard matrix analysis.  Recall that for a matrix $M \in M_n(\mb{C})$, its \emph{spectral radius} is
$\rho(M) = \max \{ |\lambda| \mid \lambda\ \text{is an eigenvalue of}\ M \}$.  For integer matrices we have
the Perron-Frobenius theory, from which we will use the following results.  First, an $n \times n$ matrix
$M$ is \emph{irreducible} as long as there is no subset $\emptyset \neq S \subsetneq \{1, \dots, n \}$ such that
$M_{ij} = 0$ for all $i \in S, j \not \in S$.  It is straightforward to see that the incidence matrix $M$ of a quiver $Q$ is irreducible if and only if $Q$ is
strongly connected, that is, every two vertices are joined by some directed path.
The Perron-Frobenius theorem states that if $M$ is an irreducible matrix with nonnegative real entries, then
(i) the spectral radius $\rho = \rho(M)$ of $M$ is an eigenvalue of $M$; (ii) The eigenvalue $\rho$ has a corresponding eigenvector
$v \in \mb{R}^n$ whose entries are strictly positive real numbers; (iii) The $\rho$-eigenspace is $1$-dimensional and hence spanned by $v$; and (iv) the only eigenvectors of $M$ whose entries are all positive are multiples of $v$.  See \cite[Theorem 8.3.4, Theorem 8.4.4]{HJ}.

\begin{theorem}
\label{thm:GK2}
Let $A = A_2(Q, \tau)$ be an algebra with mesh relations as defined in the previous section, where $Q$ has all arrows of weight $1$.  Assume that $A$ is indecomposable, and let $M$ be the incidence matrix of $Q$.
Suppose that $A$ is twisted Calabi-Yau of dimension $2$.  Then:
\begin{enumerate}
\item $\rho(M) \geq 2$.
\item If $\GKdim(A) < \infty$, then $\rho(M) = 2$ and $M$ has an eigenvector $v$ corresponding
to the eigenvalue $2$ such that every entry of $v$ is a positive integer.  Moreover, $\GKdim(A) = 2$ and $A$ is noetherian.
\item If $\GKdim(A) = \infty$ then $\rho(M) > 2$, $A$ has exponential growth and $A$ is not noetherian.
\end{enumerate}
\end{theorem}

\begin{proof}
By assumption, $\tau$ is a vector space map with $\tau(e_i V_d e_r) = e_{\mu^{-1}(r)} V_{\ell_r - d} e_i$ for all $i,d,r$,
where $V$ is the span of arrows in $kQ$, and where $W = \tau(V)$ is also an arrow space.  We are assuming that $V = V_1 = kQ_1$, and this forces $W = V$ and $\ell_r = 2$ for all $r$.  Now by Proposition~\ref{prop:gldim2}, $A$ has matrix Hilbert series $h_A(t) = q(t)^{-1}$, where $q(t) = (I - Mt + Pt^2)^{-1}$, with $P$ the permutation matrix associated to the action of the Nakayama automorphism $\mu$ on the vertices of $Q$.  Moreover, that result shows that $M = P M^T$, and that $M$, $P$ and $M^T$ pairwise commute.

By Proposition~\ref{prop:growth}, $A$ has finite GK-dimension if and only if all roots of the matrix polynomial $q(t)$ lie on the unit circle.  Since $M$ commutes with its transpose, it is a normal matrix, and the permutation matrix $P$ is also normal as it commutes with $P^T = P^{-1}$.  Commuting normal matrices can be simultaneously diagonalized by a unitary matrix, so there is a unitary $U \in \operatorname{GL}_n(\mb{C})$ (that is, $U^{-1} = \overline{U}^T)$
such that $D = U M U^{-1}$ and $Z = U P U^{-1}$ are diagonal.  Also $U M^T U^{-1} = (\overline{U}^{-1})^T M^T \overline{U}^T = (\overline{U} M \overline{U}^{-1})^T = \overline{D}$, and $M = P M^T$ implies that $D = Z \overline{D}$.
Write $D = \operatorname{diag}(\delta_1, \dots, \delta_n)$ and $Z = \operatorname{diag}(\zeta_1, \dots, \zeta_n)$.
Then $U (I - Mt + Pt^2) U^{-1} =  \operatorname{diag}(f_1(t), \dots, f_n(t))$
where $f_i(t) = 1 - \delta_i t + \zeta_i t^2$.  Conjugation by an invertible scalar matrix does not change the roots of a matrix polynomial, so $A$ will have finite GK-dimension if and only if  the matrix polynomial $\operatorname{diag}(f_1(t), \dots, f_n(t))$ has all of its roots
on the unit circle; in other words, if and only if each $f_i(t)$ has its roots on the unit circle.
Using a change of variable $u = \sqrt{\zeta_i} t$, since $\zeta_i$ and hence  $\sqrt{\zeta_i}$ are roots of unity,
this occurs if and only if $f_i(u) = 1 - au + u^2$ has all of its roots on the unit circle, where $a = \delta_i \overline{\sqrt{\zeta_i}}$.
In fact, $a$ is real:  the equation $D = Z \overline{D}$
implies $\delta_i = \zeta_i \overline{\delta_i}$ from which we get
$\delta_i^2 = \zeta_i |\delta_i|^2$, hence $\delta_i = \pm \sqrt{\zeta_i} |\delta_i|$ and thus $a = \pm |\delta_i|$.   Then it is easy to show that $f_i(u)$ has all of its roots on the unit circle if and only if $a \in [-2, 2]$:
\begin{itemize}
\item if $a > 2$, then $(a + \sqrt{a^2 -4})/2 > a/2 > 1$;
\item if $a < -2$ then $(a - \sqrt{a^2-4})/2 < a/2 < -1$; and
\item if $|a| \leq 2$ then $(a \pm \sqrt{a^2 -4})/2 = (a \pm i \sqrt{4 - a^2})/2$ has modulus-squared equal to $(a + i \sqrt{4 - a^2})(a - i \sqrt{4 - a^2})/4 = (a^2 + (4 - a^2))/4 = 1$.
\end{itemize}

We conclude that $A$ has finite GK-dimension if and only if  $|\delta_i| \leq 2$ for all $i$, in other words if and only if $\rho(M) \leq 2$.
When $\GKdim(A) = \infty$, by Proposition~\ref{prop:exact} $A$ has exponential growth and is not noetherian.  Thus (3) is proved.

Suppose now that $A$ has finite GK-dimension.  Then every perfect graded $A$-module $M$ has a multiplicity $\eps(M)$ which is positive and an integer multiple of $\eps(\det q(t))^{-1}$, by Proposition~\ref{prop:perfect-hs}(3).
Let $1 = e_1 + \cdots + e_n$ be the fixed idempotent decomposition,
set $d_i = \eps(Ae_i)/|\eps(\det q(t))^{-1}|$ which is a positive integer for all $i$,
and let $v$ be the column vector $(d_1, \dots, d_n)^T$.  By Proposition~\ref{prop:gldim2},
the minimal projective resolution of the $i$th graded simple $Se_i$ has the form
\begin{equation}
\label{exact}
0 \to  A e_{\mu^{-1}(i)}(-2) \to  \bigoplus_{x} A e_{s(x)}(-1) \overset{\delta_1}{\lra} A e_i \to Se_i \to 0,
\end{equation}
where $x$ ranges over those arrows in $Q$ with target $t(x) = i$, and $s(x)$ is the source of $x$.
By Lemma~\ref{lem:strongconnected}(2), for all $i$ we have $\GKdim Ae_i = \GKdim A = c$, say.
By Lemma~\ref{lem:rational series}, $c$ is the order of the pole at $t = 1$ of $h\tot_{Ae_i}(t)$, so
we can write its Laurent expansion in powers of $(1-t)$ as
\[
h\tot_{Ae_i}(t) = \eps(Ae_i) (1-t)^{-c} + a_{-c+1} (1-t)^{-c+1} + \cdots.
\]
In our setting where the weights of the arrows in $Q$ are all $1$, by Proposition~\ref{prop:cygk}(2)
we get $h_{Ae_{\mu(i)}}(t) = P h_{Ae_i}(t)$ as vector Hilbert series, and thus
$h\tot_{Ae_{\mu(i)}}(t) = h\tot_{Ae_i}(t)$.
Thus $d_i = d_{\mu(i)}$ for all $i$.  This shows that $Pv = v$.

Suppose that $c = 0$, so $A$ has GK-dimension $0$ and thus $\dim_k A < \infty$.  If $A = A_0 \cong k^n$ then $A$ does not have global dimension $2$, a contradiction.  Thus $A$ has elements of positive degree, and so $A$ has a nonzero socle, contradicting Lemma~\ref{lem:cy2-criterion}.  Thus  $c \geq 1$.  Then if $K = \ker(\delta_0)$, we see that $\eps(K) = \eps(Ae_i)$,
and by Proposition~\ref{prop:perfect-hs}(4),
\[
\eps(A e_{\mu^{-1}(i)}(-2)) + \eps(K) = \eps(\bigoplus_{x} A e_{s(x)}(-1)).
\]
Also, it is easy to check that for any perfect module $M$, $\eps(M) = \eps(M(m))$ for any shift $m$.
Thus we obtain the equation
\[
\sum_{\{x \mid t(x) = i\}} d_{s(x)} = d_{\mu^{-1}(i)} + d_i = 2d_i.
\]
Since there are $m_{ij}$ arrows from $i$ to $j$ in $Q$ where $M = (m_{ij})$, we have
$(M^T v)_i = \sum_j m_{ji} d_j  = 2d_i = 2v_i$ by the equation above,  and so $M^T v = 2v$.  Since $Pv = v$,  $Mv = PM^Tv = 2v$ as well.   This implies that in the diagonalized matrix polynomial
$U (q(t)) U^{-1} = \operatorname{diag}(f_1(t), \dots, f_n(t))$,
one of the polynomials $f_i(t)$ is equal to $1 - 2t + t^2 = (1-t)^2$.
By Proposition~\ref{prop:growth}, the GK-dimension of $A$ is the maximal order of $1$ as a pole of the rational
functions $(q(t)^{-1})_{ij}$.  Since $U$ is matrix of scalars, this is the same as the maximal order of $1$ as a pole of the rational functions
$(U (q(t)^{-1}) U^{-1})_{ij}$.  This is the same as the maximal order of $1$ as a pole of the functions $f_i(t)^{-1}$; since
the $f_i(t)$ are quadratic and we showed above that one of them vanishes twice at $1$, this maximal order is $2$.  Thus $c = \GKdim(A) = 2$.

We have now shown that if $\GKdim(A) < \infty$, then $\GKdim(A) = 2$ and in fact $M$ has an eigenvector $v$ whose
entries are all positive integers, corresponding to the eigenvalue $2$.  Since $v$ has all positive entries, it follows from the Perron-Frobenius theory that $\rho(M) = 2$.  (This also follows from the first part of the proof, which shows that $\rho(M) \leq 2$, so as soon as $2$ is an eigenvalue we have $\rho(M) = 2$).  It is shown in \cite[Theorem 6.6]{RR1} that locally finite
graded Calabi-Yau algebras of dimension $2$ with finite GK-dimension are noetherian.  Thus part (2) is proved.

Finally, part~(1) follows directly from both~(2) and~(3).
\end{proof}

\section{Twisted CY algebras of global dimension 3}
\label{sec:dim3}

In this last section, we attempt to extend the ideas of the previous sections to the case of twisted
Calabi-Yau algebras of dimension $3$.   Unsurprisingly, our results will be more partial.  In particular,
we will be able to give an explicit condition on the underlying quiver which is equivalent to the finite GK-dimension of
the algebra only in special cases.

First, we study the basic form of the generators and relations for such algebras.
Recall that it was shown in~\cite{Bo, BSW} that graded (twisted) Calabi-Yau algebras of dimension~$3$
that are generated in degree~1 are derivation-quotient algebras of (twisted) superpotentials.
In our general setting where generators and relations may occur in arbitrary degrees, it is not even clear
what the best definition of a ``twisted superpotential'' should be. However, we will show that a general
locally finite elementary twisted Calabi-Yau algebra of dimension~3 has a presentation that is similar to
a derivation-quotient algebra, with relations coming from an element that is similar to a twisted superpotential.

\begin{definition}
Let $Q$ be a weighted quiver, let $\mc{B} = \{x_j \}$ denote its standard arrow basis, and let $S = \bigoplus ke_i$
denote its standard vertex space.
An element $\omega \in kQ$ is said to be a \emph{twisted semipotential} if it satisfies the following properties:
\begin{itemize}
\item[(i)] there is a permutation $\mu$ of the vertices of $Q$ and integers $\ell_r$ such
that $\omega = \sum \omega_r$ with $\omega_r \in e_{\mu^{-1}(r)} kQ_{\ell_r} e_r$, or
equivalently, $\omega$ is a $\mu$-twisted weak potential for the corresponding automorphism $\mu$ of $S$;
\item[(ii)] there is another arrow basis $\mc{B}' = \{y_i \}$ such that $\omega = \sum_{i,j} y_i g_{ij} x_j$
for some homogeneous elements $g_{ij} \in kQ$;
\item[(iii)] the elements $h'_j = \sum_i y_i g_{ij}$ and $h_i = \sum_j g_{ij} x_j$ form two $S$-compatible minimal generating sets $\{ h'_j \}$ and $\{ h_i \}$ of the same ideal $I$.
\end{itemize}
We denote $A_3(Q, \omega) = kQ/I$, which is an elementary graded algebra.
\end{definition}

We begin by showing that every elementary locally finite graded twisted Calabi-Yau algebra of dimension~$3$ is of the form $A_3(Q, \omega)$
as above. The following proof closely follows the method of the corresponding result for dimension $2$ twisted Calabi-Yau algebras
in Proposition~\ref{prop:gldim2}.

\begin{proposition}
\label{prop:gldim3}
Let $A = kQ/I$ be a locally finite graded elementary twisted Calabi-Yau algebra of dimension $3$, and assume
all of the conventions and notation as in the paragraph before Proposition~\ref{prop:gldim2}.  In particular, let
$\mu$ be the action of the Nakayama automorphism of $A$ on the vertices of $Q$ and let
$\underline{\ell} = (\ell_1, \dots, \ell_n)$ be the AS-index. Let $\mc{B} = \{x_i \}$ be the standard arrow basis of $kQ$ with span $V$.
\begin{enumerate}
\item There exists another arrow basis $\mc{B}' = \{y_i \}$ with span $W$, and homogeneous elements $\{ g_{ij} \}$ in $kQ$
such that $\omega = \sum_{i,j} y_i g_{ij} x_j$ is a twisted semipotential and $A = A_3(Q,\omega)$.

\item The minimal graded projective resolution of the left simple module $S e_r$ has the form
\begin{align}
\begin{split} %single equation number
\label{eq:res3}
0 \to  Ae_{\mu^{-1}(r)}(-\ell_r) \overset{(\overline{y_i})}{\lra} \bigoplus_{i = 1}^b Ae_{m_i}(-s_i) 
&\overset{(\overline{g_{ij}})}{\lra} \bigoplus_{j = 1}^p Ae_{k_j}(-d_j) \cdots \\
\cdots &\overset{(\overline{x_i})^T}{\lra} Ae_r \to Se_r \to 0,
\end{split}
\end{align}
where $\{x_j \}$ ranges over $\mc{B}_r = \mc{B} \cap kQe_r$, $\{y_i \}$ ranges over $\mc{B}'_r = \mc{B}' \cap e_{\mu^{-1}(r)} kQ$, and $\sum_{i,j}  y_i g_{ij} x_j = \omega_r$.

\item The matrix-valued Hilbert series of $A$ is given by $h_A(t) = q(t)^{-1}$,
where
\[
q(t) = I - N(t) + Pt^L N(t^{-1})^T  - Pt^L
\]
with $N(t)$ the weighted incidence matrix of $Q$, $P$ the permutation matrix corresponding to $\mu$, and $L = \on{diag}(\ell_1, \dots, \ell_n)$.
Furthermore, $Pt^L$ commutes with $- N(t) + Pt^L N(t^{-1})^T$.
\end{enumerate}
\end{proposition}
\begin{proof}
(1),(2)  As usual, write the minimal projective resolution of $Se_r$ as $M_{\bullet}$ with
$M_i = \bigoplus_{j,m} [ Ae_j(-m)]^{n(Se_r, i, j, m)}$.  By Proposition~\ref{prop:cygk},
\[
n(S e_r, i, j, m) = n(S e_{\mu(j)}, 3-i,  r , \ell_{\mu(j)}-m)
\]
so that taking $i = 3$, $M_3 = Ae_{\mu^{-1}(r)}(-\ell_r)$.
Then by Lemma~\ref{lem:minres}(2), the (deleted) resolution $M_{\bullet}$ looks like
\[
0 \to  Ae_{\mu^{-1}(r)}(-\ell_r) \overset{(\overline{y_i})}{\lra} \bigoplus_{i = 1}^b Ae_{m_i}(-s_i) \overset{(\overline{g_{ij}})}{\lra} \bigoplus_{j = 1}^p Ae_{k_j}(-d_j) \overset{(\overline{x_j})^T}{\lra} Ae_r \to 0,
\]
for some homogeneous elements $\{ y_i \}$ and $\{ g_{ij} \}$ in $kQ$, and where $\{ x_1, \dots, x_p \}= \mc{B} \cap kQe_r$, where the elements $h_i = \sum_j  g_{ij} x_j$ are such that
$\{h_1, \dots,  h_b \} = G \cap kQe_r$ for some $S$-compatible minimal set of relations
$G$ generating $I$.
We have $d_j = \deg(x_j)$ and $s_i = \deg(h_i)$.
Applying $\Hom_A(-, A)$ and shifting by $-\ell_r$ gives
\begin{align}
\begin{split} %single equation number
\label{eq:gldim3}
0 \leftarrow  e_{\mu^{-1}(r)}A \overset{(\overline{y_i})}{\longleftarrow} \bigoplus_{i = 1}^s e_{m_i}A(-\ell_r+s_i) 
&\overset{(\overline{g_{ij}})}{\longleftarrow} \bigoplus_{j = 1}^p e_{k_j}A(-\ell_r+ d_j) \cdots \\
\cdots &\overset{(\overline{x_j})^T}{\longleftarrow} e_rA(-\ell_r) \leftarrow 0
\end{split}
\end{align}
where the free modules are now column vectors and the maps are left multiplication by the matrices.
Since $A$ is twisted Calabi-Yau, it is also generalized AS-regular by Theorem~\ref{thm:cy-vers-asreg}.  Similarly as
in the proof of Proposition~\ref{prop:gldim2}, this implies that \eqref{eq:gldim3} is a deleted minimal graded projective resolution of
the right module $e_{\mu^{-1}(r)}S$.  By the right sided version of Lemma~\ref{lem:minres}, setting $\mc{B}'_r = \{ y_1, \dots, y_b \}$, then $\mc{B}' = \bigcup_r \mc{B}'_r$ is also
an arrow basis for $kQ$, with $\mc{B}'_r = \mc{B}' \cap e_{\mu^{-1}(r)}kQ$.   Also, writing $h'_j = \sum_{i=1}^b y_i g_{ij}$, then there must be a $S$-compatible minimal set $G'$ of relations generating $I$ such that
$G' \cap e_{\mu^{-1}(r)}kQ = \{h'_1, \dots, h'_p \}$.  In particular, the number $p$ of relations beginning at $\mu^{-1}(r)$ in a minimal set of decomposed relations for $I$ is the same as the minimal number of generators for $A$ ending at $r$.   All parts of (1) and (2) are now clear.

(3)  From Proposition~\ref{prop:perfect-hs},  we may write $h_A(t) = q(t)^{-1}$ where
$q(t) = (q_{jr}(t))$ for $q_{jr}(t) = \sum_{i,m} (-1)^i n(Se_r, i,j,m) t^m$.
Setting 
\[
H_i = (h^i_{jr}) \quad \mbox{where} \quad h^i_{jr} = \sum_{m} (-1)^i n(Se_r, i, j, m)t^m,
\]
we have $q(t) = H_0 - H_1 + H_2 - H_3$.
By the similar arguments as in Proposition~\ref{prop:gldim2}, we have
$H_0 = I$, $H_1 = N(t)$ is the weighted incidence matrix of $Q$, and $H_3 = Pt^L$.
Now using the equations $n(S e_r, 2, j, m) = n(S e_{\mu(j)}, 1,  r , \ell_{\mu(j)}-m)$,
the same method as in Proposition~\ref{prop:cygk} shows that $H_2 = Pt^L H_1(t^{-1})^T = Pt^L N(t^{-1})^T$.
Thus $q(t) = I - N(t) + Pt^L N(t^{-1})^T - Pt^L$ as claimed.  The matrix $Pt^L$ always commutes with $q(t)$
by Proposition~\ref{prop:cygk}(2), so clearly this means that $Pt^L$ commutes with the
sum of the middle terms $- N(t) + Pt^LN(t^{-1})^T$.
\end{proof}

The previous proposition shows that if $A$ is finitely graded twisted Calabi-Yau of global dimension $3$ then $A \cong A_3(Q, \omega)$ for some weighted quiver $Q$ and twisted semipotential $\omega$.

Note that the twisted semipotential $\omega$ above determines the permutation $\mu$
and list of integers $(\ell_1, \dots, \ell_n)$.
In contrast to the case of dimension $2$, given a weighted quiver $Q$, one expects
$A_3(Q, \omega)$ to be twisted Calabi-Yau only for certain semipotentials $\omega$, and the problem
of determining which semipotentials is difficult.
Remarkably, it has been solved in the case of connected graded Calabi-Yau algebras of dimension~3
generated in degree~1 in~\cite{MoSm} and~\cite{MoU}.

We now give analogs of Lemma~\ref{lem:strongconnected} and Lemma~\ref{lem:cy2-criterion}, though the results in the global dimension $3$ case are weaker.

\begin{lemma}
\label{lem:strongconnected3}
Let $A$ be a locally finite graded elementary twisted Calabi-Yau algebra of dimension 3.
Assume that $A$ is indecomposable, and let $Q$ be the underlying weighted quiver.
Then $Q$ is strongly connected.
\end{lemma}
\begin{proof}
Fix a vertex $r$ and consider \eqref{eq:res3}.  Suppose that there is an arrow from $u$ to $r$
in $Q$, say the arrow $x_j$ in the fixed arrow space, so that $u = k_j$.  Then in the notation of Proposition~\ref{prop:gldim3},
$h'_j = \sum_i y_i g_{ij}$ is part of a minimal generating set of relations for $A$, and so it is a nonzero sum of paths from
$\mu^{-1}(r)$ to $u$.  In particular there is a path from $\mu^{-1}(r)$ to $u$, and composing with the original arrow we
get a path from $\mu^{-1}(r)$ to $r$ which goes through $u$.
The rest of the proof of Lemma~\ref{lem:strongconnected} now goes through verbatim.
\end{proof}

\begin{lemma}
\label{lem:cy3-criterion}
Let $A = A_3(Q, \omega)$ with its natural grading.  Let $\mu$ be the associated permutation of $\{1, \dots, n \}$, with corresponding
permutation matrix $P$ such that $P_{ij} = \delta_{\mu(i) j}$.  Let $(\ell_1, \dots, \ell_n)$ be the associated
set of integers, and let $L = \operatorname{diag}(\ell_1, \dots, \ell_n)$.  Let $N = N(t)$ be the weighted incidence matrix of $Q$.  Then the following
are equivalent:
\begin{enumerate}
\item $A$ is twisted Calabi-Yau of global dimension $3$;
\item  The matrix Hilbert series of $A$ satisfies
$h_A(t) = (I - N(t) + Pt^L N(t^{-1})^T  - Pt^L)^{-1}$, and $A$ has trivial
graded left and right Socle.
\end{enumerate}
\end{lemma}
\begin{proof}
By definition, there is another arrow basis $\mc{B}'$ such that $\omega  = \sum_{x_j \in \mc{B}} h_j x_j = \sum_{y_i \in \mc{B}'} y_i h'_i$
with both $\{h_j \}$ and $\{h'_i \}$ minimal sets of generators for the ideal of relations $I$ where $A = kQ/I$.
Then for each $r$ there a complex
\begin{align}
\begin{split}
\label{eq:gldim3res}
0  \to  Ae_{\mu^{-1}(r)}(-\ell_r) \overset{\delta_3 = (y_i)}{\lra} \bigoplus_{i = 1}^b Ae_{m_i}(-s_i) &\overset{(g_{ij})}{\lra} \bigoplus_{j = 1}^p Ae_{k_j}(-d_j) \cdots \\
\cdots &\overset{(x_j)^T}{\lra} Ae_r \to Se_r \to 0,
\end{split}
\end{align}
where $x_j$ runs over $\mc{B} \cap kQe_r$, $y_i$ runs over $\mc{B}' \cap e_{\mu^{-1}(r)}kQ$, and $\omega_r = \sum_{i,j} y_i g_{ij} x_j$.
By Lemma~\ref{lem:minres}, since the $h'_i = \sum_j g_{ij} x_j$ are the part ending at $r$ of a minimal generating set of relations,
the complex is exact except perhaps at the last two spots.  The same argument as in Lemma~\ref{lem:cy2-criterion}
shows that $\Ker \delta_3 = \rsoc(A) \cap Ae_{\mu^{-1}(r)}$ (up to shift).  If $A$ is twisted Calabi-Yau of global dimension~$3$, then Proposition~\ref{prop:gldim3} shows that $h_A(t) = (I - N(t) + Pt^L N(t^{-1})^T  - Pt^L)^{-1}$
and that \eqref{eq:gldim3res} is exact for all $t$, which forces $\rsoc(A) = 0$.  By symmetry we also have $\lsoc(A) = 0$, so
that $(1) \implies (2)$.

Conversely, if $(2)$ holds, then $\delta_3$ is injective in \eqref{eq:gldim3res} since $\rsoc(A) = 0$.  Let $H_r$ be the homology
at the $\bigoplus_{i = 1}^b Ae_{m_i}(-s_i)$ spot.   Let $E(t)$ be the matrix polynomial with $E(t)_{ij} = h_{e_i H_j}(t)$.  Then a similar argument as in Lemma~\ref{lem:cy2-criterion} shows that the matrix Hilbert series of $A$ satisfies
 $h_A(t)(I - N(t) + Pt^L N(t^{-1})^T  - Pt^L) = I + E(t)$.  Since we assume that $h_A(t) = (I - N(t) + Pt^L N(t^{-1})^T  - Pt^L)^{-1}$
holds, this forces $E(t) = 0$.  Thus the sequence \eqref{eq:gldim3res} is exact for each $t$.  Thus every $Se_r$ has a minimal projective resolution of length $3$, and it follows that $A$ has global dimension $3$, by \cite[Proposition 3.18]{RR1}.
Applying $\Hom_A(-, A)$ to the (deleted) projective resolution of $Se_r$ and shifting by $-\ell_r$ gives
a complex
\[
0 \leftarrow  e_{\mu^{-1}(r)}A \overset{(y_i)}{\longleftarrow} \bigoplus_{i = 1}^b e_{m_i}A(-\ell_r+s_i) \overset{(g_{ij})}{\longleftarrow} \bigoplus_{j = 1}^p e_{k_j}A(-\ell_r+ d_j) \overset{(x_j)^T}{\longleftarrow} e_rA(-\ell_r) \leftarrow 0
\]
where the free modules are column vectors and the maps are now left multiplications.  Using that the $\{ y_i \}$ are the span of
$\mc{B}' \cap e_{\mu^{-1}(r)}kQ$ for an arrow basis $\mc{B}'$, and that the $h_j = \sum_i y_i g_{ij}$ are the part of a minimal generating set of relations beginning at $\mu^{-1}(r)$, we see that
\begin{align*}
\begin{split}
0 \leftarrow e_{\mu^{-1}(r)}S \leftarrow  e_{\mu^{-1}(r)}A \overset{(y_i)}{\longleftarrow} 
 \bigoplus_{i = 1}^b e_{m_i}A(-\ell_r+s_i) \cdots \\
 \cdots \overset{(g_{ij})}{\longleftarrow}  \bigoplus_{j = 1}^p e_{k_j}A(-\ell_r+ d_j) \overset{(x_j)^T}{\longleftarrow} e_rA(-\ell_r) \leftarrow 0
\end{split}
\end{align*}
is exact by Lemma~\ref{lem:minres} except possibly at the final two terms; an analogous argument as above using that $\lsoc(A) = 0$ shows that the complex is exact.
This shows that
\[
\Ext^i(Se_r, A) \cong \begin{cases} 0 & i \neq 3 \\ e_{\mu^{-1}(r)}S(\ell_r) & i = 3 \end{cases}
\]
holds for all $r$.  By \cite[Theorem 5.2]{RR1}, this means that $A$ is generalized AS-regular, which is equivalent to the twisted Calabi-Yau condition in this elementary case by Theorem~\ref{thm:cy-vers-asreg}.
Thus $(2) \implies (1)$.
\end{proof}

To close, we turn our attention to results on the growth of twisted Calabi-Yau algebras of dimension~$3$.
The analysis is related to some interesting geometry of the complex plane.  Thus to start this section we review some preliminary results about the locations of the roots of certain special polynomials, using some results of Petersen and Sinclair \cite{PS}.

We first define some special regions of the complex plane.    Fix an integer $k \geq 3$ and define a function $h_k: \mb{R} \to \mb{C}$
by
\begin{align*}
h_k(\theta) &= [(k-1) \cos (\theta) + \cos((k-1)\theta)] + [(k-1) \sin(\theta) - \sin((k-1) \theta)] i \\
&= (k-1)e^{i \theta} + e^{-i (k-1)  \theta}.
\end{align*}
Let $H_k = h_k(\mb{R})$ be the image of $h_k$.  Then $H_k$ is a closed curve called a \emph{hypocycloid}. It can be formed by rolling a circle of radius 1 along the inside of a circle of radius $k$, and taking the curve traced out by a fixed point on the moving circle.  It is simple closed curve with $k$ cusps and $k$-fold rotational symmetry.  When $k = 3$, the curve is also known as a \emph{deltoid}, and when $k = 4$, an \emph{astroid}.  Let $\Omega_k$ be the closed region of $\mb{C}$ consisting of $H_k$, together with all points in the simply connected region inside of it.  See \cite[Figure 1]{PS} for a picture of $\Omega_3$.   The region $\Omega_k$
is clearly a star-domain in the sense that it is the union of all rays joining the origin to points of $H_k$.

To justify the claim that $\Omega_k$ has $k$-fold rotational symmetry, notice for all $\theta$ that
\begin{align*}
h(\theta + 2\pi/k) &= (k-1)e^{i \theta}e^{2 \pi i/k} + e^{-i (k-1) \theta}e^{-i2(k-1)\pi/k} \\
&= (k-1)e^{i \theta}e^{2 \pi i/k} + e^{-i (k-1) \theta}e^{2\pi i/k} \\
&= h(\theta) e^{2 \pi i/k}.
\end{align*}
Thus the image $H_k = h_k(\mb{R})$ of $h_k$ is invariant under rotation about the origin by $2 \pi/k$ radians, so the same is true of the star domain $\Omega_k$ which it bounds.

The results of \cite{PS} are not quite in the form we require. The next proposition extracts the information we need.
Given a real number $r > 0$, we write $r \Omega_k = \{ rz \mid z \in \Omega_k \}$ for the
dilation of the region $\Omega_k$ by $r$.

\begin{remark}\label{rem:boundary}
In the following proof, we will need to verify that two regions of the complex plane are equal using the following method.  Consider two closed subsets $X \subseteq Y \subseteq \C$ such that $\partial X = \partial Y$ and the interior of $Y$ is path-connected.
(Note that this hypothesis is satisfied for the regions $Y = \Omega_k$ above that are bounded by hypocycloids.)
If $X$ contains a point $p$ in the interior of $Y$, then $X = Y$. Indeed, because $X$ contains $\partial X = \partial Y$, to show that $X = Y$
it suffices to let $q \in \on{int}(Y)$ and deduce that $q \in X$. Assume toward a contradiction that $q \notin X$.
Let $f \colon [0,1] \to \on{int}(Y)$ be a path with $f(0) = p$ and $f(1) = q$. Then for $t = \inf\{a \in [0,1] \mid f(a) \notin X\}$,
one may show that $f(t) \in \partial X = \partial Y$. But this contradicts the fact that $f(t) \in \on{int}(Y)$.
\end{remark}

\begin{proposition}
\label{prop:hypo}
Let $a \in \C$.
\begin{enumerate}
\item If $f(x) = 1 - ax + \ol{a} x^2 - x^3$, then all of the roots of $f$ in $\mb{C}$ are on the unit circle if and only if $a \in \Omega_3$.  In particular if $a \in \mb{R}$, this happens if and only if
$a \in \Omega_3 \cap \mb{R} = [-1, 3]$.

\item If $f(x) = 1 - ax + \ol{a} x^3 - x^4$, then all roots of $f$ in $\mb{C}$ are on the unit
circle if and only if $a \in (1/2)\Omega_4$.    In particular if $a \in \mb{R}$, this happens if and only if
$a \in (1/2)\Omega_4 \cap \mb{R} = [-2, 2]$.
\end{enumerate}
\end{proposition}

\begin{proof}
(1) A polynomial $h(x) = \sum_{i = 0}^n a_i x^i \in \mb{C}[x]$ of degree $n$ is called \emph{conjugate reciprocal} if
$h(x) = x^n \ol{h(1/\ol{x})}$, or equivalently if $a_i = \ol{a_{n-i}}$ for all $i$.  Although $f(x)$ is not conjugate reciprocal, it is after a certain change of variable.  Let $x = \zeta z$, where $\zeta$ is a primitive $6$th root of $1$.  Then
$f(x) = f_a(x) = 1 - a \zeta z + \ol{a} \zeta^2 z^2 - \zeta^3 z^3 = 1 + b z + \ol{b} z^2 + z^3$, where
$b = - a \zeta = a \zeta^4$.  Thus if $g(z) = g_b(z) = 1 + b z + \ol{b} z^2 + z^3$, then $g$ is conjugate reciprocal and
has all roots on the unit
circle if and only if $f$ does.  Let $X_3 = \{ b \in \mb{C} \mid g_b(z)\ \text{has all its roots on the unit circle} \}$.
Then $f_a(x)$ has all of its roots on the unit circle if and only if $a \zeta^4 \in X_3$, or $a \in \zeta^2 X_3$.
We will show that $X_3 = \Omega_3$.  Then since $\Omega_3$ is invariant under rotation by an angle of $2\pi/3$ radians, this
will show that $f_a(x)$ has all of its roots on the unit circle if and only if $a \in \zeta^2 \Omega_3 = \Omega_3$, as desired.

In \cite{PS}, the authors write $b = \rho w_1 + \rho w_2 i$ where $\rho = \sqrt{2}/2$ is a normalization factor and $w_1, w_2 \in \mb{R}$.   Then
\[
g_b(z) = g_{w_1, w_2}(z) = 1 + bz + \ol{b} z^2 + z^3 = 1 + (\rho w_1 + \rho w_2 i) z + (\rho w_1 - \rho w_2 i)z^2 + z^3.
\]
The set $W_3 = \{ (w_1, w_2) \in \mb{R}^2 \mid g_{w_1, w_2}\ \text{has all its roots on the unit circle} \}$ is certainly bounded, since the coefficients of $g_{w_1, w_2}$ are polynomial functions of its roots.    In fact $W_3$ is closed and path connected, and its
boundary is given by those $(w_1, w_2) \in W_3$ such that $g_{w_1, w_2}(z)$ has a
multiple root \cite[Proposition 1.6]{PS}.  Identifying $\mb{R}^2$ with $\mb{C}$, this implies that $X_3$ is closed and path-connected and its boundary is given by those $b \in \mb{C}$ such that $g_b(z)$ has all of its roots on the unit circle and a multiple root.   This happens if and only if
\[
1 + bz + \ol{b} z^2 + z^3 = (z-r)(z-r)(z + \ol{r^2})
\]
for some $r$ on the unit circle, and so $b = -2\ol{r} + r^2$.  Setting $r = e^{i \theta}$, this gives
$b = e^{2i \theta} - 2 e^{-i \theta}$.  Reparametrizing by setting $\theta' = -\theta + \pi$, we have
$b = e^{-2i \theta'} + 2 e^{i \theta'}$, which traces out the deltoid $H_3$ as $\theta'$ varies.
Thus $X_3$ is a closed bounded set with boundary $H_3$; also notice that $0 \in X_3$,
so $X_3$ contains at least one point in
the interior of $\Omega_3$.  It follows as in Remark~\ref{rem:boundary} that $X_3 = \Omega_3$,
as claimed.  Finally, it is easy to see that $\Omega \cap \mb{R} = [1, 3]$.

(2) This is a similar argument as in (1).  We perform a change of variable to the polynomial
$f_a(x) = 1 - ax + \ol{a} x^3 - x^4$ as follows; letting $\zeta$ be a primitive $8$th root of unity
and setting $x = \zeta z$, $f_a(x)$ becomes $g_b(z) = 1 + bz + \ol{b} z^3 + z^4$,
where $b = - a \zeta$.  Let $X_4 = \{ b \in \mb{C} \mid g_b(z)\ \text{has all its roots on the unit circle} \}$.
Then $f_a(x)$ has all of its roots on the unit circle if and only if $- a \zeta \in X_4$, or $a \in \zeta^3 X_4$.
We will show that $X_4 = (\zeta/2) \Omega_4$.  Then since $\Omega_4$ is invariant under rotation by an angle of $\pi/2$ radians,
this will show that $f_a(x)$ has all of its roots on the unit circle if and only if $a \in (\zeta^4/2) \Omega_4 = -(1/2)\Omega_4 = (1/2)\Omega_4$, as desired.

Petersen and Sinclair study conjugate reciprocal
polynomials of degree $4$ of the form
\[
d_{w_1, w_2, w_3}(z) = 1 + (\rho w_1 + \rho w_2 i) z + w_3 z^2 + (\rho w_1  - \rho w_2 i) z^3 + z^4
\]
where $(w_1, w_2, w_3) \in \mb{R}^3$ and again $\rho = \sqrt{2}/2$.  Similarly as in part (1), the set
\[ W_4 = \{(w_1, w_2, w_3) \in \mb{R}^3 \mid d_{w_1, w_2, w_3}(z)\ \text{has all of its roots on the unit circle} \}
\]
is bounded, and it is a closed path connected region with boundary those points of $W_4$ where $d(z)$ has a multiple root \cite[Proposition 1.6]{PS}.
Let $W'_4$ be the intersection of $W_4$ with the plane $w_3 = 0$, which we identify with $\mb{R}^2$ with coordinates
$(w_1, w_2)$ and thus with $\mb{C}$ via $(w_1, w_2) \mapsto w_1 + w_2 i$.   Then clearly $X_4 = \rho W'_4$.
Note that a point $b$ on the boundary of $W'_4 \subseteq \C$ must also be a boundary point of $W_4$, and so corresponds
to a polynomial $g_b(z)$ with a multiple root.    We calculate the set of $b$ such that $g_b(z)$ has all of its roots
on the unit circle and a multiple root.  In this case $g_b(z) = (z- r)^2(z -s)(z - \ol{r^2 s})$ for some $r,s$ on the unit circle,
where $b = -2\overline{r} - \overline{s} - r^2s$, and where the coefficient of $z^2$ is $2rs + r^2 + 2\ol{rs} + \ol{r^2} = 0$, or
equivalently $2rs + r^2 \in \mb{R}i$.  Then $s = (\alpha i \ol{r} - r)/2$ for some $\alpha \in \mb{R}$, and so
\[
b = -2\overline{r} - \overline{s} - r^2s = -2\ol{r} +(\alpha i r)/2 + \ol{r}/2  - (\alpha i r)/2 + r^3/2
= -3 \ol{r}/2 + r^3/2.
\]
On the other hand, the curve $H_4$ is traced out by $3r' + (\ol{r'})^3$ as $r'$ varies over the unit circle.  Thus
$(\zeta/2) H_4$ is traced out by
\[
(3/2)(\zeta r') + (1/2)(\zeta^3 \ol{r'})^3 = (3/2)(\zeta r') - (1/2)(\zeta^{-1} \ol{r'})^3 =
(3/2)(\zeta r') - (1/2)(\ol{\zeta r'})^3.
\]
It is now clear via the substitution $\ol{r} = - \zeta r'$ that as $r$ varies over the unit circle, $b = -3 \ol{r}/2 + r^3/2$ traces out $(\zeta/2) H_4$.  Now $(\zeta/2) \Omega_4$ is equal to the simple closed curve $(\zeta/2) H_4$ together with its interior.  We have seen that $X_4$ is a closed bounded region and that any point on the boundary of $X_4$ must be a point of $(\zeta/2) H_4$.  Since $0 \in X_4$,
Remark~\ref{rem:boundary} forces $X_4 = (\zeta/2) \Omega_4$, as we needed.  Again, it is easy to see that $\Omega_4 \cap \mb{R} = [-4, 4]$
and so $(1/2) \Omega_4 \cap \mb{R} = [-2, 2]$.
\end{proof}

Now suppose that $A$ is a locally finite graded elementary twisted Calabi-Yau algebra of global dimension $3$.  By Proposition~\ref{prop:gldim3}, $A \cong A_3(Q, \omega) = kQ/I$ for some weighted quiver $Q$ and twisted semipotential $\omega$.  We assume now that $A$ is indecomposable as an algebra, and that all arrows in $Q$ have weight one.  Thus by Proposition~\ref{prop:graded bimodule}, the Artin-Schelter index has $\ell_i = \ell$ for all $i$, for some fixed $\ell$.  Thus $\omega$ is homogeneous of degree $\ell$
and $I$ is generated by relations of degree $\ell -1$.  Let $P$ be the permutation matrix associated to the action of the Nakayama automorphism of $A$ on the vertices.   In this case the weighted incidence matrix $N$ of $Q$ has the form $N = Mt$ where $M$ is the usual incidence matrix of the quiver $Q$.  By Proposition~\ref{prop:gldim3},
\[
h_A(t) = (I + Mt + PM^Tt^{\ell-1} + Pt^{\ell})^{-1},
\]
and $P$ commutes with  $Mt + PM^Tt^{\ell-1}$.  Note that $\ell \geq 3$, since $\ell = \deg \omega$ where $\omega$ is of the form $\sum_{i,j} y_i g_{ij} x_j$.  Thus $P$ commutes separately with $Mt$
and $PM^Tt^{\ell-1}$, and it follows that $P$ commutes with both $M$ and with $M^T$.   However, there is no obvious reason
that $M$ must be a \emph{normal} matrix, that is, that $M$ commutes with $M^T$.

We would like to have a criterion for when $\GK(A) < \infty$ based on the eigenvalues of $M$, as we found in the dimension $2$ case
(where the criterion was $\rho(M) = 2$).  In the next result we achieve this in case $M$ is normal.   We are not sure how to characterize the GK-dimension of $A$ in terms of the matrix $M$ in the non-normal case.    The assumption in the next result that $\GKdim(A) \geq 3$ should not be necessary, but we don't see how to remove it.  In any case,  it is  conjectured that a locally finite graded twisted Calabi-Yau algebra of dimension $d$ with finite GK-dimension should have GK-dimension $d$.
\begin{proposition}
\label{prop:diagonalize}
Let $A$ be an indecomposable locally finite elementary graded twisted Calabi-Yau algebra of dimension $3$ whose underlying quiver $Q$ has arrows of weight $1$.  Retain the notation above, and assume that the incidence matrix $M$ of $Q$ is normal, so that $M$, $M^T$, and $P$
are pairwise commuting normal matrices.  Then $M$, $M^T$, and $P$ are simultaneously unitarily diagonalizable, say $UMU^{-1} = D, U M^T U^{-1} = \overline{D}, U P U^{-1} = Z$
for a unitary matrix $U$.  Let $D = \operatorname{diag}(\delta_1, \dots, \delta_n)$ and $Z = \operatorname{diag}(\zeta_1, \dots, \zeta_n)$.
Suppose that $\GKdim(A) \geq 3$. Then
\begin{enumerate}
\item If $\ell = 3$, $\GKdim(A) < \infty$ if and only if $\delta_i \in  (\sqrt[3]{\zeta_i} )\Omega_3$ for all $i$, where $\Omega_3$ is the deltoid described above; if this occurs, then $\GK(A) = 3$ and $\rho(M) = 3$.
\item If $\ell = 4$, then $\GKdim(A) < \infty$ if and only if $\delta_i  \in  (1/2)(\sqrt[4]{\zeta_i})\Omega_4$ for all $i$, where $\Omega_4$ is the astroid described above; if this occurs,
then $\GK(A) = 3$ and $\rho(M) = 2$.
\item If $\ell \geq 5$, then $\GKdim(A) = \infty$.
 \end{enumerate}
\end{proposition}
\begin{proof}
By Proposition~\ref{prop:gldim3}, $A$ has matrix-valued Hilbert series $h_A(t) = q(t)^{-1}$ where
$q(t) = I - Mt + PM^T t^{\ell-1} - Pt^\ell$.
Conjugating a matrix polynomial by a unitary matrix $U \in M_r(\mb{C})$ does not change its set of roots.  Thus in all parts we may consider a
diagonalized matrix polynomial $U q(t) U^{-1} = 1 - Dt + Z \overline{D}  t^{\ell-1} - Z t^{\ell} = \operatorname{diag}(f_1(t), \dots, f_n(t))$
where $f_i(t) = 1 - \delta_i t + \overline{\delta_i} \zeta_i t^{\ell-1} - \zeta_i t^{\ell}$.
By Proposition~\ref{prop:growth}(3), $\GKdim(A) < \infty$ if and only if all of the roots of $q(t)$
lie on the unit circle, which occurs if and only if each $f_i(t)$ has all roots on the unit circle.
If this holds, then using Proposition~\ref{prop:growth}(2), $\GKdim(A)$ is the
maximal order of the pole at $t =1$ among the entries of $q(t)^{-1}$.
This is the same as the maximal order of the pole at $t = 1$ among
the entries of $U q(t)^{-1} U^{-1}$, which is the maximal order of the pole at $t = 1$ among the $f_i(t)^{-1}$, or equivalently
the maximal order of vanishing at $t = 1$ among the $f_i(t)$.

For any fixed $\ell$th roots $\sqrt[\ell]{\zeta_i}$, performing the change of variable
$u =\sqrt[\ell]{\zeta_i} t$ to $f_i(t)$ yields $g_i(u) = 1 - au + \overline{a} u^{\ell-1} - u^{\ell}$
where $a =  \delta_i \overline{\sqrt[\ell]{\zeta_i}}$.   Then $f_i(t)$ has all roots on the unit circle if and only if $g_i(u)$ has
the same property.  It now follows from Proposition~\ref{prop:hypo} that
\begin{itemize}
\item if $\ell = 3$, then $g_i(u)$ has all of its roots on the unit circle if and only if
$a \in \Omega_3$, if and only if $\delta_i \in (\sqrt[3]{\zeta_i}) \Omega_3$;
\item if $\ell = 4$, then $g_i(u)$ has all of its roots on the unit circle if and only if $a \in (1/2) \Omega_4$, if and only if $\delta_i \in (\sqrt[4]{\zeta_i}/2) \Omega_4$.
\end{itemize}
(Note that the particular choice of primitive $\ell$th roots of $\zeta_i$ is irrelevant due to the
$\ell$-fold rotational symmetry of $\Omega_{\ell}$.)

Now suppose that $\GKdim(A) < \infty$, so that by Proposition~\ref{prop:growth}, all roots of $q(t)$ must indeed lie on the unit circle.
Since we assume $3 \leq \GK(A)$, some $f_i(t)$ has $1$ as a root to multiplicity at least $3$.  Then
\[
f''_i(t) = \overline{\delta_i} \zeta_i (\ell-1)(\ell-2) t^{\ell-3} - \ell (\ell-1) \zeta_i t^{\ell-2}
\]
also vanishes at $t = 1$, so that $\delta_i  = \overline{\delta_i} = \ell/(\ell-2)$.
But $\delta_i$ is an algebraic integer, since it is an eigenvalue of an integer matrix.
If $\ell \geq 5$, then $\ell/(\ell-2)$ is a rational number that is not an integer, and hence is not
an algebraic integer, a contradiction.  This proves (3).

Now, since $A$ is indecomposable, the underlying quiver $Q$ is strongly connected, by Lemma~\ref{lem:strongconnected3}.
Thus the incidence matrix $M$ of $Q$ is irreducible and the Perron-Frobenius theory applies.
Let $\rho = \rho(M)$ be its spectral radius; then $\rho$ is a positive real eigenvalue of $M$, say $M v= \rho v$, where $v$ is an eigenvector with all positive real entries, and where $v$ spans the $\rho$-eigenspace of $M$.  Since $P$ and $M$ commute, $v$ is also an eigenvector of $P$; since the entries of $v$ are positive real numbers and the eigenvalues of $P$ are roots of $1$, the only possible corresponding eigenvalue is $1$, so that $P v = v$.  Thus one of the polynomials, say $f_1(t)$, is equal to $1 - \rho t + \rho t^{\ell-1} - t^{\ell}$.

Suppose now that $\ell = 3$ and $3 \leq \GK(A) < \infty$.  As we saw above, one of the eigenvalues of $M$ is $\ell/(\ell-2) = 3$.
Also, every eigenvalue $\delta_i$ of $M$ satisfies $\delta_i  \in  (\sqrt[3]{\zeta_i} )\Omega_3$.  Since the maximum value of the
norm of an element in $\Omega_3$ is $3$, $\rho(M) = 3$.  Also $\GKdim(A)$ is the maximum order of vanishing at $t = 1$ among
the polynomials $f_i(t)$; since $3 \leq \GKdim(A)$ and the polynomials have degree $3$, $\GKdim(A) = 3$.
In fact $f_1(t) = 1-3t + 3t^2 - t^3 = (1-t)^3$.  Thus (1) is proved.

Finally, suppose that $\ell = 4$ and again that $3 \leq \GK(A) < \infty$.  We saw that $\ell/(\ell-2) = 2$ is an eigenvalue of $M$
above.   Again, every eigenvalue $\delta_i$ of $M$ satisfies $\delta_i  \in  (1/2)(\sqrt[4]{\zeta_i} )\Omega_4$.  Since the maximum value of the
norm of an element in $\Omega_4$ is $4$, this proves that $\rho(M) = 2$.  None of the $f_i(t)$ can be equal to $(1-t)^4$, which has a nonvanishing $t^2$ term.  Thus the maximal order of vanishing of the $f_i(t)$ at $t = 1$ is $3$, so $\GKdim(A) = 3$.
In fact, $f_1(t)  = 1- \rho t + \rho t^3 - t^4 = 1 - 2t + 2t^3 - t^4 = (1-t)^3(1+t)$.  Part (2) follows.
\end{proof}

An important special case of the result above is the following.

\begin{corollary}
\label{cor:diagonalize}
Keep the hypotheses and notation of Proposition~\ref{prop:diagonalize}.  Suppose that $M$ is symmetric and that $P = I_n$.
Let $\delta_1, \dots, \delta_n$ be the eigenvalues of $M$, which are all real.  Assume that $\GKdim(A) \geq 3$.  If $\ell = 3$,
then $\GKdim(A) < \infty$ if and only if $\delta_i \in [-1, 3]$ for all $i$, and if $\ell = 4$, then $\GKdim(A) < \infty$ if and only if $\delta_i \in [-2, 2]$
for all $i$.
\end{corollary}
\begin{proof}
This follows from the intersections of $\Omega_3$ and $\Omega_4$ with $\mb{R}$ given in Proposition~\ref{prop:hypo}.
\end{proof}

We close with one final example.  Let $Q$ be a quiver with 2 vertices and three arrows of weight $1$ in each direction, so $Q$ has incidence matrix $M = \begin{psmallmatrix} 0 & 3 \\ 3 & 0 \end{psmallmatrix}$.  Then $M$ has eigenvalues $\{-3, 3\}$ and thus by Corollary~\ref{cor:diagonalize},
there is no locally finite graded Calabi-Yau $A$ of finite GK-dimension with semipotential of degree $3$ and with $A \cong kQ/I$.  This is despite $M$ having the correct spectral radius of $3$.

 %%%%%%%%%%%%%%%%%%%%%%%

\bibliographystyle{amsplain}
\bibliography{twistedcygk}

\end{document}